\documentclass[reqno, oneside]{amsart}

\PassOptionsToPackage{a4paper, total={6in, 8.5in}}{geometry}
\PassOptionsToPackage{toc,title,page}{appendix}
\PassOptionsToPackage{status=draft}{fixme}

\usepackage[dvipsnames]{xcolor}
\usepackage{tikz}
\usepackage{amsmath,bm,bbm,amsthm, amssymb}
\usepackage{mathtools}
\usepackage{color}
\usepackage{dsfont}
\usepackage{etoolbox}
\usepackage{comment}
\usepackage{pgf}
\usetikzlibrary{calc}
\usetikzlibrary{patterns}
\usetikzlibrary{arrows}
\usetikzlibrary{decorations.pathreplacing}
\usepackage[utf8]{inputenc}
\usepackage{pgfplots}

\usepackage[theorems,skins]{tcolorbox}
\usepackage{adjustbox}
\usepackage[final]{pdfpages}
\usepackage[ruled,vlined,linesnumbered,algosection,resetcount]{algorithm2e}
\usepackage{blkarray}
\usepackage{arydshln}
\usepackage{wrapfig}
\usepackage{enumitem}

\newcommand{\aG}{\operatorname{G:}}
\newcommand{\aH}{\operatorname{H:}}
\newcommand{\aF}{\operatorname{F:}}

\newcommand{\aT}{\operatorname{T:}}
\newcommand{\aU}{\operatorname{U:}}

\newcommand{\aM}{\operatorname{M:}}
\newcommand{\aP}{\operatorname{P:}}

\newcommand{\aTp}{\operatorname{T':}}
\newcommand{\aPp}{\operatorname{P':}}
\newcommand{\aFp}{\operatorname{F':}}

\newcommand{\vol}{\operatorname{vol}}

\usepackage{%
  appendix, etoolbox, setspace, accents, geometry, fancyhdr, hyperref, xpatch, graphicx, caption, subcaption, fixme%
}

\makeatletter
\xapptocmd{\@sect}{
  \csname #1mark\endcsname{#8} 
  \gdef\sectiontitle{#8} 
}{}{}
\makeatother

\makeatletter
\renewcommand\section{\@startsection {section}{1}{\z@}%
                                       {-5ex \@plus -1ex \@minus -.2ex}%
                                       {2.5ex \@plus.2ex}%
                                       {\center\scshape\Large}}

\renewcommand\subsection{\@startsection{subsection}{2}{\z@}%
                                         {-3.5ex\@plus -1ex \@minus -.2ex}%
                                         {1ex \@plus .2ex}%
                                         {\center\scshape\large}}
\makeatother

\makeatletter
\def\l@subsection{\@tocline{2}{0pt}{3pc}{6pc}{}}
\makeatother

\newcommand\scalemath[2]{\scalebox{#1}{\mbox{\ensuremath{\displaystyle #2}}}}

\usepackage[foot]{amsaddr}

\makeatletter
\renewenvironment{abstract}{%
  \ifx\maketitle\relax
    \ClassWarning{\@classname}{Abstract should precede
      \protect\maketitle\space in AMS document classes; reported}%
  \fi
  \global\setbox\abstractbox=\vtop \bgroup
    \normalfont\Small
    \list{}{\labelwidth\z@
      \leftmargin0.1\textwidth
      \rightmargin0.1\textwidth
      \listparindent\normalparindent
      \itemindent\z@
      \parsep\z@
      \@plus\p@
    }%
    \item[\hskip\labelsep\bfseries Abstract.]%
}{%
  \endlist\egroup
  \ifx\@setabstract\relax \@setabstracta \fi
}
\makeatother

\theoremstyle{plain}
\usepackage{amsthm}
\newtheorem{theorem}{Theorem}
\newtheorem{proposition}[theorem]{Proposition}
\newtheorem{corollary}[theorem]{Corollary}
\newtheorem{lemma}[theorem]{Lemma}

\theoremstyle{definition}

\newtheorem{example}[theorem]{Example}
\newtheorem{experiment}[theorem]{Experiment}

\newtheorem{remark}[theorem]{Remark}

\numberwithin{equation}{section}
\numberwithin{theorem}{section}

\theoremstyle{plain}

\newtheorem*{assumptionF*}{Assumption F}
\newtheorem{assumptionF'}{Assumption F'}
\newtheorem*{assumptionF'*}{Assumption F'}

\newtheorem*{assumptionS*}{Assumption S}
\newtheorem{assumptionS'}{Assumption S'}
\newtheorem*{assumptionS'*}{Assumption S'}

\newtheorem*{assumptionM*}{Assumption M}

\newtheorem*{assumptionT*}{Assumption T}
\newtheorem{assumptionT'}{Assumption T'}
\newtheorem*{assumptionT'*}{Assumption T'}

\newtheorem*{assumptionU*}{Assumption U}

\newtheorem*{assumptionX*}{Assumption X}

\newtheorem*{assumptionC*}{Assumption C}
\newtheorem{assumptionC'}{Assumption C'}
\newtheorem*{assumptionC'*}{Assumption C'}

\newtheorem*{assumptionR1*}{Assumption R1}

\newtheorem*{assumptionR2*}{Assumption R2}

\newtheorem*{assumptionP*}{Assumption P}
\newtheorem*{assumptionP'}{Assumption P'}
\newtheorem*{assumptionP'*}{Assumption P'}

\newtheorem*{assumptionI*}{Assumption I}
\newtheorem*{assumptionI'}{Assumption I'}
\newtheorem*{assumptionI'*}{Assumption I'}

\newtheorem*{assumptionW*}{Assumption W}
\newtheorem*{assumptionW'}{Assumption W'}
\newtheorem*{assumptionW'*}{Assumption W'}

\newtheorem*{assumptionR*}{Assumption R}
\newtheorem*{assumptionR'}{Assumption R'}
\newtheorem*{assumptionR'*}{Assumption R'}

\newtheorem*{assumptionG*}{Assumption G}

\newtheorem*{assumptionH*}{Assumption H}

\def\R{\mathbb R}

\def\E{\mathbb E}
\def\T{\mathbb T}
\def\eps{\varepsilon}

\def\xx{\mathbf x}
\def\yy{\mathbf y}

\def\00{\mathbf 0}

\def\bet{\begin{theorem}}
\def\ent{\end{theorem}}
\def\bec{\begin{corollary}}
\def\enc{\end{corollary}}
\def\bep{\begin{proof}}
\def\enp{\end{proof}}
\def\f{\frac}

\def\a{\alpha}

\def\es{\emptyset}
\def\su{\subseteq}
\def\ms{\mathsf}

\def\mc{ \mathcal}

\def\PP{\mc P}
\def\dtv{d_{\ms{TV}}}
\def\dkr{d_{\ms{KR}}}

\def\one{\mathds1}
\def\de{\delta}
\def\d{\mathrm d}
\def\ze{\zeta}
\renewcommand\leq{\leqslant}
\def\r{\rho}
\renewcommand\geq{\geqslant}
\renewcommand\le{\leqslant}
\renewcommand\ge{\geqslant}

\DeclarePairedDelimiter{\abs}{\lvert}{\rvert}

\def\e{\varepsilon}
\def\bel{\begin{lemma}}
\def\enl{\end{lemma}}
\def\im{\item}
\def\been{\begin{enumerate}}
\def\enen{\end{enumerate}}

\def\k_d{\kappa_d}

\def\lmax{\ell_{\ms{max}}}

\def\k{\kappa}

\def\bepr{\begin{proposition}}
\def\enpr{\end{proposition}}

\usepackage{etoolbox}

\newlength{\dhatheight}

\begin{document}






\title{POISSON APPROXIMATION OF LARGE-LIFETIME CYCLES}

\author{Christian Hirsch $^1$}
\address{$^1$ Department of Mathematics, Aarhus University, Denmark, E-mail: \texttt{hirsch@math.au.dk}}

\author{Nikolaj Nyvold Lundbye $^2$}
\address{$^2$ Department of Mathematics, Aarhus University, Denmark, E-mail: \texttt{nnl@math.au.dk}}

\author{Moritz Otto $^3$}
\address{$^3$ Mathematical Institute, Leiden University, Netherlands, E-mail: \texttt{m.f.p.otto@math.leidenuniv.nl}}

\maketitle
\thispagestyle{empty} 
\vspace{-3ex}

\begin{abstract}
	In topological data analysis, the notions of persistent homology, birthtime, lifetime, and deathtime are used to assign and capture relevant 
	cycles (i.e., topological features) of a point cloud, such as loops and cavities. In particular, cycles with a large lifetime are of special interest. 
	In this paper, we study such large-lifetime cycles when the point cloud is modeled as a Poisson point process. 
	First, we consider the case with no bound on the deathtime, where we establish Poisson convergence of the centers of large-lifetime features on 
	the 2-dimensional flat torus. Afterwards, by imposing a bound on the deathtime, we enter a sparse connectivity regime, 
	and we prove joint Poisson convergence of the centers, lifetimes, and deathtimes in dimensions $d \geq 2$
	under suitable model conditions. \newline \newline
\textbf{Keywords: } \v Cech- \& Vietoris-Rips complexes, persistent homology, persistence diagram, largest lifetime,
Poisson point processes, Poisson U-statistics, Poisson approximation. \newline \newline
\textbf{MSC Classification:} 55U10, 60D05, 60G55, 60F05, 60G70.
\end{abstract}

\begingroup
\small 
\setlength{\parskip}{2pt}
\setlength{\parindent}{0pt}
\tableofcontents
\endgroup






\begin{spacing}{1.25}

\section*{Introduction}
Topological data analysis (TDA) has recently gained popularity across a wide variety of disciplines because it provides insights into datasets by 
relying on invariants from algebraic topology \cite{wasserman}. One of the key tools in TDA is \textit{persistent homology}, which tracks the scales at which topological 
features such as components, loops, or cavities appear and disappear. In this way, each feature is associated with a \textit{lifetime}. The main idea is that features 
with long lifetimes indicate interesting topological structures, while features with short lifetimes are likely statistical noise.
This approach has been applied in fields such as astronomy, chemistry, materials science and neuroscience \cite{gidea,neuro_hess,tda_nature,rien}.

Despite TDA’s widespread adoption, its statistical foundations are still in their infancy. This presents a serious challenge for applications, 
particularly in medicine, where it is crucial to devise hypothesis tests to determine whether a given dataset is consistent with a null model. 
Currently, the vast majority of TDA-based hypothesis tests rely on Monte Carlo or bootstrap methods. While these methods offer flexibility 
with respect to the considered model, their results do not generalize beyond a specific sampling window. Additionally, the high computational
costs create substantial obstacles for large datasets. This has recently motivated the development of tests based on the asymptotic normality
of persistence-related functionals over large domains.
For example, recent works \cite{svane,cyl} have examined the asymptotic normality of test statistics based on feature lifetimes averaged over entire domains. 
However, deviations from the null model can often be more clearly identified by considering the maximum lifetime of all features in the sampling window. 
For instance, there are currently no statistically rigorous limit results for the immensely popular persistence landscape in large windows \cite{bubenik}.
This has led to growing interest in the analysis of extreme events in TDA, a fruitful area of research at the intersection of extreme value theory and 
algebraic topology \cite{CH20,BKS17,ecc2,ecc3}. 

The first breakthrough in the probabilistic analysis of extreme lifetimes occurred in \cite{BKS17} when the high-probability order of the largest lifetime was determined. 
While \cite{BKS17} focused on multiplicative lifetimes, in \cite{CH20} maximal lifetimes were investigated in an additive sense, 
establishing Poisson approximation results. However, these findings on maximal lifetimes are restricted to codimension 1,
meaning that Poisson approximations are not yet available for applications requiring the analysis of loops in 3D space. More recently, 
\cite{univ} developed a general framework from which a precise and broadly 
applicable conjecture on the asymptotic distribution of large lifetimes has been derived. 
Although these conjectures have not yet been
fully established at a rigorous level, \cite{BS24} managed to
obtain a universal law of large numbers for multiplicative lifetimes in arbitrary dimensions.

In the present work, we take the first steps toward establishing a theory of Poisson approximation for 
large-lifetime features in arbitrary dimensions. 
First, we consider an \textit{unbounded regime}, which refers to having no bound on the \textit{deathtime} of features. 
Here, we focus solely on multiplicative lifetimes for a stationary Poisson point process, 
and our findings can be seen as a continuation of the work in \cite{BKS17}. Our work goes beyond \cite{BKS17} in the following way:
\been 
\im By considering only dimension $d=2$, we are able to achieve Poisson convergence of the spatial locations of large-lifetime features 
and prove that the scaling of the threshold for what constitutes a large lifetime grows asymptotically as the 
expected largest lifetime in \cite{BKS17}.
\enen

A strategy that has been successfully implemented in other contexts is the consideration of the \emph{sparse regime} \cite{KM18},
 where the expected number of $r_n$-distance neighbors of a typical point tends to 0 as $n \to \infty$. To reflect the sparsity assumption,
  we focus our investigations on large-lifetime features whose deathtime is at most $r_n>0$. Henceforth, we refer to the \textit{$m$-sparse regime}, 
  as where, with high probability, connected clusters are of size at most $m \geq 1$.
  This key focus on the $m$-sparse regime allows us to go substantially beyond the findings in the unbounded regime and from \cite{CH20}:
\begin{enumerate}[resume]
    \item We prove Poisson convergence of the location of large-lifetime features in arbitrary dimensions $d \geq 2$, 
	and we are no longer restricted to features of codimension 1. Instead, we are restricted to dimensions $d$, 
	feature sizes $k$, and cluster sizes $m$, where each cluster of size $m$ has at most one $k$-feature with lifetime close to the maximal possible lifetime $\lmax$.
	We verify this uniqueness for various combinations, including, for example, loops in 3D space.
    \item Our framework is sufficiently flexible to accommodate inhomogeneities in the initial Poisson point cloud over 'nice' regions. 
	This includes cases where intensities are equipped with a bounded, continuously differentiable Lebesgue density on a convex set of finite volume.
    \item Besides proving Poisson convergence of the location of large-lifetime features, we also show that by adding the deviation of the lifetime to $\lmax$ and 
	deathtime as marks,
	 Poisson convergence still holds under certain regularity assumptions. These assumptions are verified for $d = 2$. In particular, we investigate what 
	 this additional information reveals about the asymptotic distributional behavior of the \textit{persistence diagram} and the \textit{largest lifetime}.
\end{enumerate}

The key probabilistic tool in our proofs is the Malliavin-Stein method from \cite{decreusefond,BSY21},
which not only proves Poisson convergence, but also quantifies its speed. In the $m$-sparse regime, a major challenge is that \cite{decreusefond} only
applies to U-statistics and the large-lifetime features are not of this form.
We address this by approximating large-lifetime features with a U-statistic 
and control the error of this approximation.
In the unbounded regime, such an approximation is not possible and therefore we use more involved methods from \cite{BSY21}.
The most difficult step, excluding clustering of multiple large-lifetime features, is resolved using the
recently established continuous BK inequality \cite{HHLM19}. 

The rest of the paper is organized as follows. 
In Section \ref{sec:1}, we recall key concepts from topology and stochastic geometry and explain the basic principles of TDA.
In Section \ref{sec:2}, we describe the unbounded and $m$-sparse regimes in detail and present the main theorems. 
Sections \ref{sec:3} and \ref{sec:4} prove Poisson convergence in the unbounded and $m$-sparse regimes, respectively. 
Section \ref{sec:5} discusses the model assumptions we use.
Section \ref{sec:6} demonstrates applications of the main theorems, focusing on certain choices of the intensity, the largest lifetime and the persistence diagram. 
Finally, the value of $\lmax$ and the persistence diagram are explored through simulations in Section \ref{sec:7}.
\vspace{-1ex}






\section{Preliminaries}
\label{sec:1}

First of all, this section contains a brief review of the \textit{$d$-dimensional flat torus}, \textit{point processes}
and in particular the \textit{Poisson point process}. Secondly, we review relevant notions from algebraic topology
like \textit{\v Cech} and \textit{Vietoris-Rips complexes}, \textit{persistent homology}
and the concepts of \textit{birthtimes}, \textit{deathtimes} and \textit{lifetimes}.
We refer to \cite{chazal} for comprehensive study of these topics.


\vspace{2.5ex}
\begin{center}
	\large{\textsc{\v Cech \& Vietoris-Rips complexes}}	
\end{center}
\vspace{1ex}
Let $(X,\rho)$ denote a metric space and let $\PP \su X$ be finite.
A \textit{simplicial complex} on $\PP$
is a family of subsets of $\PP$ closed under taking subsets.
One can think of a
simplicial complex as a generalization of a graph, where, e.g., the 0-simplices are vertices, the 1-simplices edges and the 2-simplices solid triangles. 
For radius $t > 0$, the \textit{\v Cech complex} $\mathsf{C}_t(\PP)$ and \textit{Vietoris-Rips complex} $ \mathsf{VR}_t(\PP)$
satisfy that
	 \vspace{0.75ex} $$
	 \begin{aligned}
		\{x_0,\ldots,x_p\} \in  \mathsf{C}_t(\PP)  & \iff \text{ there exists } x \in X
	  \text{ such that } \rho(x,x_j) \leq t \text{ for all } 0 \leq j \leq p. \\
	  \{x_0,\ldots,x_p\}\in  \mathsf{VR}_t(\PP)  & \iff  \rho(x_i,x_j) \leq 2t \text{ for all } 0 \leq i, j \leq p.
		\end{aligned}
	 \vspace{0.75ex} $$
In other words, a $p$-simplex is included in the \v Cech complex at radius $t$ if there is a ball of this radius which contains all $p+1$ points,
 whereas a $p$-simplex is included in the Vietoris-Rips complex at radius $t$
when the pairwise distance of all $p+1$ points is less than twice this radius (see Figure \ref{fig:1}). 

\vspace{1ex}
\begin{figure}[htbp]
    \centering
	\makebox[0.9\textwidth]{
    \begin{subfigure}[t]{0.4 \textwidth}
        \centering
        \begin{tikzpicture}[scale=0.75]
			\begin{scope}[shift={(0,0)}]
				\fill[blue!40!black] (0,0) -- (2,0) -- (1,1.73) -- cycle; 
				\node[draw, circle, fill=black, minimum size=4pt, inner sep=0pt] (A) at (0,0) {};
				\node[draw, circle, fill=black, minimum size=4pt, inner sep=0pt] (B) at (2,0) {};
				\node[draw, circle, fill=black, minimum size=4pt, inner sep=0pt] (C) at (1,1.73) {};
				\draw[black] (A) -- (B) -- (C) -- (A);
			\end{scope}
			
			\node[draw, circle, fill=black, minimum size=4pt, inner sep=0pt] (D) at (3.25,2.25) {};
			\draw[black] (B) -- (C) -- (D) -- (B);
		
			\begin{scope}[shift={(3,-0.5)}]
				\node[draw, circle, fill=black, minimum size=4pt, inner sep=0pt] (1) at (1,0) {};
				\node[draw, circle, fill=black, minimum size=4pt, inner sep=0pt] (2) at (2.25,0.5) {};
				\node[draw, circle, fill=black, minimum size=4pt, inner sep=0pt] (3) at (3,0) {};
				\node[draw, circle, fill=black, minimum size=4pt, inner sep=0pt] (4) at (4,1) {};
				\draw[black] (1) -- (2) -- (3) -- (4);
			\end{scope}
			
			\node[draw, circle, fill=black, minimum size=4pt, inner sep=0pt] at (5.75,2) {};
		\end{tikzpicture}
		\caption{\text{Simplicial complex with nine 0-simplices}, eight 1-simplices and one
		2-simplex (blue).
		}
    \end{subfigure}%
    \hfill
    \begin{subfigure}[t]{0.4 \textwidth}
        \centering
		\begin{tikzpicture}[scale=0.65]

			\begin{scope}[shift={(-2.5,0)}] 
			
			
			\coordinate (A1) at (0,0);
			\coordinate (B1) at (2,0);
			\coordinate (C1) at (1,{sqrt(3)});
			
			\filldraw[black] (A1) circle (2pt);
			\filldraw[black] (B1) circle (2pt);
			\filldraw[black] (C1) circle (2pt);
			
			\draw (A1) -- (B1);
			\draw (A1) -- (C1);
			\draw (B1) -- (C1);
			
			
			\draw[black, dashed] (A1) circle (1.02);
			\draw[black, dashed] (B1) circle (1.02);
			\draw[black, dashed] (C1) circle (1.02);
			
			
			\end{scope}
			
			\begin{scope}[shift={(2.5,0)}] 
			
			
			\coordinate (A2) at (0,0);
			\coordinate (B2) at (2,0);
			\coordinate (C2) at (1,{sqrt(3)});
			
			\filldraw[black] (A2) circle (2pt);
			\filldraw[black] (B2) circle (2pt);
			\filldraw[black] (C2) circle (2pt);
			
			\draw (A2) -- (B2);
			\draw (A2) -- (C2);
			\draw (B2) -- (C2);
			
			\fill[blue!40!black, opacity=0.9] (A2) -- (B2) -- (C2) -- cycle;
			
			\draw[black, dashed] (A2) circle (1.02);
			\draw[black, dashed] (B2) circle (1.02);
			\draw[black, dashed] (C2) circle (1.02);
			
			
			\end{scope}
			
			\end{tikzpicture}
        \caption{\text{Left: Čech complex at radius $t>0$.} \text{Right: Vietoris–Rips complex
		 at the same $t$.}}
    \end{subfigure}
	} 
	\caption{Illustration of simplices and \v Cech 
	versus Vietoris-Rips complexes.}
	\label{fig:1}
\end{figure}
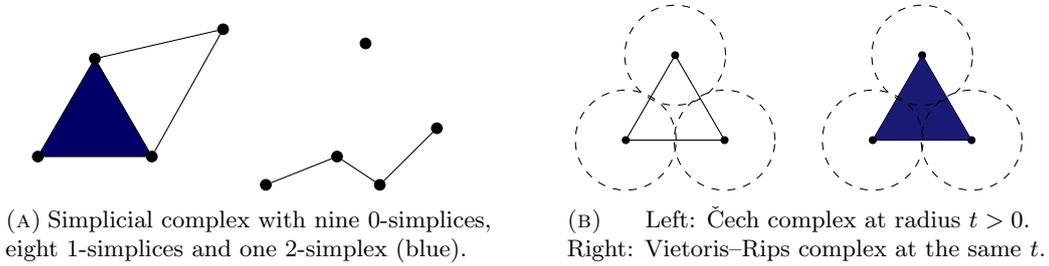
 

\vspace{0.75ex}
\begin{center}
  \large{\textsc{The flat torus}}
\end{center}
\vspace{1ex}

Henceforth, we will replace $X$ by a region $W \subseteq \R^d$, where $d \geq 2$. The canonical choice of metric $\rho$ on $W$ is  
the Euclidean distance, which is the metric induced by the Euclidean norm $\Vert \cdot \Vert$. However, we can also define the flat torus $\T^d$ 
as the unit cube $[0,1]^d$ equipped with the \textit{toroidal metric} \vspace{0.3ex} 
\begin{equation*}
	\rho(x,y) = \min_{q \in \mathbb{Z}^d} \Vert x-y+q \Vert, \quad \text{for any } x,y \in [0,1]^d. \vspace{0.3ex} 
\end{equation*} Equivalently, one can think of the flat torus as the quotient space $\R^d / \mathbb{Z}^d$, 
which induces a Lebesgue measure $\vol_{\T^d}(\cdot)$ on the torus from the image measure of
the regular Lebesgue measure $\vol_{\R^d}(\cdot)$ under the quotient map $q:\R^d \to \R^d /\mathbb{Z}^d$, i.e., 
$
	\vol_{\T^d}(A) = \vol_{\R^d}(q^{-1}(A) \cap [0,1]^d), 
$
which becomes a translation invariant probability measure. 
The main advantage of the flat torus is that it is compact, yet has no boundary, which 
removes edge effects when studying point processes.


\vspace{2.5ex}
\begin{center}
	\large{\textsc{Point processes and Kantorovich-Rubinstein distance}}
\end{center}
\vspace{1ex}
Let $\mathbb{F}^d$ denote $\R^d$ or $\T^d$, and $\mathcal{B}(\mathbb{F}^d)$ the Borel $\sigma$-algebra on $\mathbb{F}^d$ equipped with two 
measures $\mu$, $\nu$. Then, we define the \textit{total variation distance} as
\vspace{0.75ex} $$
   \dtv(\mu,\nu) = \sup\big\{\abs{\mu(A)-\nu(A)}\colon A \in \mathcal{B}(\mathbb{F}^d) \text{ and }  \mu(A),\nu(A) < \infty \big\}.
\vspace{0.75ex} $$
Let $(\Omega, \mc F, P)$ denote a probability space large enough to
contain all random objects in the present work and
let $\mathfrak{N}(\mathbb{F}^d)$ denote the space of locally finite (i.e., finite on compact sets) counting measures on $\mathcal{B}(\mathbb{F}^d)$, and equip $\mathfrak{N}(\mathbb{F}^d)$
with the $\sigma$-algebra $\mc G$ generated by the maps $\mu \mapsto \mu(A)$ for $A \in \mathbb{F}^d$. A (locally finite) \textit{point process} 
$\PP:\Omega \to \mathfrak{N}(\mathbb{F}^d)$ is a $\mc F$-$\mc G$-measurable map. For another point process $\mc Q$, let $\Sigma(\PP,\mc Q)$ denote the set of
couplings of the probability measures $P(\PP^{-1}(\cdot))$ and $P(\mc Q^{-1}(\cdot))$. Finally, we define the \textit{Kantorovich-Rubinstein distance} of the point processes $\PP$ and $\mc Q$ as
\vspace{0.75ex} $$
   \dkr(\PP,\mc Q) = \inf_{\mc C \in \Sigma(\PP,\mc Q)} \int_{\mathfrak{N}(\mathbb{F}^d) \times \mathfrak{N}(\mathbb{F}^d)} \dtv(\mu,\nu)\  \mc C(\d \mu, \d \nu).
\vspace{0.75ex} $$
By \cite[Proposition 2.1]{decreusefond}, then $\dkr(\PP_n,\PP) \to 0$ implies that $\PP_n$ converges in distribution to $\PP$.

\vspace{2.5ex}
\begin{center}
	\large{\textsc{The Poisson Point Process}}
\end{center}
\vspace{1ex}

By a \textit{Poisson point process}, abbreviated \textit{P.p.p.}, with locally finite atom-free \textit{intensity measure} $\lambda$, 
we refer to a point process $\PP$ on $W \su \mathbb{F}^d$ such that (i) $\PP(A) \sim \text{Pois}(\lambda(A))$, 
and (ii) $\PP(A)$ and $\PP(B)$ are independent for any disjoint Borel sets $A,B$.
 An equivalent construction of the Poisson point process is the following: 
 Assuming that $W$ is compact, then by taking a sequence of random points $x_1,x_2,\ldots \sim \lambda/\lambda(W)$ independent
 and $N \sim \text{Pois}(\lambda(W))$, then
$
\vert \{x_1, \ldots, x_N \}  \cap A \vert \sim \PP(A)
$
for any Borel set $A$. Thus, we can also think of $\PP$ as a collection of random points. For non-compact $W$, 
we can discretize into bounded regions and repeat the above construction in each region and take the union over all regions. 
When $\lambda$ is translation invariant (on $\mathbb{F}^d$), we say that $\PP$ is \textit{stationary}.
 If $\lambda$ has a density $f$ with respect to the Lebesgue measure, we use the term \textit{intensity function} 
 to refer to $f$. Finally, a crucial property of the Poisson point process is \textit{Mecke's formula} 
 \cite[Theorem 3.2.5]{schneider}: For any non-negative, measurable function $g$, \vspace{0.5ex} \begin{equation*} \label{eq
	} \E\bigg[\sum_{\xx \in \PP^{(k)}} g(\xx, \PP) \bigg] = \int_{W^k} g(\xx,\PP \cup \{\xx\}) \ \lambda^k(\d \xx), \vspace{0.5ex} \end{equation*}
 where $\PP^{(k)}$ denotes all $k$-tuples of different points from $\PP$ and $\lambda^k$ denotes the $k$-fold product measure. 
 The above sum is often called a \textit{Poisson functional}. When $g$ is symmetric in $\xx$ and does not depend on the entire process $\PP$,
  we get a special type of Poisson functional called a \textit{Poisson $U$-statistic}.


\vspace{2.5ex}
\begin{center}
	\large{\textsc{Cycles, birthtimes, deathtimes and lifetimes}}
\end{center}
\vspace{1ex}
For any simplicial complex, we define the \textit{space of $p$-chains} as 
  the linear span of $p$-simplices over the field $\mathbb{Z}/2\mathbb{Z}$, and the \textit{$p$th homology group}
   as the quotient space of the $p$-chains whose boundary is 0 (which is understood in terms of a boundary operator),
    where two $p$-chains are equivalent if their difference is a boundary of a $(p+1)$-chain. We call elements of 
	the $p$th homology group \textit{$p$-dimensional cycles} or simply \textit{$p$-cycles}. 
	For example, 0-cycles are the connected components, 1-cycles are the 'loops,' and 2-cycles are the 'cavities'.
	In topological data analysis, to recover properties of the point process which generated $\PP$, we build a \textit{filtration}
	which is an increasing sequence of simplicial complexes (e.g. in \v Cech or Vietoris-Rips by using all radii at once).
	The analysis of $\PP$ is then based on \textit{persistent homology}, which tracks how each homology group changes throughout the filtration. 
When new simplices are added in the filtration, they are called \textit{critical} if they cause one of the homology groups to change. 
A critical simplex that destroys a cycle, i.e. reduces the dimension of one of the homology groups by 1, is called \textit{negative}
and we let $\mc N(\PP)$ denote all negative simplices. We rely on the identification found in \cite[Algorithm 9]{chazal} to identify
 every non-trivial $p$-cycle with a negative $(p+1)$-simplex for $p \geq 1$.
For simplices $\xx$, introduce the notation/terminology:
\vspace{0.25ex} \begin{equation*}
	\begin{aligned}
		z_\xx & \quad \quad \quad \text{ The center of the smallest ball containing $\xx$, } \\
		r_\xx & \quad \quad \quad  \text{ The filtration time where $\xx$ is added,} \\
		b_{\xx,\mc P} & \quad \quad \quad \text{ The time when the cycle destroyed by $\xx$ was added (if $\xx \in \mc N(\PP)$),} \\
		\ell^{+}_{\xx,\PP} & \quad \quad \quad \text{ The \textit{additive lifetime} of the cycle destroyed by $\xx$ computed as $r_\xx - b_{\xx,\PP}$}. \\
		\ell^{*}_{\xx,\PP}  & \quad \quad \quad \text{ The \textit{multiplicative lifetime} of the cycle destroyed by $\xx$ computed as $r_\xx/b_{\xx,\PP}$}. \\
	\end{aligned}
	\vspace{0.75ex} \end{equation*}

	We will call $b_{\xx,\PP}$ the \textit{birthtime} and $r_\xx$ the \textit{deathtime} of the cycle destroyed by adding $\xx \in \mc N(\PP)$. 
	Note for the \v Cech filtration $(\mathsf{C}_t(\PP))_{t\geq 0}$, the deathtime is the radius of the circumscribed ball of $\xx$. In both the \v Cech and Vietoris-Rips filtrations,
	 the deathtime depends only on the points in the simplex, while the birthtime may depend on all points in the point cloud. 
	 Note that only the additive lifetime is bounded by the deathtime, 
	 and only the multiplicative lifetime remains unchanged when scaling the point cloud. 
	Note that some extra care is needed in order to study negative simplices on the torus rather than in Euclidean space.
	 This is treated in detail for \v Cech complexes in Section \ref{sec:3}.

	\begin{example}["The fish"]

Let $\PP \su \R^2$ be the point cloud consisting of the six points 
$x_1 = (1,\tfrac{\sqrt{3}}{2})$, $x_2 = (-1,-\tfrac{\sqrt{3}}{2})$, $x_3 = (0,0)$, $x_4 = (1,1)$, $x_5 = (1,-1)$, $x_6 = (2,0)$. 
Choosing first the \v Cech filtration up to filtration time of $t=1$, it follows that the edges $\{x_3,x_4\},\{x_3,x_5\}$,
$\{x_4,x_6\},\{x_5,x_6\}$ 
are added at filtration time $t = \tfrac{1}{\sqrt 2}$, 
the edges $\{x_1,x_2\},\{x_2,x_3\},\{x_1,x_3\}$ at $t=\tfrac{\sqrt{3}}{2}$ and the 2-simplices $\{x_1,x_2,x_3\}$
$\{x_3,x_4,x_5\}$, $\{x_4,x_5,x_6\}$ at $t=1$. Thus, two 1-cycles appear and disappear
  during the filtration, which we identify with the negative 2-simplices $\xx = \{x_1, x_2, x_3\}$ and $\xx' = \{x_3, x_4, x_5\}$ 
  Then $z_\xx = (-1, 0)$, $b_{\xx, \PP} = \tfrac{\sqrt{3}}{2}$, $r_\xx = 1$, $\ell^+_{\xx, \PP} = 1 - \tfrac{\sqrt{3}}{2}$, 
  $\ell^*_{\xx, \PP} = \tfrac{2}{\sqrt{3}}$, and $z_{\xx'} = (1, 0)$, $b_{\xx', \PP} = \tfrac{1}{\sqrt 2}$, $r_{\xx'} = 1$,
   $\ell^+_{\xx', \PP} = 1 - \tfrac{1}{\sqrt 2}$, $\ell^*_{\xx', \PP} = \sqrt{2}$. If we instead choose the Vietoris-Rips filtration, 
   then nothing changes with the edges. However, the 2-simplex $\{x_1, x_2, x_3\}$ is also added at $t=\tfrac{\sqrt{3}}{2}$, 
   which removes one of the 1-cycles from before (see Figure 2).
\end{example}

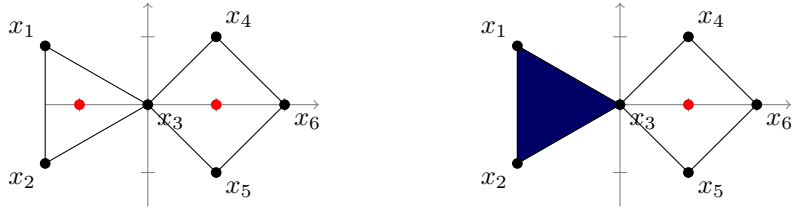
\begin{figure}[h]
    \centering
    \begin{subfigure}{0.4\textwidth}
        \centering
        \begin{tikzpicture}[scale=0.9]
            \draw[->] (-1.5, 0) -- (2.5, 0) [gray] ;
            \draw[->] (0, -1.5) -- (0, 1.5) [gray] ;
        
            \foreach \x in {-1, 0, 1, 2} {
                \draw (\x, 0.1) -- (\x, -0.1) [gray];
            }
        
            \foreach \y in {-1, 0, 1} {
                \draw (0.1, \y) -- (-0.1, \y) [gray];
            }
        
             \filldraw[red] (-1, 0) circle (2pt);
             \filldraw[red] (1, 0) circle (2pt);
            \filldraw[black] (-1.5, {sqrt(3)/2}) circle (2pt) node[above left] {$x_1$};
            \filldraw[black] (-1.5, -{sqrt(3)/2}) circle (2pt) node[below left] {$x_2$};
            \filldraw[black] (0, 0) circle (2pt) node[below right] {$x_3$};
            \filldraw[black] (1, 1) circle (2pt) node[above right] {$x_4$};
            \filldraw[black] (1, -1) circle (2pt) node[below right] {$x_5$};
            \filldraw[black] (2, 0) circle (2pt) node[below right] {$x_6$};
        
            \draw (0, 0) -- (1, 1); 
            \draw (0, 0) -- (1, -1); 
            \draw (1, 1) -- (2, 0); 
            \draw (1, -1) -- (2, 0); 
            \draw (-1.5, {sqrt(3)/2}) -- (-1.5, -{sqrt(3)/2}); 
            \draw (-1.5, -{sqrt(3)/2}) -- (0, 0); 
            \draw (-1.5, {sqrt(3)/2}) -- (0, 0); 
        \end{tikzpicture}
    \end{subfigure}
    \begin{subfigure}{0.4\textwidth}
        \centering
        \begin{tikzpicture}[scale=0.9]
            \draw[->, gray] (-1.5, 0) -- (2.5, 0);
            \draw[->, gray] (0, -1.5) -- (0, 1.5);

            \foreach \x in {-1, 0, 1, 2} {
                \draw (\x, 0.1) -- (\x, -0.1) [gray];
            }
        
            \foreach \y in {-1, 0, 1} {
                \draw (0.1, \y) -- (-0.1, \y) [gray];
            }

            \filldraw[red] (1, 0) circle (2pt);
            \filldraw[blue!40!black] (-1.5, {sqrt(3)/2}) -- (-1.5, -{sqrt(3)/2}) -- (0, 0) -- cycle;
            \filldraw[black] (-1.5, {sqrt(3)/2}) circle (2pt) node[above left] {$x_1$};
            \filldraw[black] (-1.5, -{sqrt(3)/2}) circle (2pt) node[below left] {$x_2$};
            \filldraw[black] (0, 0) circle (2pt) node[below right] {$x_3$};
            \filldraw[black] (1, 1) circle (2pt) node[above right] {$x_4$};
            \filldraw[black] (1, -1) circle (2pt) node[below right] {$x_5$};
            \filldraw[black] (2, 0) circle (2pt) node[below right] {$x_6$};

            \draw (0, 0) -- (1, 1); 
            \draw (0, 0) -- (1, -1); 
            \draw (1, 1) -- (2, 0); 
            \draw (1, -1) -- (2, 0); 
            \draw (-1.5, {sqrt(3)/2}) -- (-1.5, -{sqrt(3)/2}); 
            \draw (-1.5, -{sqrt(3)/2}) -- (0, 0); 
            \draw (-1.5, {sqrt(3)/2}) -- (0, 0); 
        \end{tikzpicture}
    \end{subfigure}
    \caption{Example 1.1 with the \v Cech (left) and Vietoris-Rips (right)
    complex at time $t=\tfrac{\sqrt{3}}{2}$. The red points mark the centers $z_\xx$ and $z_{\xx'}$ of the negative 2-simplices.}
\end{figure}

\pagebreak






\section{Main results}
\label{sec:2}

Let $\mathbb{F}^d$ denote either $\R^d$ or $\T^d$ and $\PP \su \mathbb{F}^d$ denote a
point cloud, where we build the \v Cech or Vietoris-Rips
filtration. We are now ready to define the main object of interest: For deathtime bound $r \in (0, \infty]$, transformation 
$\varphi(\cdot,\PP):\PP^{(k)} \to \mathbb{F}^{d} \times [0,\infty)^p$ and threshold $\ell>0$, let
\vspace{0.75ex} \begin{equation}
\label{eq:xi_def}
	\xi = \xi(k, r, \ell, \varphi, \PP) =\sum_{\xx \in \PP^{(k)}} 
\one\big\{\xx \in \mc N(\PP), \ r_\xx \leq r, \ \ell_{\xx,\PP}^\circ \ge \ell \big\} \de\{\varphi(\xx,\PP)\}. \\[-0.5ex]
\vspace{0.75ex} \end{equation}
Here the set $\PP^{(k)}$ denotes all $k$-tuples of pairwise different points from $\PP$, the set $\mc N(\PP)$ all negative simplices,
$\ell_{\xx,\PP}^\circ$ the additive or multiplicative lifetime, and $A \mapsto \de\{y\}(A)$ the Dirac measure of $y \in \mathbb{F}^{d} \times [0,\infty)^p$.
For instance, when $\varphi(\xx,\PP)$ is the center $z_\xx$ of $\xx$, then (\ref{eq:xi_def}) on a set $A \subseteq \mathbb{F}^d$
counts how many $(k-2)$-cycles which are \textit{centered} in $A$ have
a deathtime \textit{smaller} than $r$ and a lifetime \textit{larger} 
than $\ell$.
When $\varphi(\xx,\PP)$ is $(z_\xx, \ell_{\xx,\PP}, r_\xx)$, then (\ref{eq:xi_def}) on
a set $A \times B \times C$ counts the number of cycles, as previously described, that are centered in $A$, have a lifetime in $B$, and a deathtime in $C$.
Letting $\PP_n$ denote a Poisson point process on $\mathbb{F}^d$
with intensity increasing in $n \geq 1$, we will study the asymptotic behavior of (\ref{eq:xi_def})
as $n \to \infty$. First, we will consider
an \textit{unbounded regime} on $\T^2$, where there is no bound on the deathtime (i.e., $r = \infty$).
Afterwards, we will consider an \textit{$m$-sparse regime} on $\R^d$, where we require the deathtime bound
$r = r_n$ to vanish sufficiently fast as $n \to \infty$. 


\begin{center}
	\subsection{The unbounded regime}
\end{center}

Suppose there is no bound on the deathtime ($r=\infty$). In this section, we will only work with a stationary Poisson point process
denoted $\PP_n$ with intensity $n \geq 1$ on $\T^2$ and simplices of size $k=3$. To specify the thresholds $\ell_{n,\alpha}$ in this regime, 
we will need the function $v$ defined as 
\vspace{0.75ex} \begin{equation}
	\label{eq:v_def}
v(\ell) = \int_{\T^{2 \cdot 3}}
P(\xx \in \mc N(\PP_n \cup \{\xx\} ), \ \ell_{\xx,\PP_{n} \cup \{\xx \} }^* \geq \ell) \ \d\xx.
\vspace{0.75ex} \end{equation}
Let $\dkr$
	denote the Kantorovich-Rubinstein metric (Section \ref{sec:1}), and for a sequence $(a_n)_{n \geq 1}$, let
	\vspace{0.75ex} \begin{equation*}
	O(a_n) = \{(b_n)_{n \geq 1} \colon \text{ there exist a } C > 0 \text{ such that } b_n \leq C a_n \text{ for all sufficiently large } n \geq 1\}.
	\vspace{0.75ex} \end{equation*}
	\begin{theorem}[\textbf{"Centers $\rightarrow$ Poisson"}]
		\label{thm:4}
		Let $\alpha > 0$ and define
		\vspace{0.5ex} \begin{equation}
			\label{eq:xi_def_unbounded}
				\xi_n^1=\sum_{\xx \in \PP_{n}^{(3)}}
			 \one\big\{\xx \in \mc N(\PP_n), \ \ell_{\xx,\PP_{n}}^* \ge \ell_{n,\alpha} \}
			  \de\{z_\xx\}, \\[-0.5ex]
		\vspace{0.55ex} \end{equation}
		where we assume that
		\begin{itemize}[leftmargin=5mm,labelindent=50pt]
			\item[] $\aFp$ The filtration consists of \v Cech complexes, 
			\item[] $\aPp$ The points $\PP_{n}$ is a stationary Poisson point process on $\T^2$ with intensity $n \geq 1$.
			\item[] $\aTp$ The thresholds $\ell_{n,\alpha}$ are equal to $v^{-1}(\alpha/n^3)$
			for all sufficiently large $n \geq 1$.
		\end{itemize}
		Let $\zeta^{1}$ denote a stationary Poisson point process on $\T^2$ with intensity $\alpha$. For every $\e > 0$, then
				\vspace{0.75ex} $$
				\dkr(\xi_n^1, \zeta^{1})\in O\big(n^{-1/36 + \e}\big). \\[0.1em]
				\vspace{0.75ex} $$	
		\end{theorem}

		\begin{remark}
			\label{rem:unbounded_extension}
			(\textit{Concerning Theorem \ref{thm:4}})
			\setlength{\topsep}{0pt}
			\begin{enumerate}[topsep=0pt, partopsep=0pt, parsep=0pt, itemsep=0pt]
				 \item \textit{Proof:} In very broad strokes, the proof of Theorem \ref{thm:4} consists of reducing to \textit{mortal} cycles 
				 that satisfy a reasonable deathtime bound, which implies a subcritical birthtime bound in the continuum percolation sense. 
				 As a consequence, we may restrict ourselves to cycles that stabilize to a certain ball around their center 
				 and apply \cite[Theorem 4.1]{BSY21} for stabilizing Poisson functionals, which provides several error terms that need to be controlled. 
				 The most interesting of these, the probability of multiple cycles (exceedances), is controlled by showing that 
				 multiple cycles imply the disjoint occurrence of one of the cycles and a large connected component in the subcritical regime, 
				 and then applying the continuous BK-inequality.
				 \item \textit{Intensity:} Note, assumption $T'$ ensures that the expected number of large-lifetime features on the torus remains constant as $n \to \infty$. Actually,
			 	for all sufficiently large $n \geq 1$, then
				 \vspace{0.5ex} \begin{equation}
					\label{eq:alpha_expr_unbounded}
				\E[\xi_n^1(A)] = \alpha \cdot \vol_{\T^2}(A), \quad \text{ for all } A \in \mathcal{B}(\T^2).	
				\vspace{0.5ex} \end{equation}
				Indeed by assumptions $T'$, $F'$, translating by $(0,x_1,x_1)$ and using (\ref{eq:v_def}) yields that
				\vspace{0.5ex} \begin{equation}
				\label{eq:unb_alpha_expansion}
				n^3 \int_{\T^{2 \cdot 2} }
				P\big((0,\xx) \in \mc N(\PP_{n} \cup \{(0,\xx)\}), \ \ell^*_{(0,\xx),\PP_{n}\cup \{(0,\xx)\}} 
				\geq \ell_{n,\alpha} \big) \ \d\xx = \alpha.
				\vspace{0.5ex} \end{equation}
				Using Mecke's formula (Section \ref{sec:1}) and repeating the translation above,
				\vspace{0.5ex} \begin{equation}
					\label{eq:expansion_3}
				\quad \quad \quad \E[\xi_n^1(A)] = \vol_{\T^2}(A) n^3\int_{\T^{2\cdot 2}} 
				P((0,\xx) \in \mc N(\PP_{n}\cup \{(0,\xx)\}), \ \ell^*_{(0,\xx),\PP_{n}\cup \{(0,\xx)\}} 
				\geq \ell_{n,\alpha} ) \ \d\xx. 
				\vspace{0.5ex} \end{equation}
				Combining (\ref{eq:unb_alpha_expansion}) and (\ref{eq:expansion_3}) yields the claim.
			\item \textit{Thresholds:} It is not immediately clear if the thresholds in assumption $T'$ are well-defined (i.e., that $v$ is invertible) but we prove this in Section \ref{sec:5}.
				Moreover, relying on the approach outlined in \cite{BKS17}, we prove in Section \ref{sec:3} that for sufficiently large $n\geq 1$, then for any $\e > 0$,
				\vspace{0.5ex} $$
				\frac{1}{18 + \e } \frac{\log(n)}{\log(\log(n))} \leq \ell_{n,\alpha} \leq 2 \frac{\log(n)}{\log(\log(n))}.
				\vspace{0.5ex} $$
			\item \textit{Extension:} Suppose we instead consider the marked point process,
			\vspace{0.5ex} $$
				\Xi_n = \sum_{\xx \in \PP_{n}^{(3)}}
				\one\{\xx \in \mc N(\PP_n)\}
				 \de\{(z_\xx,n^3 v(\ell^*_{\xx,\PP_n}))\},
			\vspace{0.5ex} $$
			and let $\Xi_n \cap A$ denote $\Xi_n$ restricted to the set $A$. For any $\alpha > 0$, we can also prove that
			\vspace{0.5ex} $$
				\dkr\big( \Xi_n \cap(\T^2 \times [0,\alpha]), \Xi \cap(\T^2 \times [0,\alpha])\big) \in O\big(n^{-1/36 + \e}\big),,
			\vspace{0.5ex} $$
			where $\Xi$ is a stationary Poisson point process with unit intensity on $\R^d \times [0,\infty)$. Indeed
			letting $\hat \Xi_n $ denote the restriction of $\Xi_n$ to mortal 1-cycles and noting $n^3 v(\ell^*_{\xx,\PP_n}) \in [0,\alpha]$ if and only if  $\ell^*_{\xx,\PP_n} \geq \ell_{n,\alpha}$,
			where $\ell_{n,\alpha} $ is chosen as in assumption $T'$, we may proceed as in Section \ref{sec:3} and apply \cite[Theorem 4.1]{BSY21}.
			Thus, we can actually use the same method
			to prove joint Poisson convergence of the centers and the transformed lifetimes $n^3 v(\ell)$. As a consequence, in Section 6, we find that the
		 \textit{largest} lifetime, after this transformation, asymptotically follows an Exponential law with mean 1.
		\end{enumerate}		
		\end{remark}

\pagebreak

\begin{center}
	\subsection{The $m$-sparse regime}
\end{center}

Now assume that for a natural number $m \geq 1$ it holds that $n(nr_n^d)^{m-1} \to \infty$ whereas $n(nr_n^d)^m \to 0$ as $n \to \infty$, 
which we call the \textit{$m$-sparse regime}. One way to interpret this condition is through a stationary Poisson point process $\PP_n$ on $\R^d$
with intensity $n \geq 1$. Indeed, let $B(x,t)$ denote the open ball with
center $x \in \R^d$ and radius $t>0$, and define $C_t = C_{t}(m)$ as the set of $m$-tuples $(y_1,\ldots,y_m)$ where
the union of the balls $B(y_j,t/2)$ for $j=1,\ldots,m$ is connected.
We call the elements of $C_{t}$ for \textit{$m$-clusters at level $t$}, and we stress that this refers to the points, rather than the connected components. Using Mecke's formula, we can compute that 
the expected number of $m$-clusters at level $t$ in $\PP_n$ centered inside the unit cube $[0,1]^d$ is proportional to $n(nt^d)^{m-1}$. 
Thus, the condition on $nr_n^d$ is actually requiring that the expected number of $m$-clusters at level $r_n$
grows to infinity yet the expected number of $(m+1)$-clusters vanishes. 
Define the deterministic \textit{maximal lifetime},
\vspace{0.5ex} \begin{equation} 
    \label{eq:lmax_def}
    \lmax^\circ = \lmax^\circ(k,m) = \sup \big\{ \ell_{\xx,\PP}^\circ \colon \vert \PP \vert = m \text{ and } \xx \in \PP^{(k)} \cap \mc N(\PP) \text{ with } r_\xx \leq 1 \big\}.
\vspace{0.5ex} \end{equation}
Since $r_n \to 0$ in the $m$-sparse regime, the lifetimes will be too small to compare to $\lmax^\circ$. Instead, since multiplicative lifetimes are scale-invariant and additive lifetimes are homogeneous,
we introduce
 \vspace{0.5ex} \begin{equation}
	\label{eq:u_statistic}
	\hat \ell^{+}_{\xx,\PP}  = \frac{\ell^{+}_{\xx,\PP}}{r_n} \quad \text{ and } \quad  \hat \ell^{*}_{\xx,\PP} = \ell^{*}_{\xx,\PP},
\vspace{0.5ex} 	\end{equation}
as the \textit{scaled lifetimes} and use the common notation $ \hat \ell^{\circ}_{\xx,\PP} $ for $\circ \in \{+,*\}$ as before.
To specify the threshold, introduce the notation $(0,\yy)_k = (0,y_1,\ldots,y_{k-1})$ and define for $u \in [0,\lmax^\circ]$, the map
\vspace{0.5ex} \begin{equation} 
	\label{eq:gdef_int}
	g(u) = \int_{\R^{d(m-1)}}
	\one\{ (0,\yy) \in C_1, \ (0,\yy)_k \in \mc N((0,\yy)), \
	r_{(0,\yy)_k} \leq 1, \
	\ell^\circ_{(0,\yy)_k, (0,\yy) }\geq \lmax^\circ - u\big\} \ \d \yy.
	\vspace{0.5ex} \end{equation} 
\bet[\textbf{"Centers $\rightarrow$ Poisson"}]
			\label{thm:1}
			Let $\alpha > 0$, $k \geq 3$ and define 
			\vspace{0.75ex} \begin{equation}
				\label{eq:xi_def_2}
					\xi_n^2 =\sum_{\xx \in \PP_{n\kappa}^{(k)}} 
				\one\big\{\xx \in \mc N(\PP_{n\kappa}), \ r_\xx \leq r_n, \ \hat \ell^\circ_{\xx,\PP_{n\kappa}}\ge \lmax^\circ - u_{n,\alpha} \big\} \de\{z_\xx\}, \\[-0.5ex]
				\vspace{0.75ex} \end{equation}
where we assume that
			\begin{itemize}[leftmargin=5mm,labelindent=50pt]
				\item[] $\aF$ The filtration is either \v Cech or Vietoris-Rips and 
				the lifetime either additive or multiplicative,	
				\item[] $\aM$ The deathtime bound $r_n>0$ satisfies that 
				$n(nr_n^d)^{m-1}\to \infty$ whereas $n(nr_n^d)^{m} \to 0$ as $n \to \infty$,	
				\item[] $\aP$ The points $\PP_{n\kappa}$ is a Poisson point process with intensity function $n \cdot \kappa$, where $\kappa$ is a bounded, \phantom{m}\phantom{i} 
				integrable function supported on a Borel set $W \su \R^d$
				satisfying that 
				\vspace{0.5ex} $$
			  \quad \quad \ \  \sup_{y,y'\in W \colon |y-y'| \le 2m r_n} \abs{\k(y) - \k(y') } \in O(r_n) \quad \text{ and } \quad \int_W 
\one\{B(y,2m r_n) \nsubseteq W\} \kappa(y) \ \d\yy \in O(r_n),
			 	 \vspace{0.5ex} $$
				\item[] $\aT$ The sequence $u_{n,\alpha}$ is equal to $g^{-1}\big(\alpha/\big(n(nr_n^d)^{m-1}\binom{m}{k}/m!\big)\big)$
				for all large $n \geq 1$.
				\item[] $\aU$ The pair $m \geq k$ satisfies that there is an $\e >0$ such that for all $\PP$ of size $m$
				there is at most \phantom{i}\phantom{i} \phantom{i}
				 one subset $\PP' \su \PP$ 
				and one $\xx \in \mc \PP^{(k)} \cap \mc N(\mc P)$ with $r_\xx \leq 1$ for which $\ell_{\xx,\PP'}^\circ \ge \lmax^\circ -  \e$.
			\end{itemize}
			If $\zeta^2$ denotes a Poisson point process on $\R^d$ with intensity function
			$ y \mapsto \a \kappa(y)^m$, it holds that
			\vspace{0.75ex} \begin{equation*}
			\dkr(\xi_n^2, \zeta^{2})\in O(n(nr_n^d)^{m}). \\[0.1em]
			\vspace{0.25ex} \end{equation*}
		\ent
	In other words, Theorem \ref{thm:1} states that
	the centers of large-lifetime cycles
	are approximately a Poisson point process with points scattered according to $\kappa^{m}$.
	In Remark \ref{rem:sparse_extension}, we discuss assumptions $P$, $T$ and $U$.
	Next, we add the scaled deviation to $\lmax^\circ$ as a mark on the centers. Let $g^{(j)}$ 
	denote the $j$'th derivative of $g$ and $C^{p}(A)$ denote all $p$-times continuous differentiable maps on the set $A$.

\bet[\textbf{"(Centers, Lifetimes) $\rightarrow$ Poisson"}]
\label{thm:2}
Let $\alpha > 0$, $k \geq 3$  and define 
\vspace{0.75ex} \begin{equation}
    \label{eq:xi_def_2}
    \xi_n^3 =\sum_{\xx \in \PP_{n\kappa}^{(k)}} 
    \one\big\{\xx \in \mc N(\PP_{n\kappa}), \ r_\xx \leq r_n \big\} 
    \de\big\{\big(z_\xx,\tfrac{\lmax^\circ-\hat \ell_{\xx,\PP_{n\kappa}}}{u_{n,\alpha}}\big)\big\}, \\[-0.5ex]
\vspace{0.75ex} \end{equation}
under assumptions $F$, $M$, $P$, $T$ and $U$ from Theorem \ref{thm:1}. Additionally, assume that
\begin{enumerate}[leftmargin=*, labelindent=14pt, label=$\aG$]
    \item There is a $q \geq 1$ such that $g \in C^{q+1}([0,\lmax^\circ])$
    with $g^{(j)}(0) = 0$ for all $j < q$, 
    and $g^{(q)}(0) > 0$.
\end{enumerate}
If $\zeta^{3}$ denotes a Poisson point process on $\R^d \times [0,\infty)$ with 
intensity function $
(y,u) \mapsto \alpha\k(y)^{m}qu^{q-1},
$
then for every compact set $K \su [0,\infty)$, it holds that
\vspace{0.75ex} \begin{equation*}
\dkr\big(\xi_n^3\cap(\R^d \times K), \zeta^{3}\cap(\R^d \times K)\big)
 \in O\big(n(nr_n^d)^{m} + \big(n(nr_n^d)^{m-1}\big)^{-1/q} \big).
\vspace{0.25ex} \end{equation*}
\ent
Theorem \ref{thm:2} tells us that larger (scaled) deviations to $\lmax^\circ$ are more likely
and these deviations are independent of the location of the cycle. For further consequences, including the \textit{largest lifetime}, see Remark \ref{rem:sparse_extension} below.
Since $\xi_n^3(\cdot \times [0,1]) = \xi_n^2$, we see Theorem \ref{thm:2} is an extension of Theorem \ref{thm:1}. Finally, we add the deviation to deathtime 1.
Define for $u\in [0,\lmax^\circ]$ and $v \in [0,1]$, the map
	 \vspace{0.75ex} \begin{equation}
	\scalebox{0.96}{$ \displaystyle
    \label{eq:hdef_int}
    h(u, v) =  \int
	\one\{ (0,\yy) \in C_1, \ (0,\yy)_k \in \mc N((0,\yy)), \
	r_{(0,\yy)_k} \in [1-v,1], \
	\ell^\circ_{(0,\yy)_k, (0,\yy) }\geq \lmax^\circ - u\big\} \ \d \yy.$}
	\vspace{0.25ex} \end{equation}
and note $h(u,1)=g(u)$. Let $h^{(i,j)}$ denote the $i,j$'th partial derivative of $h$ and set $\tilde h(u, v) = h(u, u v)$.

\bet[\textbf{"(Centers, Lifetimes, Deathtimes) $\rightarrow$ Poisson"}]
\label{thm:3}
Let $\alpha > 0$, $k \geq 3$  and define
\vspace{0.75ex} \begin{equation}
    \label{eq:xi_def_4}
    \xi_n^4 =\sum_{\xx \in \PP_{n\kappa}^{(k)}} 
    \one\big\{\xx \in \mc N(\PP_{n\kappa}) \big\} 
    \de\big\{\big(z_\xx,\tfrac{\lmax^\circ-\ell_{\xx,\PP_{n\kappa}}}{u_{n,\alpha}}, \tfrac{r_n-r_\xx}{u_{n,\alpha}} \big)\big\}, \\[-0.5ex]
\vspace{0.75ex} \end{equation}
under assumptions $F$, $M$, $P$, $T$ and $U$ from Theorem \ref{thm:1}. Additionally, assume that
\begin{enumerate}[leftmargin=*, labelindent=14pt, label=$\aH$]
    \item There is a $q\geq 1$ such that $h(\cdot,v)\in C^{q+1}([0,\lmax^\circ])$ and $h(u,\cdot) \in C^{q+1}([0,1])$
    with $\tilde h(u,0)=0$, $h^{(j,0)}(0,1) = 0$, $h^{(q,0)}(0,1)>0$,
     $\tilde h^{(j,1)}(0,v) = 0$, $\tilde h^{(q,1)}(0,v)>0$ 
    when $j<q$ for all $u,v$.
\end{enumerate}
If $\zeta^{4}$ denotes a P.p.p. on 
$\R^d \times [0,\infty)\times [0,\infty)$ with intensity function
$ (y,u,v) \mapsto \alpha\k(y) ^{m}qu^{q-1}H_q(v)$, where
$H_q(v) = \tilde h^{(q,1)}(0,v)/h^{(q,0)}(0,1)$,
then for every compact $K_1,K_2 \su [0,\infty)$, it holds that
\vspace{0.75ex} \begin{equation*}
\dkr\big(\xi_n^{4}\cap (\R^d \times K_1 \times K_2), \zeta^{4}\cap (\R^d \times K_1 \times K_2)\big)
 \in O\big(n(nr_n^d)^{m}+ \big(n(nr_n^d)^{m-1}\big)^{-1/q} \big).
\vspace{0.75ex} \end{equation*}
\ent

\begin{remark}[\textit{Concerning Theorems \ref{thm:1} - \ref{thm:3}}]
	\label{rem:sparse_extension}
	\
	\begin{enumerate}
		\item \textit{Proof:} The proof strategies of Theorems \ref{thm:1} - \ref{thm:3} are similar and in broad strokes as follows: First, we introduce
		an $U$-statistic approximation using only cycles inside $m$-clusters and bound the error of this approximation by upper bounding the expected number of $j$-clusters
		for $j \geq m+1$. Then we apply \cite[Theorem
		3.1]{decreusefond} to the U-statistic approximation and need to control the intensity measure as well as the first and second moments of the U-statistic.
		Concerning the intensity, assumption $T$ ensures the approximation of the unmarked large-lifetime centers has constant intensity. When adding the marks, we compare the marked centers 
		to the marked first coordinate and use assumption $P$ to control this difference.
		\item \textit{Assumption P:} While the conditions on the density $\kappa$ in assumption $P$ are sufficient, they are also quite technical.
		 We prove in Section \ref{sec:5} that if $W=\R^d$ or $W$ is convex with finite volume, 
		 and $\kappa$ is bounded and continuously differentiable on $W$, then assumption $P$ holds.
		 \item \textit{Assumption U:} It is not immediately clear if assumption $U$ holds for all combinations of $d$, $k$ and $m$. 
		 However, as we prove in Section \ref{sec:5}, assumption $U$ holds for the \v Cech filtration and $m=k \leq d+1$ as well as
		 for the Vietoris-Rips filtration, $k=d+1$ and $m=2(k-1)$.
		 \item \textit{Assumptions G/H:} First, note that assumption $H$ implies assumption $G$. These assumptions are the most technical in the present work as
		 we need to understand which configuration of points realizes $\lmax$ and determine how the probability of having lifetime larger than $\lmax - u$ decreases
		 as $u \to 0$. In Section \ref{sec:5}, we
		 prove that for $d=2$ and $k=3$, assumption $H$ holds with $m=q=3$ in the \v Cech case, and $m=4$ and $q=5$ in the Vietoris-Rips case.
		 Roughly speaking, the intuition of the former is that after placing one point at the origin, the second point can vary in an area of order $u^2$, but choosing
		 a deathtime $r$, which can vary by an order of $u$, will determine the third point by definition of the \v Cech complex to lie on a circumference of a circle, which has area 0.
		 Similarly, for Vietoris-Rips and $q=5$, place a point at the origin and a second point at distance $1-u$ away. Requiring the third and fourth point to be more than $1$ apart, 
		 they can each vary in an areas of order $u^2$.
		 
		 \item \textit{Thresholds:} We prove in Section \ref{sec:5} that the thresholds in assumption $T$ are well-defined.
		 When assumption $G$ holds, it follows by a Taylor expansion that 
		 \vspace{0.5ex} $$
		 g(u_{n,\alpha}) = \frac{g^{q}(0)}{q!}u_{n,\alpha}^q + O(u_{n,\alpha}^{q+1}).
		 \vspace{0.5ex} $$
		 Hence ignoring constants and the error term, it follows that
		 \vspace{0.5ex} \begin{equation}
			\label{eq:threshold_scaling}
		 u_{n,\alpha} \sim \big(n(nr_n^d)^{m-1}\big)^{-1/q}.
		 \vspace{0.5ex} \end{equation}
		 Even without assumption $G$, we will still prove in Section \ref{sec:5} that $u_{n,\alpha} \to 0$ as $n \to \infty$.
		 \item \textit{Torus:} We conjecture that it is possible to extend the Poisson approximation in Theorems \ref{thm:1} - \ref{thm:3} to $\T^d$. In particular in the case
		 of stationarity, $d=2$, $k=3$, \v Cech filtration and multiplicative lifetimes, we can use the detailed remarks in Section \ref{sec:3} to treat $\T^2$ like $\R^2$.
		 \item \textit{Largest lifetime:} Another consequence of Theorem \ref{thm:2} is that the deviation of the largest lifetime from $\lmax$ will asymptotically follow a 
		 Weibull distribution with shape parameter $q$ and scale parameter $(\gamma \alpha)^{-1/q}$, where $\gamma = \int_W \kappa^{m}(y) \ \d y$. This claim is proven in Section \ref{sec:6}.
		 \item \textit{Deathtime deviation:} Compared to Theorem \ref{thm:2}, it is more difficult to interpret the new factor $\tilde h^{(q,1)}(0,v)$ in the limiting density in
		 Theorem \ref{thm:3}. However since the density factorizes, it means that the distribution of the deathtime deviation is i.i.d. for all large-lifetime cycles. In particular, we can
		study the behavior of a region of the persistence diagram (see \text{Section 6}) and devise a simple model-checking procedure for statistical analysis (see Section 7).
\end{enumerate}		
\end{remark}






\pagebreak

\vspace{-1ex}
\section{Convergence in the unbounded regime}
\label{sec:3}

This section is dedicated to proving Theorem \ref{thm:4}. Let $\PP_n$ denote a stationary Poisson point process on $\T^2$ 
with intensity $n \geq 1$, and fix the simplex size $k=3$ for the rest of the section. 
Moreover, all references to assumptions $F'$, $P'$, and $T'$ in this section refer to the assumptions in Theorem \ref{thm:4}. The rest of the section is structured as follows: 
In Section \ref{sec:3.1}, we define \textit{mortal} cycles with deathtimes less than a reasonable bound and use the asymptotic behavior of the threshold sequence to obtain a birthtime bound. 
In Section \ref{sec:3.2}, we prove that this birthtime bound is subcritical and bound the probability of the existence of a large connected component at this filtration time. 
In Section \ref{sec:3.3}, we define a stabilizing approximation of (\ref{eq:xi_def_unbounded}) and use \cite[Theorem 4.1]{BSY21} to prove Theorem \ref{thm:4} by bounding the error terms $E_1$ through $E_4$. 
In Section \ref{sec:3.4}, we establish that we may reasonably treat \v{C}ech complexes on the torus as \v{C}ech complexes on $\R^2$. Moreover, we prove that multiple 
large-lifetime cycles imply the disjoint occurrence of a large connected component from one of the cycles. 
Finally, in Section \ref{sec:3.5}, we prove all lemmas stated in Sections  \ref{sec:3.1} - \ref{sec:3.3}.




	\subsection{Mortal large-lifetime cycles}
\label{sec:3.1}
First, we introduce an approximation of $\xi_n^1$ in (\ref{eq:xi_def_unbounded}): For $n \geq 1$, let
\vspace{0.75ex} \begin{equation}
	\label{eq:deathtime_bound}
	r_n = \sqrt{\frac{\log(n)}{n}}.
\vspace{0.75ex} \end{equation}
From \cite[Theorem 6.1]{K11}, it follows that for large enough $C>0$ and any $r > Cr_n$,
the \v Cech complex $\mathsf{C}_{r}(\PP_n)$ is contractible with high probability,
which implies the homology is trivial. Delving into the proof of \cite[Theorem 6.1]{K11},
we compute that $C \geq (4^d d^{d/2} \vol(\T^2))^{-1/2} = 1/\sqrt{32}$ works in our setting
and by choosing $C=1$, we can use $r_n$ in (\ref{eq:deathtime_bound}) as a
natural deathtime bound. We will call 1-cycles with deathtime less than $r_n$ for \textit{mortal}.
The reason why it is advantageous to only consider mortal large-lifetime cycles is that the deathtime and
lifetime bound together imply a birthtime bound.
Hence, we are interested in lower bounding the asymptotic behavior of the threshold sequence $\ell_{n,\alpha}$.
\begin{lemma}[Threshold]
	\label{prop:asymptotics}
	Under assumption $T'$, it holds for any $\e > 0 $ that
	\vspace{0.4ex} $$
	\frac{1}{18 + \e} \frac{\log(n)}{\log(\log(n))} 
	\leq \ell_{n,\alpha} \leq
	2 \frac{\log(n)}{\log(\log(n))} 
	\vspace{0.4ex} $$
	for all sufficiently large  $n \geq 1$.
\end{lemma}
The proof of Lemma \ref{prop:asymptotics} is also found in Section \ref{sec:3.5}. Using the lower bound in Lemma \ref{prop:asymptotics}
immediately yields
the following upper bound on the birthtimes, 
\vspace{0.75ex} \begin{equation}
	\label{eq:birthtime_bound}
	b_n = \frac{19 \log(\log(n))}{\sqrt{n\log(n)}}.
\vspace{0.75ex} \end{equation} 
\begin{lemma}[Birthtime bound]
	\label{lem:birthtime_bound}
	Under assumptions $F'$ and $T'$ 
	as well as $r_n$ and $b_n$ as in (\ref{eq:deathtime_bound}) and (\ref{eq:birthtime_bound}), respectively, it holds for any point cloud $\PP \in \mathfrak{N}(\T^2)$ and negative simplex $\xx \in \mc N(\PP)$ that
	\vspace{0.75ex} $$
	r_\xx \leq r_n \quad \text{ and } \quad \ell^*_{\xx,\PP} \geq \ell_{n,\alpha} \quad \implies \quad b_{\xx,\PP} \leq b_n .
	\vspace{0.75ex} $$
\end{lemma}




\subsection{Excluding large clusters}
Note that for the birthtime bound $b_n$ in (\ref{eq:birthtime_bound}), as $n \to \infty$, then
\vspace{0.75ex} \begin{equation*}
	nb_n^2 =  \frac{19^2 \log(\log(n))^2}{\log(n)} \longrightarrow 0,
\vspace{0.75ex} \end{equation*} 
and therefore all mortal large-lifetime cycles are born within a subcritical connectivity regime, in 
continuum percolation sense. 
\label{sec:3.2}
Based on this birthtime bound, we define for $n \geq 1$,
\vspace{0.4ex} \begin{equation}
	\label{eq:radius_seq}
R_n = \frac{\log(n)}{\sqrt{n}} \quad \text{ and } \quad m_n = \bigg\lceil \frac{R_n}{b_n}\bigg\rceil = \bigg\lceil \frac{\log(n)^{3/2}}{19\log(\log(n))}\bigg\rceil,
\vspace{0.75ex} \end{equation}
where $\lceil x \rceil = \min \{m \in \mathbb{N} \colon m \geq x\}$. The goal is to establish that it is very unlikely to have a connected component in $\PP_n$ of size $m_n$ or larger at filtration time $b_n$.
Recall that $C_{b}(m)$ denotes the set of $m$-clusters at filtration time $b$ (Section \ref{sec:2}).
For $m \geq 1$ and $b > 0$, define the set 
\vspace{0.75ex} \begin{equation}
   \label{eq:L_n}
\mc C_b^m(\PP_n) =   \{\yy \in \PP_n^{(m)} \colon \yy \in C_{b}(m)\}.
\vspace{0.75ex} \end{equation}
When convenient,
write $\vol(\cdot)$ instead of $\vol_{\T^2}(\cdot)$,
and $\PP_n^\xx$ instead of $\PP_n \cup \{ x_1,x_2,x_3 \}$.
	 
	 \begin{lemma}[Large clusters]
		 \label{lem:Ln_bound}
	 Under assumptions $P'$ and $T'$, then for any Borel set $W \su \T^2$ and $\xx \in \T^{2\cdot p}$, for $p \leq m$, it holds that
	 \vspace{0.75ex} $$
	 P\big(\mc C_{b_n}^m(\PP_n^\xx \cap W) \neq \emptyset \big) \leq  p!p n \vol(W) \bigg(19^2 \e \pi\frac{\log(\log(n))^2}{\log(n)}\bigg)^{m-1}.
	 \vspace{0.75ex} $$
	 \end{lemma}
	 From applying Lemma \ref{lem:Ln_bound}
	 with $W = \T^2$ and $m = m_n$ as well as using that $m_n \geq \log(n)$ for all sufficiently large $n$, then for any $\xx \in \T^{2\cdot p}$, we see that
	 \vspace{0.75ex} \begin{equation}
	 P\big(\mc C_{b_n}^{m_n}(\PP_n^\xx) \neq \emptyset \big) \in O\bigg( n \bigg(19^2 \e \pi \frac{\log(\log(n))^2}{\log(n)}\bigg)^{\log(n)-1}\bigg) \su O(n^{-\beta}) \quad \text{ for any $\beta > 0$}.
	 \vspace{0.75ex} \end{equation} 
	In other words, this proves the claim that clusters of size $m_n$ or larger are very unlikely at time $b_n$. 




\subsection{Proof of Theorem \ref{thm:4}}
\label{sec:3.3}	
	Define for a point cloud $\PP$ the set of all 2-simplices $\xx$, where if $\PP$ is \textit{restricted} to $B(z_\xx,R_n)$,
	then $\xx$ is mortal large-lifetime cycle, i.e,
	\vspace{0.75ex} \begin{equation}
		\label{eq:common_clusters}
		\mc B(\PP) = \big\{\xx \in \mc N(\PP \cap B(z_\xx, R_n)) \colon r_\xx \leq r_n \text{ and } \ell^*_{\xx,\PP \cap B(z_\xx, R_n )} \geq \ell_{n,\alpha} \},
	\vspace{0.75ex} \end{equation}
	as well as the centers of such 2-simplices in the Poisson point process $\PP_n$,
	\vspace{0.75ex} \begin{equation}
		\label{eq:xi_def_mortal}
		\hat \xi_n^1 =\sum_{\xx \in \PP_{n}^{(3)}} 
		\one \{\xx \in \mc B(\PP_{n}) \} \de\{z_\xx\}.
	\vspace{0.75ex} \end{equation}
	The next lemma bounds the Kantorovich-Rubinstein error in using (\ref{eq:xi_def_mortal}) to approximate (\ref{eq:xi_def_unbounded})
	and the proof is found in Section \ref{sec:3.5}.
	Recall $\dtv$ denotes the total variation distance (Section \ref{sec:1}).
	\begin{lemma}[Error of approximation]
			\label{lem:deathtime_bound}
		Under assumptions $F'$ and $P'$, it holds that
		\vspace{0.75ex} $$
		\dkr(\xi_n^1, \hat \xi_n^1) \leq \dtv(\E[\xi_n^1],\E[\hat \xi_n^1]) \in O\big(n^{-1/8}\big). \\[0.1em]
		 \vspace{0.75ex} $$
		\end{lemma}
A consequence of Lemma \ref{lem:deathtime_bound}, assumption $T'$ and Remark \ref{rem:unbounded_extension}(2) is that the intensity of $\hat \xi_n^1$ is asymptotically constant, i.e.,
there is a constant $C > 0$ such that for all sufficiently large $n \geq 1$,
\vspace{0.75ex} \begin{equation}
	\label{eq:almost_alpha}
\alpha - C n^{-1/8} \leq \E[\hat \xi_n^1 (\T^2)] \leq \alpha + C n^{-1/8} \leq 2\alpha.
\vspace{0.75ex} \end{equation}
The idea is to apply the Poisson approximation for stabilizing functionals in \cite[Theorem 4.1]{BSY21} to $\hat \xi_n^1$ instead of $\xi_n^1$
in order to prove Theorem \ref{thm:4}. Recall that a map $f$ is localized to a set $S$ if $f(\xx,\PP) = f(\xx, \PP \cap S)$ for any point cloud $\PP \in \mathfrak{N}(\mathbb{T}^2)$ and $\xx \in \PP^{(3)}$. As a stopping set, we choose the ball $B(z_\xx,R_n)$ and note that the maps
\vspace{0.75ex} \begin{equation}
	\label{eq:scores}
f: \xx \mapsto z_\xx \quad \text{and} \quad g: (\xx,\PP) \mapsto \one\big\{\xx \in \mc B(\PP) \},
\vspace{0.75ex} \end{equation}
are permutation invariant in $\xx$. Moreover $f$ is localized to $B(z_\xx,R_n)$ since $r_n < R_n$ and $g$ is localized to $B(z_\xx,R_n)$
by definition of $B(\PP)$.
Choosing the bounding set to be $B(z_\xx,R_n)$ as well, it follows that the first error term $E_1$ from \cite[Theorem 4.1]{BSY21} is 
\vspace{0.75ex} $$
E_1 = n^3\int_{\T^{2\cdot 3}}
P\big(\xx \in \mc B(\PP_n^\xx),
 \ B(z_\xx, R_n) \not\subseteq B(z_\xx, R_n)  \big) \ \d\xx = 0.
\vspace{0.75ex} $$
Hence, we only have to deal with three error terms: For $\es \subsetneq I \subsetneq \{1,2,3 \}$ and $\xx_I = (x_i)_{i \in I}$, let
\vspace{0.75ex} \begin{equation}
	\label{eq:E_terms_def}
\begin{aligned}
	E_2 &= n^6 \int_{\T^{2\cdot 3}} \int_{\T^{2\cdot 3}}
	\one \{ B(z_\xx,R_n) \cap B(z_\yy,R_n) \neq \emptyset  \} P\big(\xx \in \mc B(\PP_n^\xx)\big) 
	P\big(\yy \in \mc B(\PP_n^\yy)\big)
		 \ \d\xx \d\yy, \\[1ex]
	E_3 & = n^6 \int_{\T^{2\cdot 3}} \int_{\T^{2\cdot 3}}
	\one \{ B(z_\xx,R_n) \cap B(z_\yy,R_n) \neq \emptyset  \} P\big(\xx \in \mc B(\PP_n^{\xx,\yy}), \yy \in \mc B(\PP_n^{\xx,\yy}) \big) 
		  \ \d\xx \d\yy, \\[1ex]
	E_4^I &=
	  n^{6-|I|} \int_{\mathbb{T}^{2(3-|I|)}} \int_{\T^{2\cdot 3}}
	  P\big(\xx \in \mc B(\PP_n^{\xx,\yy}), (\xx_I,\yy) \in \mc B(\PP_n^{\xx,\yy}) \big) 
		  \ \d\xx \d\yy. \\[1ex]
\end{aligned}
\vspace{0.75ex} \end{equation}


\begin{lemma}[$E_2 \rightarrow 0$]
	\label{lem:E_2}
	Under assumptions $F'$, $P'$, and $T'$, it holds for every $\e > 0$ that 
	\vspace{0.75ex} $$
	E_2 \in O\big(n^{-1+\e}\big).
	\vspace{0.75ex} $$
\end{lemma}

	
\begin{lemma}[$E_3 \rightarrow 0$]
	\label{lem:E_3}
	Under assumptions $F'$, $P'$, and $T'$, it holds for every $\e > 0$ that 
	\vspace{0.75ex} $$
		E_3 \in  O\big(n^{-1/36+ \e}\big).
	\vspace{0.75ex} $$
\end{lemma}


\begin{lemma}[$E_4^I \rightarrow 0$]
	\label{lem:E_4}
	Under assumptions $F'$, $P'$, and $T'$, it holds for every $\e > 0$ that 
	\vspace{0.75ex} $$
		E_4^I \in  O\big(n^{-1/36+ \e}\big).
	\vspace{0.75ex} $$
\end{lemma}
The proof of Lemma \ref{lem:E_2} 
relies on translation invariance of $\vol(\cdot)$ and Remark \ref{rem:unbounded_extension}(2).
The proof of Lemma \ref{lem:E_3} deals with the interaction of the two cycles
$\xx$ and $\yy$. In the next subsection, we will show that this implies that a 
large connected component disjointly occurs from one of the cycles, and then apply the continuous BK inequality.
The proof of Lemma \ref{lem:E_4} is very similar to the proof of Lemma \ref{lem:E_3}
once we establish that sharing a vertex implies that $R_n$-balls around the centers of
$\xx$ and $\yy$ must overlap as in $E_2$ and $E_3$. We postpone the proofs in full detail until Section \ref{sec:3.5}.
Note it is the rate of convergence in the $E_3$ and $E_4$ terms which determine
the convergence speed of the upper bound in Theorem \ref{thm:4}. While 
the proof for $E_2$ also works for $d>2$, the decomposition needed
for $E_3$ and $E_4$ relies on the fact that $d=2$, the Jordan Curve Theorem, 
and that we can define the loop associated with a negative simplex through the vacant set of a graph.


\begin{proof}[Proof of Theorem \ref{thm:4}]
	Recall we want to prove that
$
	\dkr(\xi_n^1, \zeta^{1})\in O\big(n^{-1/36+\e}\big)
$	for every $\e > 0$. 
Since we already argued that $f$ and $g$ in (\ref{eq:scores}) localizes to $B(z_\xx,R_n)$, then by
\cite[Theorem 4.1]{BSY21},
\vspace{0.75ex} \begin{equation}
	\label{eq:terms_from_BSY}
\dkr(\hat \xi_{n}^1, \zeta^1) \leq \dtv(\E[\hat \xi^1_n], \E[\zeta^1]) + E_2 + E_3 + \sum_{\es \subsetneq I \subsetneq \{1,2,3 \}} E_4^I,
\vspace{0.75ex} \end{equation} 
where $E_2$, $E_3$ and $E_4^I$ are as defined in (\ref{eq:E_terms_def}). By Lemma \ref{lem:deathtime_bound} and Remark \ref{rem:unbounded_extension}(2), we have
\vspace{0.75ex} \begin{equation}
	\label{eq:intensity_rate}
\dtv(\E[\hat \xi^1_n], \E[\zeta^1]) \leq \dtv(\E[\hat \xi^1_n], \E[\xi^1_n]) + \dtv(\E[\xi^1_n], \E[\zeta^1]) \leq \E[\dtv(\hat \xi^1_n, \xi^1_n)]  \in O(n^{-1/8}).
\vspace{0.75ex} \end{equation}
Combining and (\ref{eq:terms_from_BSY}) and (\ref{eq:intensity_rate}) with Lemmas \ref{lem:E_2}, \ref{lem:E_3}, and \ref{lem:E_4}, then 
\vspace{0.75ex} $$
\dkr(\hat \xi_{n}^1, \zeta^1) \in O(n^{-1/8} + n^{-1+\e} + n^{-1/36+ \e} + n^{-1/36+ \e}) \subseteq O(n^{-1/36+ \e}),
\vspace{0.75ex} $$
and therefore using Lemma \ref{lem:deathtime_bound} once more, it follows that  
\vspace{0.75ex} $$
\dkr(\hat \xi_{n}^1, \zeta^1) \leq \dkr(\hat \xi_{n}^1, \hat \xi_{n}^1) + \dkr(\hat \xi_{n}^1, \zeta^1) \in O(n^{-1/36+ \e}),
\vspace{0.75ex} $$ 
which completes the proof.
\end{proof}




\subsection{\v Cech complexes on the torus}
\label{sec:3.4}
The goal of this section is to transfer some properties of the \v Cech filtration in Euclidean space to the torus as well. 
The main challenge is of course that regions on the torus can 'wrap around' the boundary, which for instance
means that the intersection of two balls may not be a contractible set. However, we now argue that our objects are sufficiently local
so we can avoid these issues.
\begin{remark}[\textit{\v{C}ech properties}]
	\label{rem:cech_properties}
	\
	\begin{enumerate}
		\item \textit{Negative property:} Since any two points on the torus are at most a toroidal distance of $1/\sqrt{2}$ from one another, it follows that $B(z_\xx,r_\xx) \subseteq B(z_\xx, 1/\sqrt{6})$ by using the relationship between the side length and the distance from a vertex to the centroid of an equilateral triangle. By translating such that $z_\xx$ is in the middle of the torus, we see that $B(z_\xx, r_\xx)$ can be isometrically embedded in $\R^2$. Letting $\text{conv}(\xx)$ denote the convex hull of $\xx$, it follows from \cite[Lemma 2.4]{O22} that
		\vspace{0.75ex} \begin{equation}
			\label{eq:cech_convex}
			\mc N(\PP_n) = \{\xx \in \PP_n^{(3)} \colon \PP_n \cap B(z_\xx,r_\xx) = \es, \ z_\xx \in \text{int}(\text{conv}(\xx)) \},
		\vspace{0.75ex} \end{equation}
		since all critical 2-simplices are negative when $d=2$.

		\item \textit{Alpha complex:} As $R_n \to 0$ as $n \to \infty$, it follows that for sufficiently large $n$, the ball $B(z_\xx, R_n)$ does not wrap around the torus and may also be isometrically embedded in $\R^2$. In particular, we may use the Jordan Curve Theorem restricted to this ball. Let $\xx \in \mc B(\PP_n)$ and consider the set of 1-simplices in the \v{C}ech complex $\mathsf{C}_{b_{\xx,\PP_n}}(\PP_n)$ embedded as edges in $\R^2$. Then remove all edges between points with non-neighboring Voronoi cells (\cite{chazal}). This creates the \textit{Alpha complex} (\cite{chazal}). Some properties of the Alpha complex include that it has the same homology as the \v{C}ech complex and that the center $z_\xx$ is in the interior of the closed subgraph at the birthtime of $\xx$.

		\item \textit{Associated cluster:} Since $\PP_n$ is Poisson, almost surely there is a unique edge of length $2b_{\xx,\PP_n}$. Let the cluster at time $b_{\xx,\PP_n}$ contain the vertices of this unique edge, which we call \textit{the cluster associated to $\xx$} (see Figure \ref{fig:associated}). Due to the properties of the Alpha complex, it follows that $z_\xx$ lies in the interior of the convex hull of the associated cluster of $\xx$.

		\item \textit{Associated loop:} Let $G(\PP_n, b_{\xx,\PP_n})$ denote the planar graph with the cluster associated to $\xx$ as its vertices and the 1-simplices between them at time $b_{\xx,\PP_n}$ as its edges. Due to the Jordan Curve Theorem (if $\xx \in \mc B(\PP_n)$), the vacant set $B(z_\xx,s_{\xx,\PP_n}) \setminus G(\PP_n, b_{\xx,\PP_n})$ has exactly two connected components, each with this edge on its boundary. Now take the connected component containing $z_\xx$, which is possible via the Alpha complex, and consider the boundary of this component. If this boundary is disconnected and consists of several loops, then pick the external boundary as described in \cite{GHK14,HV22}, i.e., the boundary component with no other components in its exterior. Henceforth, we refer to this closed, non-intersecting graph as \textit{the loop associated to $\xx$} (see Figure \ref{fig:associated}).
	\end{enumerate}
\end{remark}

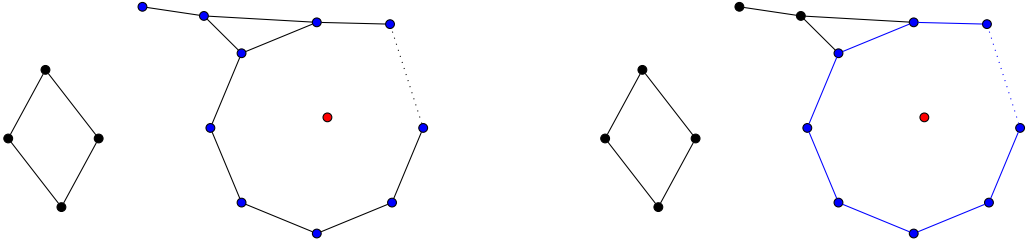
\begin{figure}[h]
    \centering
	\hspace{-0.1\textwidth}
    \begin{subfigure}[t]{0.25\textwidth}
        \centering
        \begin{tikzpicture}[scale=0.7, shift={(-0.5,0)}]
            \node[draw, circle, fill=blue, scale=0.35] (oct1) at (0:2) {};
            \node[draw, circle, fill=blue, scale=0.35] (oct2) at (55:2.4) {};
            \node[draw, circle, fill=blue, scale=0.35] (oct3) at (90:2) {};
            \node[draw, circle, fill=blue, scale=0.35] (oct4) at (135:2) {};
            \node[draw, circle, fill=blue, scale=0.35] (oct5) at (180:2) {};
            \node[draw, circle, fill=blue, scale=0.35] (oct6) at (225:2) {};
            \node[draw, circle, fill=blue, scale=0.35] (oct7) at (270:2) {};
            \node[draw, circle, fill=blue, scale=0.35] (oct8) at (315:2) {};
            
            \draw[dotted, black] (oct1) -- (oct2);
            \draw[black] (oct2) -- (oct3) -- (oct4) -- (oct5) -- (oct6) -- (oct7) -- (oct8) -- (oct1);
            
            \node[draw, circle, fill=blue, scale=0.35] (new1) at (135:3) {};
            \node[draw, circle, fill=blue, scale=0.35] (new2) at (145:4) {};
            \draw[black] (oct4) -- (new1);
            \draw[black] (new1) -- (new2);
            \draw[black] (oct3) -- (new1);
            
            \node[draw, circle, fill=black, scale=0.35] (d1) at (-4.1, -0.2) {};
            \node[draw, circle, fill=black, scale=0.35] (d2) at (-5.1, 1.1) {};
            \node[draw, circle, fill=black, scale=0.35] (d3) at (-5.8, -0.2) {};
            \node[draw, circle, fill=black, scale=0.35] (d4) at (-4.8, -1.5) {};
            
            \draw (d1) -- (d2) -- (d3) -- (d4) -- (d1);
        
            \node[draw, circle, fill=red, scale=0.35] at (0.2, 0.2) {};
        \end{tikzpicture}
        
    \end{subfigure}
    \hspace{0.25\textwidth}
    \begin{subfigure}[t]{0.25\textwidth}
        \centering
        \begin{tikzpicture}[scale=0.7, shift={(-6.5,0)}]
            \node[draw, circle, fill=blue, scale=0.35] (oct1) at (0:2) {};
            \node[draw, circle, fill=blue, scale=0.35] (oct2) at (55:2.4) {};
            \node[draw, circle, fill=blue, scale=0.35] (oct3) at (90:2) {};
            \node[draw, circle, fill=blue, scale=0.35] (oct4) at (135:2) {};
            \node[draw, circle, fill=blue, scale=0.35] (oct5) at (180:2) {};
            \node[draw, circle, fill=blue, scale=0.35] (oct6) at (225:2) {};
            \node[draw, circle, fill=blue, scale=0.35] (oct7) at (270:2) {};
            \node[draw, circle, fill=blue, scale=0.35] (oct8) at (315:2) {};
            
            \draw[dotted, blue] (oct1) -- (oct2);
            \draw[blue] (oct2) -- (oct3) -- (oct4) -- (oct5) -- (oct6) -- (oct7) -- (oct8) -- (oct1);
            
            \node[draw, circle, fill=black, scale=0.35] (new1) at (135:3) {};
            \node[draw, circle, fill=black, scale=0.35] (new2) at (145:4) {};
            \draw[black] (oct4) -- (new1);
            \draw[black] (new1) -- (new2);
            \draw[black] (oct3) -- (new1);
            
            \node[draw, circle, fill=black, scale=0.35] (d1) at (-4.1, -0.2) {};
            \node[draw, circle, fill=black, scale=0.35] (d2) at (-5.1, 1.1) {};
            \node[draw, circle, fill=black, scale=0.35] (d3) at (-5.8, -0.2) {};
            \node[draw, circle, fill=black, scale=0.35] (d4) at (-4.8, -1.5) {};
            
            \draw (d1) -- (d2) -- (d3) -- (d4) -- (d1);
        
            \node[draw, circle, fill=red, scale=0.35] at (0.2, 0.2) {};
        \end{tikzpicture}
        
    \end{subfigure}
    \caption{In both cases, the red point is the center $z_\xx$ of $\xx$ and the dotted line is the unique
	edge of length $b_{\xx,\PP_n}$. Left: All blue points is the cluster associated to $\xx$.
	Right: All blue vertices \textit{and} edges is the loop associated of $\xx$. }
	\label{fig:associated}
\end{figure}

In order to control the last two error terms, $E_3$ and $E_4^I$, we need to tool to
	deal with multiple exceedances. In particular, we will prove than 
	when the birthtimes are bounded by $b \geq 0$,
	then two large cycles implies that at filtration time $b \geq 0$
	there exists a large connected component disjointly from 
	one of the cycles. Let $A \circ B$ denote the disjoint occurrence (\cite{HHLM19,Meester}) of 
	the sets $A$ and $B$.

	\begin{lemma}[Occurence]
		\label{lem:disjoint_occurence}
	Under assumptions $P'$ and $F'$, then for any  $b, \ell \geq 0$
	and $\xx, \yy \in \mc B(\PP_n)$,
	$$
	 \left\{ b_{\xx,\PP_n} \leq b, \ \ell^*_{\xx,\PP_n} \geq \ell \right\}
	 \cap \left\{b_{\yy,\PP_n} \leq b, \ \ell^*_{\yy,\PP_n} \geq \ell \right\}
	  \su \left\{ b_{\xx,\PP_n} \leq b, \ \ell^*_{\xx,\PP_n} \geq \ell \right\}  \circ
	 \bigcup_{\yy \in \PP_n^{(m )}} \{\yy \in C_{b}(m ) \},
	 $$
	where $m = \lceil \ell/2 \rceil$. Moreover, the loop associated to $\xx$ is of size larger than $\ell^*_{\xx,\PP_n}$. 
	\end{lemma}
	\begin{proof}
     Suppose $b_{\xx,\PP_n} \leq b_{\yy,\PP_n} \leq b$ and $\ell^*_{\xx,\PP_n}, \ell^*_{\yy,\PP_n} \geq \ell$.
	Since $\xx$ and $\yy$ are negative, then by Remark \ref{rem:cech_properties}(1) both $B(z_\xx,r_\xx)\cap \PP_n$ and
	$B(z_\yy,r_\yy)\cap \PP_n$ are empty. Let $V(G)$ and $L(G)$ denote the vertex set and the total edge length of a graph $G$, respectively.
	Moreover, let $\gamma_{\xx,\PP_n}$ and $\gamma_{\yy,\PP_n} $ denote the loops associated to $\xx$ and $\yy$ as described in Remark \ref{rem:cech_properties}(4).
	Finally, let $\gamma_1 \su \gamma_{\xx,\PP_n} $ and $\gamma_2 \su \gamma_{\yy,\PP_n} $ denote two connected subgraphs
	sharing no edges and chosen as large as possible so there is a third (possibly empty) connected subgraph $\gamma_0$ such that 
	$\gamma_0 \cup \gamma_1 \cup \gamma_2 = \gamma_{\xx,\PP_n} \cup \gamma_{\yy,\PP_n}$ and $\abs{V(\gamma_0) \cap V(\gamma_1) \cap V(\gamma_2)} \leq 2$
	(see Figure \ref{fig:curves}).
	Note that it now suffices to prove that $\abs{V(\gamma_1)} + \abs{V(\gamma_2)} \geq \ell$
	as we then take either $\gamma_1$ or $\gamma_2$ as the connected component with more than $\ell/2$ vertices.
	Suppose towards a
	contradiction that $\abs{V(\gamma_1)} + \abs{V(\gamma_2)} < \ell$. Hence
	\vspace{0.75ex} \begin{equation}
		\label{eq:graph_length}
	 L(\gamma_1 \cup \gamma_2) \leq (\abs{V(\gamma_1)} + \abs{V(\gamma_2)})b_{\yy,\PP_n} \leq r_{\yy}.
	\vspace{0.75ex} \end{equation}
	With the notion of
	inside and outside of a closed loop from the Jordan Curve Theorem, it follows that if 
	$\abs{V(\gamma_0) \cap V(\gamma_1) \cap V(\gamma_2)} \leq 1$, i.e. $\gamma_0$ is empty or a single vertex, 
	then as $z_\yy$ is in the inside of the closed loop $\gamma_2$, then $\gamma_2 \su B(z_\yy, r_\yy )$ by (\ref{eq:graph_length}) and therefore contradicting
	that $B(z_\yy, r_\yy ) \cap \PP_n$ is empty. If instead $\abs{V(\gamma_0) \cap V(\gamma_1) \cap V(\gamma_2)} = 2$, then 
	$z_\yy$ is inside the closed loop $\gamma_1 \cup \gamma_2$ and again by (\ref{eq:graph_length}), then 
	$\gamma_2 \su B(z_\yy, r_\yy )$ which again is a contradiction. Finally, if $\abs{V(\gamma_{\xx,\PP_n})} < \ell^*_{\xx,\PP_n}$, then
	\vspace{0.75ex} \begin{equation*}
	 L(\gamma_{\xx,\PP_n}) \leq  \abs{V(\gamma_{\xx,\PP_n})} b_{\xx,\PP_n} \leq r_{\xx}.
	\vspace{0.75ex} \end{equation*}
	which again leads to a contradiction and thus completes the proof. 
	\end{proof}

	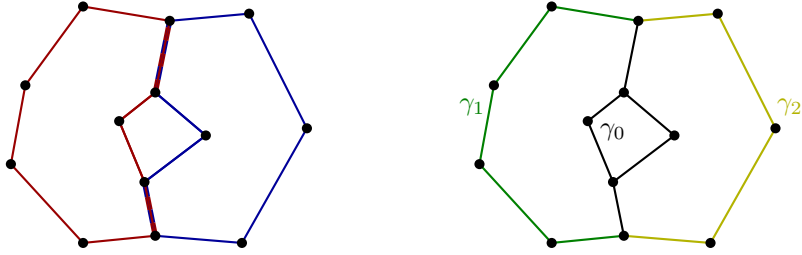
\begin{figure}[htbp]
    \centering
    \begin{subfigure}[t]{0.5\textwidth}
        \centering
        
\begin{tikzpicture}[scale=0.95]

    \definecolor{dark red}{rgb}{0.6,0,0}
    \definecolor{dark blue}{rgb}{0,0,0.6}
    
    \coordinate (A) at (-0.1, 0);
    \coordinate (B) at (-0.25, 0.75); 
    \coordinate (C) at (-0.1, 2);
    \coordinate (D) at (0.1, 3);
    
    \coordinate (K) at (-0.6, 1.6);
    \coordinate (L) at (0.6, 1.4);
    
    \coordinate (E) at (1.2, 3.1);
    \coordinate (M) at (2.0, 1.5); 
    \coordinate (F) at (1.1, -0.1);
    
    \coordinate (G) at (-1.1, 3.2);
    \coordinate (H) at (-1.9, 2.1); 
    \coordinate (I) at (-2.1, 1.0); 
    \coordinate (J) at (-1.1, -0.1);
    
    \tikzset{
        alternating edge/.style={
            line width=2pt,
            dash pattern=on 5pt off 5pt,
        }
    }
    
    \draw[alternating edge, dark red, dash phase=0pt] (A) -- (B);
    \draw[alternating edge, dark blue, dash phase=5pt] (A) -- (B);
    
    \draw[alternating edge, dark red, dash phase=0pt] (C) -- (D);
    \draw[alternating edge, dark blue, dash phase=5pt] (C) -- (D);
    
    \draw[dark red, thick] (B) -- (K);
    \draw[dark red, thick] (K) -- (C);
    
    \draw[dark blue, thick] (B) -- (L);
    \draw[dark blue, thick] (L) -- (C);
    
    \draw[dark blue, thick] (D) -- (E) -- (M) -- (F) -- (A);
    \draw[dark blue, thick] (A) -- (B) -- (L) -- (C) -- (D);
    
    \draw[dark red, thick] (D) -- (G) -- (H) -- (I) -- (J) -- (A);
    \draw[dark red, thick] (A) -- (B) -- (K) -- (C) -- (D);
    
    \foreach \point in {A, B, C, D, E, F, G, H, I, J, K, L, M}
    {
        \fill (\point) circle (2pt);
    }
    
    \end{tikzpicture}
    \end{subfigure}%
    \hspace{-0.1\textwidth} 
    \begin{subfigure}[t]{0.5\textwidth}
        \centering
        \begin{tikzpicture}[scale=0.95]

            \definecolor{dark green}{rgb}{0,0.5,0}
            \definecolor{dark yellow}{rgb}{0.7,0.7,0}
            
            \coordinate (A) at (-0.1, 0);
            \coordinate (B) at (-0.25, 0.75); 
            \coordinate (C) at (-0.1, 2);
            \coordinate (D) at (0.1, 3);
            
            \coordinate (K) at (-0.6, 1.6);
            \coordinate (L) at (0.6, 1.4);
            
            \coordinate (E) at (1.2, 3.1);
            \coordinate (M) at (2.0, 1.5); 
            \coordinate (F) at (1.1, -0.1);
            
            \coordinate (G) at (-1.1, 3.2);
            \coordinate (H) at (-1.9, 2.1); 
            \coordinate (I) at (-2.1, 1.0); 
            \coordinate (J) at (-1.1, -0.1);
            
            \draw[black, thick] (A) -- (B);
            \draw[black, thick] (B) -- (K);
            \draw[black, thick] (K) -- (C);
            \draw[black, thick] (C) -- (D);
            \draw[black, thick] (B) -- (L);
            \draw[black, thick] (L) -- (C);
            
            \draw[dark green, thick] (D) -- (G);
            \draw[dark green, thick] (G) -- (H);
            \draw[dark green, thick] (H) -- (I);
            \draw[dark green, thick] (I) -- (J);
            \draw[dark green, thick] (J) -- (A); 
            
            \draw[dark yellow, thick] (D) -- (E);
            \draw[dark yellow, thick] (E) -- (M);
            \draw[dark yellow, thick] (M) -- (F);
            \draw[dark yellow, thick] (F) -- (A);
            
            \foreach \point in {A, B, C, D, E, F, G, H, I, J, K, L, M}
            {
                \fill (\point) circle (2pt);
            }
            
            \node[dark green] at ($(H) + (-0.3,-0.3)$) {$\gamma_1$};
            \node[dark yellow] at ($(M) + (0.2,0.3)$) {$\gamma_2$};
            \node[black] at ($(B) + (0,0.7)$) {$\gamma_0$};
            
            \end{tikzpicture}
    \end{subfigure}
    \caption{Two loops (left) sharing 2 edges and 4 vertices and an illustration of the choice of $\gamma_0$, $\gamma_1$ and $\gamma_2$ (right)
     in the proof of Lemma \ref{lem:disjoint_occurence}.}
     \label{fig:curves}
\end{figure}

%
%

	\vspace{0.75ex}
	\subsection{Proof of Lemmas \ref{prop:asymptotics} to \ref{lem:E_4}}
	\label{sec:3.5}

\begin{proof}[Proof of Lemma \ref{prop:asymptotics}]
	First, to make the notation more compact, let 
	\vspace{0.4ex} \begin{equation}
		\label{eq:threshold}
	a_n = \frac{\log(n)}{\log(\log(n))}.
	\vspace{0.4ex} \end{equation}
	\noindent \textit{Lower bound:} Note that it suffices to show for every $\e > 0$ that as $n \to \infty$, then
	\vspace{0.4ex} \begin{equation*}
		\label{eq:threshold_ineq}
		\E\bigg[\sum_ {\xx \in \PP_n^{(3)}} \one \big\{ \xx \in \mc N(\PP_n), \ \ell^*_{\xx,\PP_n} > \tfrac{a_n}{9(2+\e)} \big\}\bigg] \longrightarrow \infty,
	\vspace{0.4ex} \end{equation*}
	since if this holds and $\ell_{n,\alpha} < a_n/(9(2+\e))$, then the constant expected number $\alpha>0$ of 1-cycles with lifetime larger than $\ell_{n,\alpha} $
	would be violated (i.e., Remark \ref{rem:unbounded_extension}(2)). To that end, we adapt the discretization argument found in \cite{BKS17}: 

	\noindent Let $\e>0$ and $\beta = 2 + \e$. Let $\ell_n = n^{-1/2}\log(n)^{-\beta/2}$, and let $L_n = a_n\ell_n/(4\beta)$.
	Now split $[0,1]^2$ into 
	$M_n \geq 1$ squares $Q_i$, which are the \textit{largest} number of squares with equal side length $2L_n$ which can be stacked to fit inside $[0,1]^2$.
	Let $\tilde Q_i$ denote the shrunk version of $Q_i$ with side length $L_n$ such that distance to the boundary of $Q_i$ is $L_n/2$
	and along the edge of $\tilde Q_i$, place $m_n \geq 1$ smaller squares
	$Q_{i,j}$ of side length $\ell_n$ (see Figure 5). Thus we can think of $Q_{i,1} \cup \ldots \cup Q_{i,m_n}$ as an (approximate) $\ell_n$-thick boundary of $\tilde Q_i$.
	Define the events
	\vspace{0.75ex} $$
		A_i = \big\{\PP_n(Q_i) = m_n \big\} \cap \bigg( \bigcap_{j=1}^{m_n} \big\{ \PP_n(Q_{i,j}) = 1 \big\}\bigg), \quad \text{for } i=1,\ldots,M_n,
	\vspace{0.75ex} $$
	i.e., $A_i$ is the event that $Q_i$ has $m_n$ points and each smaller cube $Q_{i,j}$ has a single point and no points outside these cubes.
	From the Poisson distribution and spatial independence, it follows that
	\vspace{0.75ex} $$
		P(A_i) = P(\PP_n(Q_i) = m_n) \prod_{j=1}^{m_n} P(\PP_n(Q_{i,j}) = 1 \mid \PP_n(Q_i) = m_n ) = (n \ell_n^2)^{m_n}\text{e}^{-4nL_n^2}.
	\vspace{0.75ex} $$
	Thus, using that $n \ell_n^2 = \log(n)^{-\beta}$ and $n L_n^2 = a_n^2 \log(n)^{-\beta}/(16 \beta^2)$, then
\vspace{0.75ex} $$
	\prod_{i=1}^{M_n} (1-P(A_i)) =  \big(1-(n \ell_n^2)^{m_n}\text{e}^{-4nL_n^2} \big)^{M_n} \leq \exp\big(-M_n\text{e}^{-\beta m_n\log(\log(n))-a_n^2 \log(n)^{-\beta}}\big).
\vspace{0.75ex} $$
Since $m_n \leq 4(L_n/\ell_n) = a_n/\beta$ and $M_n \geq 16 \beta^2 L_n^{-2} / 2 = 8 \beta^2 n \log(n)^\beta a_n^{-2}$, it follows that
\vspace{0.75ex} $$
M_n\text{e}^{-\beta m_n\log(\log(n))-a_n^2 \log(n)^{-\beta}} \geq 8 \beta^2 n \log(n)^{\beta-2}\log(\log(n))^2 \text{e}^{- \log(n)-4\log(n)^{2-\beta}/\log(\log(n))} \rightarrow \infty,
\vspace{0.75ex} $$
since $\beta > 2$. Thus with high probability, at least one $A_i$ occurs. On this $A_i$, since $L_n/\ell_n \to \infty$, it follows for sufficiently large $n$ that the $m_n$ points consistute a 1-cycle
with birthtime upper bounded by $\sqrt{5} \ell_n/2$, since the largest pairwise distance between two points in two $Q_{i,j}$ squares is $\sqrt{5}$, and deathtime lower bounded  by $L_n/2 - \ell_n$.
 Hence with high probability, there is a cycle with multiplicative lifetime larger than $a_n / (4\sqrt{5}\beta) - 2/\sqrt{5} \geq a_n/(9 \beta)$ for all sufficiently large $n$.
Now split $\T^2$ into $n^2$
squares with side length $1/n$ and repeat the above construction (with appropriate scaling of $\ell_n$ and $L_n$) to conclude
that, with high probability, there exist $n^2$ 1-cycles with lifetime larger than $a_n/(9\beta)$, which completes the proof of the lower bound. 

\noindent \textit{Upper bound:} As before, it suffices to show, as $n \to \infty$, that 
\vspace{0.75ex} \begin{equation*}
		\label{eq:threshold_zero}
		\E\bigg[\sum_ {\xx \in \PP_n^{(3)}} \one \{ \xx \in \mc N(\PP_n), \ \ell^*_{\xx,\PP_n} > 2 a_n\}\bigg] \longrightarrow 0.
		\vspace{0.75ex} \end{equation*}
	By Lemma \ref{lem:deathtime_bound}, it suffices to show that
	\vspace{0.75ex} \begin{equation*}
		\label{eq:threshold_zero}
		\E\bigg[\sum_ {\xx \in \PP_n^{(3)}} \one \{ \xx \in \mc B(\PP_n), \ b_{\xx,\PP_n} \leq b_n, \ \ell^*_{\xx,\PP_n} > 2 a_n\}\bigg] \longrightarrow 0.
	\vspace{0.75ex}\end{equation*}
	First, we
  call a cycle $\xx$ \textit{early-born}
	if
	\vspace{0.75ex} \begin{equation}
		\label{eq:early_birth}
		n b_{\xx,\PP_n}^2 \leq \log(n)^{-1/2}, 
	\vspace{0.75ex} \end{equation}
	and otherwise we call the cycle \textit{late-born}. If $\xx$ is an early-born cycle with an associated loop of size $m \geq 1$, then 
	\cite[Lemma 4.2]{BKS17} 
gives that $m \leq 2 a_n $ with high probability. However, by Lemma \ref{lem:disjoint_occurence}, then  $m \geq \ell^*_{\xx,\PP_n}$
and therefore $\ell^*_{\xx,\PP_n} \leq 2 a_n$. If $\xx$ instead is a late-born cycle, it follows from $r_\xx \leq r_n$ and (\ref{eq:early_birth})  that 
		$\ell^*_{\xx,\PP_n} \leq \log(n)^{3/4} \leq 2a_n$ for all sufficiently large $n$, which completes the proof.
\end{proof}

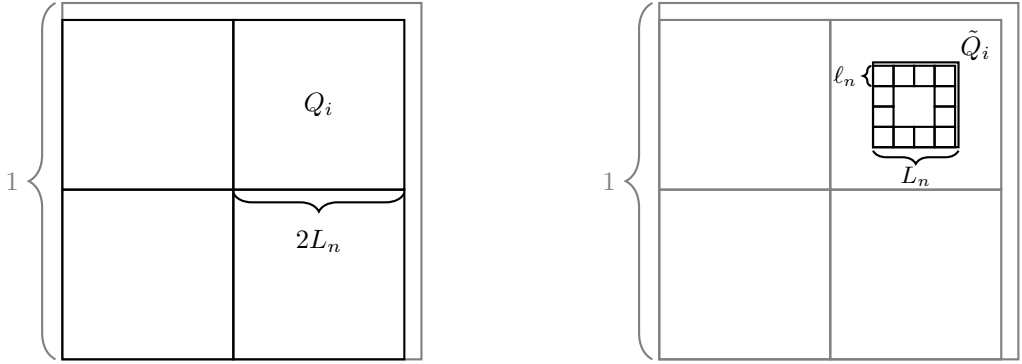
\begin{figure}[h!]
	\begin{subfigure}[t]{0.4\textwidth}
        \centering
	\begin{tikzpicture}[scale=0.45]
		\draw[thick, gray] (0, 0) rectangle (10.5, 10.5);
		\draw[decorate, decoration={brace, amplitude=10pt}, thick, gray] (- 0.2, 0) -- (-0.2 , 10.5) node[midway, left =10pt] { $1$};
	  
		\draw[thick] (0, 0) rectangle (5, 5);
		\draw[thick] (5, 0) rectangle (10, 5);
	  
		\draw[thick] (0, 5) rectangle (5, 10);
		\draw[thick] (5, 5) rectangle (10, 10);

		\node at (7.5, 7.5) { $Q_i$};
		\draw[decorate, decoration={brace, mirror, amplitude=7pt}, thick] (5, 4.9) -- (10, 4.9) node[midway, below=10pt] { $2L_n$};
	  \end{tikzpicture}
	  \end{subfigure}
	  \hspace{0.1\textwidth}
	  \begin{subfigure}[t]{0.4\textwidth}
        \centering
		\begin{tikzpicture}[scale=0.45]
			\draw[thick, gray] (0, 0) rectangle (10.5, 10.5);
			\draw[decorate, decoration={brace, amplitude=10pt}, thick, gray] (- 0.2, 0) -- (-0.2 , 10.5) node[midway, left =10pt] { $1$};
		  
			\draw[thick, gray] (0, 0) rectangle (5, 5);
			\draw[thick, gray] (5, 0) rectangle (10, 5);
		  
			\draw[thick, gray] (0, 5) rectangle (5, 10);
			\draw[thick, gray] (5, 5) rectangle (10, 10);

			\draw[thick] (6.25, 6.25) rectangle (8.75, 8.75);
			\draw[decorate, decoration={brace, mirror, amplitude=5pt}, thick] (6.25, 6.15) -- (8.75, 6.15) node[midway, below=2.5pt] { \small $L_n$};
			
			\draw[thick] (6.25, 6.25) rectangle (6.85, 6.85);
			\draw[thick] (6.85, 6.25) rectangle (7.45, 6.85);
			\draw[thick] (7.45, 6.25) rectangle (8.05, 6.85);
			\draw[thick] (8.05, 6.25) rectangle (8.65, 6.85);

			\draw[thick] (6.25, 8.05) rectangle (6.85, 8.65);
			\draw[thick] (6.85, 8.05) rectangle (7.45, 8.65);
			\draw[thick] (7.45, 8.05) rectangle (8.05, 8.65);
			\draw[thick] (8.05, 8.05) rectangle (8.65, 8.65);
			
			\draw[thick] (6.25, 6.85) rectangle (6.85, 7.45);
			\draw[thick] (6.25, 7.45) rectangle (6.85, 8.05);

			\draw[thick] (8.05, 6.85) rectangle (8.65, 7.45);
			\draw[thick] (8.05, 7.45) rectangle (8.65, 8.05);

			\node at (9.25, 9.25) { $\tilde Q_i$};
			\draw[decorate, decoration={brace, amplitude=3pt}, thick] (6.2, 8.05) -- (6.2, 8.65) node[midway, left = 1.5 pt] { \small $\ell_n$};
		  \end{tikzpicture}
	  \end{subfigure}
	  \caption{Left: Construction of the $Q_i$-squares of length $2L_n$ inside the unit cube.
	  Right: The smaller $Q_{i,j}$-squares of length $\ell_n$ inside the shrunk version $\tilde Q_i$ of $Q_i$.}
\end{figure}

\begin{proof}[Proof of Lemma \ref{lem:Ln_bound}]
	Let $\mc T_m$ denote the set of spanning trees on $\{1,\ldots,m\}$. If $\yy$ is $b_n$-connected, there is a spanning tree $T \in \mc T_m$ where each edge is at most of length $b_n$. By the union bound,
\vspace{0.75ex} $$
P\big(\mc C_{b_n}^m(\PP_n^\xx \cap W) \neq \emptyset\big) 
\leq  \frac{1}{m!} 
\E\bigg[  \sum_{\yy \in (\PP_{n}^\xx \cap W)^{(m)}} \sum_{T \in \mc T_m} 
\prod_{(i,i') \in T} \one\{\rho(y_i,y_{i'}) < b_n \} \bigg],
\vspace{0.75ex} $$
where $\rho$ is the toroidal metric (Section \ref{sec:1}). Letting $\yy_{j} = (y_1,\ldots,y_j)$ and $\yy^{j} = (y_{j+1},\ldots,y_{m})$,
then
\vspace{0.75ex} $$
P\big(\mc C_{b_n}^m(\PP_n^\xx \cap W) \neq \emptyset\big) 
\leq  \frac{1}{m!} 
\sum_{j=0}^{p} \sum_{\yy_{j} \in \delta_{\xx}^{(j)}}
\E\bigg[  \sum_{\yy^{j} \in (\PP_{n}\cap W)^{(m-j)}} \sum_{T \in \mc T_{m-j}} 
\prod_{(i,i') \in T} \one\{\rho(y_i,y_{i'})  < b_n \} \bigg].
\vspace{0.75ex} $$
Since $\vert \delta_{\xx}^{(j)} \vert \leq p!$ for any $j \leq p$, then
applying Mecke's formula,
\vspace{0.75ex} $$
P\big(\mc C_{b_n}^m(\PP_n^\xx \cap W) \neq \emptyset\big) 
\leq  \frac{p!}{m!} 
\sum_{j=0}^{p} n^{m-j}
\sum_{T \in \mc T_{m-j}} \int_{W^{m-j}} 
\prod_{(i,i') \in T} \one\{\rho(y_i,y_{i'})  < b_n \} \ \d \yy .
\vspace{0.75ex} $$
For $T \in \mc T_m$, select a vertex $i_0$ of degree 1 in $T$ with unique
neighbor $i_1$. Then 
\vspace{0.75ex} $$
P\big(\mc C_{b_n}^m(\PP_n^\xx \cap W) \neq \emptyset\big) 
\leq  \frac{p!}{m!} 
\sum_{j=0}^{p} n^{m-j} \pi b_n^2
\sum_{T \in \mc T_{m-j}} \int_{W^{m-j-1}} 
\prod_{\substack{(i,i') \in T \\ i \neq i_0, i' \neq i_1}} 
\one\{\rho(y_i,y_{i'})  < b_n \} \ \d \yy^{-i_0}, 
\vspace{0.75ex} $$
where $\yy^{-i_0} = (y_1, \ldots,y_{i_0-1},y_{i_0 + 1}, \ldots, y_{m-j})$.
Iterating this on the remaining vertices of $T$, then
\vspace{0.7ex} $$
P\big(\mc C_{b_n}^m(\PP_n^\xx \cap W) \neq \emptyset\big) 
\leq  \frac{p!}{m!} 
\sum_{j=0}^{p} n^{m-j} (\pi b_n^2)^{m-j-1} \vol(W)
\sum_{T \in \mc T_{m-j}}1.
\vspace{0.7ex} $$
Putting $j=0$ and using Cayley's formula for the number of spanning trees, 
it follows that
\vspace{0.7ex} $$
P\big(\mc C_{b_n}^m(\PP_n^\xx \cap W) \neq \emptyset\big) 
\leq p!p n \vol(W) (\pi nb_n^2)^{m-1} m^{m-2}/(m!) 
\vspace{0.7ex} $$
By Stirling's formula, then $m^{m-2}/(m!) \leq \text{e}^m$ and hence
\vspace{0.7ex} \begin{equation}
\label{eq:comp_bound}
P\big(\mc C_{b_n}^m(\PP_n^\xx \cap W) \neq \emptyset\big) 
\leq p!p n \vol(W) (\text{e} \pi n b_n^2)^{m-1}
\vspace{0.7ex} \end{equation}
Plugging the definition of $b_n$ from (\ref{eq:birthtime_bound}) into
the bound in (\ref{eq:comp_bound}) completes the proof.
\end{proof}
	\begin{proof}[Proof of Lemma \ref{lem:deathtime_bound}]
		First, introduce the intermediary point process,
		\vspace{0.75ex} $$
		\breve \xi_n^1 = \sum_{\xx \in \PP_n^{(3)}} \one\{\xx \in \mc N(\PP_n), \ r_\xx \leq r_n,\  \ell_{\xx, \PP_n} \geq \ell_{n,\alpha}\} \delta\{z_\xx\}
		\vspace{0.75ex} $$
		Let $A$ denote a Borel set in $\T^2$. Note if $r_\xx \leq r_n$, then $\xx \in \mc N(\PP_n)$ 
		if and only if $\xx \in \mc N(\PP_n \cap B(z_\xx,R_n))$ by Remark \ref{rem:cech_properties}(1). Using the triangle inequality and this observation, we see that 
		\vspace{0.75ex} \begin{equation*}
			\label{eq:string_0}
		\vert \breve \xi_n^1 (A) - \hat \xi_n^1 (A) \vert \leq  \sum_{\xx \in \mc N(\PP_n)} 
		\one\{ r_\xx \leq r_n\}\one \big\{\ell_{\xx, \PP_n} \neq \ell_{\xx, \PP_n \cap B(z_\xx, R_n)}\big\}\one_A(z_\xx),
		\vspace{0.75ex} \end{equation*}
		Moreover, it also follows from Remark \ref{rem:cech_properties} that if the lifetime increases or decreases by restricting 
		to the ball $B(z_\xx,R_n)$, then
		one point in the connected cluster associated to $\xx$ must lie in $B(z_\xx,R_n)^c$ and hence by Lemma \ref{lem:birthtime_bound} this implies that this cluster
		is of size $m_n$ or larger. Thus,
		\vspace{0.75ex} $$
		\vert \breve \xi_n^1 (A) - \hat \xi_n^1(A) \vert \leq  \sum_{\xx \in \mc N(\PP_n)} 
		\one\{ r_\xx \leq r_n\} \one\{\mc C_{b_n}^m(\PP_n) \neq \emptyset \} \one_A(z_\xx),
		\vspace{0.75ex} $$
		By Mecke's formula, Lemma \ref{lem:Ln_bound} and the definition of the total variation distance, it follows that
		\vspace{0.6ex} \begin{equation}
			\label{eq:breve_hat}
		\E[\dtv(\breve \xi_n^1,\hat \xi_n^1)] \in O(n^{-\beta}).
		\vspace{0.75ex} \end{equation}
		for any $\beta >0$. Using Remark \ref{rem:cech_properties}(1), it follows that
		\vspace{0.75ex} $$
		\vert \xi_n^1(A) - \breve \xi_n^1(A) \vert \leq \sum_{\xx \in \mc N(\PP_n)} \one\{r_\xx > r_n, \ \ell_{\xx, \PP_n}^{*}\geq \ell_{n, \alpha}\} \one_A(z_\xx)
		\leq \sum_{\xx \in \PP_n^{(3)}} \one\{\PP_{n}(B(z_\xx,r_n)) = 0 \}
		\vspace{0.75ex} $$
		By Mecke's formula and substituting $r_n$ in (\ref{eq:deathtime_bound}),
		\vspace{0.75ex} \begin{equation}
			\label{eq:breve_original}
		\E[\dtv( \xi_n^1,\breve \xi_n^1)] \leq n^3 \int  P \big(\PP^\xx_{n}(B(z_\xx,r_n)) = 0 \big) \ \d\xx = n^3\exp(-n\pi r_n^2) \in O(n^{-1/8}).
		\vspace{0.75ex} \end{equation}
		Hence, by definition of the Kantorovich-Rubinstein distance as well as equations (\ref{eq:breve_hat}) and (\ref{eq:breve_original}),
		\vspace{0.75ex} \begin{equation*}
			\label{eq:string_1}
		\dkr(\xi_n^1, \hat \xi_n^1) \leq \E[\dtv(\xi_n^1, \hat \xi_n^1)] \leq \E[\dtv( \xi_n^1,\breve \xi_n^1)] + \E[\dtv(\breve \xi_n^1,\hat \xi_n^1)] \in O(n^{-1/8}),
		\vspace{0.75ex} \end{equation*}
		which completes the proof.
		\end{proof}
		
	\bep[Proof of Lemma \ref{lem:E_2}]
	Recall we want to prove that for every $\e >0$, then
	\vspace{0.75ex} $$
	E_2 = n^6 \int_{\T^{2\cdot 3}} \int_{\T^{2\cdot 3}}
	\one \{ B(z_\xx,R_n) \cap B(z_\yy,R_n) \neq \emptyset  \} P\big(\xx \in \mc B(\PP_n^\xx)\big) 
	P\big(\yy \in \mc B(\PP_n^\yy)\big)
		 \ \d\xx \d\yy \in O(n^{-1+\e}).
	\vspace{0.75ex} $$	 
	First, note that  $B(z_\xx,R_n) \cap B(z_\yy,R_n) \neq \emptyset$
	and $r_\xx \leq r_n < R_n$
	implies that $x_1 \in B(z_\yy,2 R_n)$. Thus, 
	\vspace{0.75ex} \begin{align*} 
	E_2 & \leq  n^6\int_{\T^{2\cdot 3}}\int_{B(z_\yy,2 R_n) \times \T^{2\cdot 2}}
		P\big(\xx \in \mc B(\PP_{n}^{\xx}) \big)
		P\big(\yy \in \mc B(\PP_{n}^{\yy}) \big)
		 \ \d\xx \d\yy.
		\vspace{0.75ex} \end{align*} 
		By translation invariance of the Lebesgue measure on $\T^2$ and linearity of the center $z(\cdot)$, then
	\vspace{0.75ex} \begin{align*} 
		E_2 \leq \vol(B(0,2 R_n)) n^6 \int_{\T^{2\cdot 3}} \int_{\T^{2\cdot 2}}&
		P\big((0,\xx) \in \mc B (\PP_{n}^{(0,\xx)}) \big) P\big(\yy \in \mc B(\PP_n^\yy)\big)
		\ \d\xx \d\yy.
	\vspace{0.75ex} \end{align*} 
		By Mecke's formula and (\ref{eq:almost_alpha}), then for sufficiently large $n$,
		\vspace{0.75ex} $$
			E_2 \leq 4 \alpha^2 \vol(B(0,2 R_n)) = 16 \pi \alpha^2 \frac{\log(n)^{2}}{n},
		\vspace{0.75ex} $$
	 and using that $\log(n)^2 \in O(n^{\e})$ for any $\e > 0$, then this completes the proof.
	\enp
	 \begin{proof}[Proof of Lemma \ref{lem:E_3}]
		Recall that we want to prove that for every $\e >0$, then
		\vspace{0.75ex} $$
		E_3 = n^6 \iint
	\one \{ B(z_\xx,R_n) \cap B(z_\yy,R_n) \neq \emptyset  \} P\big(\xx \in \mc B(\PP_n^{\xx,\yy}), \yy \in \mc B(\PP_n^{\xx,\yy}) \big)
		  \ \d\xx \d\yy \in O(n^{-1/36 + \e}).
		\vspace{0.75ex} $$
		First, by Lemma \ref{lem:disjoint_occurence} and the continuous BK inequality
		(\cite[Theorem 2.1]{HHLM19}),
		\vspace{0.75ex} \begin{align*}
			E_3 &\leq n^6 \iint \one_{S_n}(\xx,\yy)
			P\big( \xx \in \mc B(\PP_{n}^{\xx,\yy}) \big) 
			 P\big(\mc C_{b_n}^{\lceil \ell_{n,\alpha}/2 \rceil}(\PP_n^{\xx,\yy} \cap B(z_\xx, 2R_n)) \neq \emptyset \big)
			 \ \d\xx \d\yy,
		\vspace{0.75ex} \end{align*} 
		where $S_n = \{ B(z_\xx,R_n) \cap B(z_\yy,R_n) \neq \emptyset  \} $. By Lemma \ref{lem:Ln_bound}, there exists a constant $C>0$ such that
		\vspace{0.75ex} \begin{align*}
			E_3 &\leq C U_{n,\alpha} n^6 \log(n)^2
			\underset{S_n}{\iint} P\big(  \xx \in \mc B(\PP_{n}^{\xx,\yy}),
			 \ \ell^*_{\xx,\PP_{n}^{\xx,\yy}}\ge \ell_{n,\alpha} \big) 
			 \ \d\xx \d\yy,
		\vspace{0.75ex} \end{align*}
		where for every $\e > 0$,
		\vspace{0.75ex} $$
			U_{n,\alpha} = 
			\bigg(\frac{\log(\log(n))^2}{\log(n)}\bigg)^{\lceil \ell_{n,\alpha} /2\rceil -1} \in O\big(n^{-1/36 + \e }\big),
		\vspace{0.75ex} $$
		using the lower bound for $\ell_{n,\alpha}$ provided in Lemma \ref{prop:asymptotics}. By Mecke's formula,
		\vspace{0.75ex} \begin{equation}
			\label{eq:E3_pitstop}
			E_3 \leq C U_{n,\alpha} n^3
			\int \E\bigg[ \sum_{\yy \in (\PP_n \cap B(z_\xx, 2R_n))^{(3)}}
			\one\big\{ \xx \in \mc B(\PP_{n}^{\xx}),
			 \ \ell^*_{\xx,\PP_{n}^{\xx}}\ge \ell_{n,\alpha} \big\}
			 \bigg] \ \d\xx
		\vspace{0.75ex} \end{equation} 
		Now, define the events 
		\vspace{0.25ex} \begin{equation*}
			\label{eq:A_events}
		A_{n}^1 = \big\{\PP_{n}(B(z_\xx, 2R_n)) \leq 2 \E[\PP_{n}(B(z_\xx, 2R_n))] \big\},
		\vspace{0.75ex} \end{equation*}
		\begin{equation*}
			\label{eq:A_events_2}
		A_{n}^2 = \big\{\PP_{n}(B(z_\xx, 2R_n)) > 2 \E[\PP_{n}(B(z_\xx, 2R_n))] \big\},
		\vspace{0.75ex} \end{equation*}
		and define for $i\in \{1,2\}$, the expression
		\vspace{0.75ex} \begin{equation}
			\label{eq:new_E3_terms}
		E_{3,i} = C U_{n,\alpha} n^3
		\int \E\bigg[ \sum_{\yy \in (\PP_n \cap B(z_\xx, 2R_n))^{(3)}}
		\one\big\{ \xx \in \mc B(\PP_{n}^{\xx}),
		 \ \ell^*_{\xx,\PP_{n}^{\xx}}\ge \ell_{n,\alpha} \big\} \one_{A_n^i}(\xx)
		 \bigg] \ \d\xx.
		\vspace{0.4ex} \end{equation}
		First, the event $A_n^1$ and assumption $P'$ imply that
		\vspace{0.75ex} $$
		\vert (\PP_{n}\cap B(z_\xx, 2R_n))^{(3)} \vert \leq 8 \E[\PP_{n}(B(z_\xx, 2R_n))]^{3} 
		= 512 \pi^3 \log(n)^{6}.
		\vspace{0.75ex}$$
		Combining this with (\ref{eq:new_E3_terms}), Remark \ref{rem:unbounded_extension}(2) and (\ref{eq:almost_alpha}), it now follows for any $\e > 0$ that 
		\vspace{0.75ex} $$
		E_{3,1} \leq 1024 \pi^3 \alpha C \log(n)^{8}
		U_{n,\alpha} \in  O(n^{-1/36+\e}).
		\vspace{0.75ex} $$
		Next, using a Poisson Chernoff-type bound (\cite[Theorem 2.1 \& Remark 2.6]{Janson}),
		\vspace{0.75ex} \begin{equation}
		\label{eq:chernoff_bound}
		P(A_n^2) \leq \exp(-3\E[\PP_{n}(B(z_\xx, 2R_n))]/8) \leq \exp(-\log(n)^{2}).
		\vspace{0.75ex} \end{equation}
		Thus by combining (\ref{eq:chernoff_bound}) and
		the Cauchy-Schwarz inequality, it follows that 
		\vspace{0.75ex} \begin{equation*}
			E_{3,2} \leq C U_{n,\alpha}
			n^{3}\exp(-\log(n)^{2}/2)  \int 
			\E\bigg[\bigg(\sum_{\yy \in (\PP_{n}\cap B(z_\xx, 2R_n))^{(3)}}
			1 \bigg)^2\bigg]^{1/2}  \d\xx.
			\vspace{0.75ex} \end{equation*}
		Since $\PP_{n}(B(z_\xx, 2R_n)) \sim \text{Pois}(n\cdot \vol(B(z_\xx, 2R_n)))$, then
		$\E[\PP_{n}(B(z_\xx, 2R_n))^{6}] \leq n^6 \exp(18/n) $ and
		\vspace{0.75ex}$$
		E_{3,2} \leq CU_{n,\alpha} n^9 \exp(-\log(n)^2/2 - 18/n) \in O(n^{-1/36 + \e }),
		\vspace{0.75ex} $$ 
		which then completes the proof.
	\end{proof}

	\begin{proof}[Proof of Lemma \ref{lem:E_4}]
		
		First as $\xx$ and $(\xx_I,\yy)$ share at least one point, then 
		if $r_\xx,r_\yy \leq r_n \leq R_n$, it follows that $B(z_\xx,R_n) \cap B(z_\yy,R_n) \neq \emptyset$
		and we proceed as in the proof of Lemma \ref{lem:E_3}:
		By Lemma \ref{lem:disjoint_occurence}, the continuous BK inequality, Lemma \ref{lem:Ln_bound}
		and Mecke's formula, it follows that
		\vspace{0.75ex} \begin{equation}
			\label{eq:E4_pitstop}
			E_4^I \leq C U_{n,\alpha} \log(n)^2
			n^{3}
			\int \E\bigg[ \sum_{\yy \in (\PP_n \cap B_n(\xx))^{(3-\abs{I})}} P\big(  \xx \in \mc B(\PP_{n}^{\xx}),
			 \ \ell^*_{\xx,\PP_{n}^{\xx}}\ge \ell_{n,\alpha} \big) \bigg]
			 \ \d\xx.
		\vspace{0.75ex} \end{equation}
		Note that in (\ref{eq:E4_pitstop}), we obtain a upper bound by putting $\abs{I} = 0$ 
		and therefore (\ref{eq:E4_pitstop}) is further bounded by the bound in (\ref{eq:E3_pitstop}).
		Thus, we merely repeat the rest of the proof of Lemma \ref{lem:E_3}. 
	\end{proof}

%
%
%
%

\section{Convergence in the $m$-sparse regime}
\label{sec:4}
This section is dedicated to proving the three Poisson approximation results 
in the $m$-sparse setting. Recall that this setting is defined by
choosing a deathtime bound $r_n > 0$ such that:

\begin{assumptionM*}
	\label{ass:S}
	As $n \to \infty$, then $\rho_{n, m} = n(nr_n^d)^{m-1} \longrightarrow  \infty$
	and $\rho_{n, m+1} = n(nr_n^d)^{m} \longrightarrow  0$.
\end{assumptionM*}

When convenient, we omit the cluster size $m$ in the notation. Let
\vspace{0.75ex} $$
\varphi_n(\xx, \yy)=\bigg(z_\xx,\frac{\lmax^\circ - \hat \ell_{\xx,\yy}^\circ}{u_{n,\alpha}},\frac{r_n - r_\xx}{u_{n,\alpha}}\bigg).
\vspace{0.75ex} $$
Then, the point processes $\xi_n^2$, $\xi_n^3$, and $\xi_n^4$ in Theorems \ref{thm:1}, \ref{thm:2}, and \ref{thm:3} can be rewritten as
\vspace{0.75ex} \begin{equation}
	\label{eq:xi_p_def}
	\begin{aligned}
	\xi_n^{4}(A \times B \times C) &=\sum_{\xx \in \PP_{n\kappa}^{(k)}} 
	\one\{\xx \in \mc N(\PP_{n\kappa})\} \delta\{\varphi_n(\xx,\PP_{n\kappa})\}(A \times B \times C), \\[0.5ex]
	\xi_n^{3}(A \times B) &= \xi_n^{4}(A \times B \times  [0,\infty)), \\[0.5ex]
	\xi_n^{2}(A) &= \xi_n^{4}(A \times [0,1] \times  [0,\infty)),
	\end{aligned}
\vspace{0.75ex} \end{equation}
where $A \times B \times C \in \mc B(\R^d \times [0,\infty) \times [0,\infty))$. The rest of the section is organized as follows: In Section \ref{sec:4.1}, we introduce a Poisson U-statistic approximation of (\ref{eq:xi_p_def}). 
In Section \ref{sec:4.2}, we present the overall proof strategy via a series of lemmas,
and in Section \ref{sec:4.3}, we prove all three theorems using these lemmas.  
Finally, in Section \ref{sec:4.4}, we prove the lemmas in Sections \ref{sec:4.1} and \ref{sec:4.2}.




\vspace{0.75ex}
	\subsection{Large-lifetime cycles in $m$-clusters}
\label{sec:4.1}
All references to assumptions $F$, $P$, $T$, $U$, $G$, and $H$ 
are to these as stated in Theorems \ref{thm:1}, \ref{thm:2}, and \ref{thm:3}. 
For a Lebesgue density $\lambda$ and Poisson point process $\PP_{n\lambda}$, with 
intensity $n \cdot \lambda$, we introduce
\vspace{0.75ex} \begin{equation}
	\label{eq:xi_p_hat_def}
	\begin{aligned}
	\hat \xi_{n,\lambda}^{4}(D) &=\frac{1}{m!} \sum_{\yy  \in \PP_{n\lambda}^{(m)}}
	\one\{\yy \in C_{r_n}\}  \sum_{\xx \in \yy^{(k)}} 
	\one\{\xx \in \mc N(\yy)\} \delta\{\varphi_n(\xx,\yy)\}(A \times B \times C), \\[0.5ex]
	\hat \xi_{n,\lambda}^{3}(A \times B) &= \hat \xi_{n,\lambda}^{4}(A \times B \times  [0,\infty)), \\[0.5ex]
	\hat \xi_{n,\lambda}^{2}(A) &= \hat \xi_{n,\lambda}^{4}(A \times [0,1] \times  [0,\infty)),
	\end{aligned}
\vspace{0.75ex} \end{equation}
where $A \times B \times C \in \mc B(\R^d \times [0,\infty) \times [0,\infty))$. In other words, (\ref{eq:xi_p_hat_def}) represents the (marked) location of large-lifetime cycles \textit{inside} $m$-clusters, which will serve as an approximation of the 
corresponding point processes in (\ref{eq:xi_p_def}) when $\lambda = \kappa$. Also note that all three point processes in (\ref{eq:xi_p_hat_def}) 
are Poisson U-statistics (Section 1). In order to derive some properties of these U-statistics, 
we first note that features with a very large lifetime must be $1$-connected 
and that the maximal lifetime $\lmax^\circ(\cdot)$ is strictly increasing in $m$.

\begin{lemma}
	\label{obs:lmaxincreasing}
	Let $\yy \in \R^{dm}$ and $\xx \in \mc N(\yy)$ with $r_\xx \leq 1$. Under assumption $F$, 
	$\ell_{\xx,\yy} > \lmax^\circ(m-1)$ implies that $\yy \in C_1(m)$. 
	Additionally, under assumption $U$, it holds that $\lmax^\circ(m) > \lmax^\circ(m-1)$.
\end{lemma}

Due to the fact that $\hat \xi_{n,\lambda}^2$ and $\hat \xi_{n,\lambda}^3$ are just $\hat \xi_{n,\lambda}^4$ restricted to certain product sets, it will, in most cases, 
be sufficient to state properties for $\hat \xi_{n,\lambda}^4$, and these properties will automatically be inherited by $\hat \xi_{n,\lambda}^2$ and $\hat \xi_{n,\lambda}^3$. 
We now use Lemma \ref{obs:lmaxincreasing} to establish that the approximation $\hat \xi_{n,\kappa}^4$ coincides with the original point process $\xi_n^4$ on sets of size $m$ or smaller.

\begin{lemma}
	\label{lem:equalmeasures}
	Under assumptions $F$, $M$, $P$, $T$, and $U$, $\hat \xi^p_{n,\kappa} \cap A = \xi^p_n \cap A$ for any $\vert A \vert \leq m$.
\end{lemma}

Furthermore, it will be useful to have an expression for the intensity of (\ref{eq:xi_p_hat_def}). For $r \geq 0$, let
\vspace{0.25ex} \begin{equation}
 \label{eq:F_set}
 C_r^{-} = \big\{(y_1,\ldots,y_m) \in C_s \colon (y_1,\ldots, y_k)\in \mc N\big((y_1,\ldots,y_m)\big) \},
 \vspace{0.75ex} \end{equation}
i.e., the $m$-clusters at filtration time $r$ where the first $k$ points form a negative simplex.

\begin{lemma}[Intensity formula]
	\label{lem:expansion}
	Under assumption $F$, for any $D \in \mc B(\R^d \times [0,\infty)^2)$,
	\vspace{0.75ex} \begin{align*}
		\E \big[ \hat{\xi}^{4}_{n,\lambda }(D)\big] = \rho_{n,m}\frac{\binom{m}{k}}{m!} 
		& \int_{\R^d} \ldots  \int_{\R^d} \one\{(0,y_2,\ldots,y_m)\in C_1^- \} \lambda(y_1) \prod_{j=2}^{m} \lambda(r_n y_j+y_1) \times \\
		& \quad \times \one\{ \varphi_n\big(r_n (0,y_2,\ldots,y_k) + y_1 , r_n (0,y_2,\ldots,y_m) + y_1 \big)\in D\}
		\ \d y_1 \ldots \d y_m.
	\vspace{0.75ex} \end{align*}
	Additionally, under assumption $T$ and $\lambda \equiv 1$, 
	\vspace{0.75ex} \begin{equation}
		\label{eq:alpha_formula}
		\E \big[ \hat{\xi}^{2}_{n,1}([0,1]^d)\big]
		 = \rho_{n,m}\frac{\binom{m}{k}}{m!} \int_{\R^{d(m-1)}} 
		 \one\big\{(0,\yy)\in C_1^-, \ \ell^\circ_{(0,\yy)_k,(0,\yy)}\ge \lmax^\circ - u_{n,\alpha }, \ r_{(0,\yy)_k} \leq r_n \big\}
		  \d\yy = \alpha.
	\vspace{0.75ex} \end{equation}
\end{lemma}

Lemma \ref{lem:expansion} states that the choice of 
thresholds in assumption $T$ implies that the number of
large-lifetimes inside 1-clusters of size $m$ in the unit cube remains 
constant as $n$ increases. This facilitates control over
the intensity of the locations of large-lifetime cycles in the Poisson approximation.
Finally, we want to prove that in our inhomogeneous Poisson model, the expected number of $m$- and $m+1$-clusters is still at most proportional to $\rho_{n,m}$ and $\rho_{n,m+1}$, respectively.

\begin{lemma}[Bound on clusters]
	\label{lem:pruning}
	Let $j \geq 2$ and define $\rho_{n,j} = n(nr_n^d)^{j-1}$. Under assumption $P$,
	\vspace{0.75ex} $$
	\E\bigg[ \sum_{\xx \in \PP_{n\kappa}^{(j)}} \one\{\xx \in C_{r_n}(j) \} \bigg] \leq c j^{j-2}\tilde c^{j-1} \rho_{n,j},
	\vspace{0.75ex} $$
	where $c = \int \kappa(y) \ \d y$ and $\tilde c = \vol(B(0,1)) \sup_{y \in W} \kappa(y)$.
\end{lemma}

The proof of Lemma \ref{lem:pruning} relies on the same pruning-type argument as in the proof of Lemma \ref{lem:Ln_bound}.




\vspace{0.75ex}
	\subsection{Overall Poisson strategy}
\label{sec:4.2}
To prove Theorems \ref{thm:1} - \ref{thm:3}, 
we rely on a series of lemmas that we postpone proving until Section \ref{sec:4.4}.
Fix $p \in \{2,3,4\}$ throughout the section. First, Lemma \ref{lem:dkrhat} 
justifies calling $\hat \xi_{n,\kappa}^p$ an approximation of $\xi_n^p$ 
and relies on Lemmas \ref{lem:equalmeasures} and \ref{lem:pruning}.
Recall that $\dkr$ denotes the Kantorovich-Rubinstein distance
and that $\rho_{n, m+1} = n(nr_n^d)^m$. For convenience, we write $\hat \xi_n^p$ instead of $\hat \xi_{n,\kappa}^p$.

\begin{lemma}[$\hat\xi^p_n \approx \xi_n^{p}$]
	\label{lem:dkrhat}
	Under assumptions $F$, $M$, $P$, $T$, and $U$, it holds that
\vspace{0.75ex} $$
		\dkr(\xi_n^{p} ,\hat\xi_n^{p}) \in O(\rho_{n, m+1}).
\vspace{0.75ex} $$
\end{lemma}

Next, we apply the Poisson approximation result for U-statistics in \cite[Theorem 3.1]{decreusefond} to the approximation $\hat \xi_n^p$. 
Recall that $\dtv$ denotes the total variation distance.

\begin{lemma}[$\hat\xi^p_n \approx \text{Poisson}$]
	\label{lem:totalvariationestimate}
	Let $\zeta$ denote a Poisson point process on $\R^d$ with finite intensity. Under assumptions 
	$F$, $M$, $P$, $T$, and $U$, it holds for any Borel set $B$ that
\vspace{0.75ex} $$
		\dkr(\hat\xi^{p}_n \cap B, \zeta) 
		\in O\big(\dtv\big(\E[\hat\xi^p_n \cap B],\E[\zeta]\big) 
		+ \rho_{n,m+1} \big).
\vspace{0.75ex} $$
\end{lemma}

As assumption $M$ ensures $nr_n^d \to 0$, it just remains to find a suitable 
$\zeta$, where the total variation distance of the intensity measures vanishes. 
To that end, recall that $\yy_k = (y_1,\ldots,y_k)$ and let
\vspace{0.75ex} $$
\hat \varphi_n(\yy)=\bigg(y_1,\frac{\lmax^\circ - \hat \ell_{\yy_k,\yy}^\circ}{u_{n,\alpha}},\frac{r_n - r_{\yy_k}}{u_{n,\alpha}}\bigg),
\vspace{0.75ex} $$
i.e., the marked first coordinate $y_1$, rather than the marked center $z_{\yy_k}$. Inspired by the expression in Lemma \ref{lem:expansion} with $\lambda = \kappa$, 
we now define the approximating measure
\vspace{0.75ex} \begin{equation}
	\label{eq:hatE}
	\hat\Theta^4_n(D) = \rho_{n, m}\frac{\binom{m}{k}}{m!} \int_{\R^d} \ldots \int_{\R^d} \one\{ (0,y_2,\ldots,y_m) \in C_1^-, \ \hat \varphi_n((0,y_2,\ldots,y_m)) \in D \} 
	\kappa(y_1)^{m} \ \d y_1 \ldots \d y_m,
\vspace{0.75ex} \end{equation}
where $D \in \mc B(\R^d \times [0,\infty)^2)$. In short, Lemma \ref{lem:removekmix} says that (\ref{eq:hatE}) is a good approximation of the intensity 
measure of $\hat\xi^4_n$ and is proven using Lemma \ref{lem:expansion}.

\begin{lemma}[$\hat\Theta^4_n \approx \E{[}\hat\xi_n^4{]}$]
	\label{lem:removekmix}
	Under assumptions $M$, $P$, $T$, and $U$, it holds for any compact sets $K_1,K_2 \subseteq [0,\infty)$ and $B=\R^d \times K_1 \times K_2$ that
	 \vspace{0.75ex} $$
	 \dtv\big(\E[\hat\xi_n^4\cap B ] ,\hat\Theta^4_n \cap B \big) \in O(r_n).
	 \vspace{0.75ex} $$
\end{lemma}

From equation (\ref{eq:hatE}) as well as $A \in \mc B(\R^d)$ and $B \in  \mc B(\R^d\times[0,\infty))$, define the restricted measures:
\vspace{0.75ex} \begin{equation}
\label{eq:hat_exp_1_2}
\hat\Theta^2_n(A)
= \hat\Theta^4_n(A\times [0,1] \times [0, \infty)) \quad \text{ and } \quad \hat\Theta^3_n(B) 
= \hat\Theta^4_n(B\times [0, \infty)).
\vspace{0.75ex} \end{equation}

By Lemma \ref{lem:removekmix}, it also 
follows that $\hat\Theta^2_n $ and $\hat\Theta^3_n $ are 
good approximations of $\E[\hat\xi^2_n]$ and $\E[\hat\xi^3_n]$, respectively. 
What remains is to bound the distance between 
$\hat\Theta^p_n $ and $\E[\zeta]$, which
will be done individually in the proofs of Theorems \ref{thm:1} - \ref{thm:3}. To that end, 
we rely on Taylor expansions of order $q$, where $q$ is as defined in assumptions $G$ or $H$, and the Mean 
Value Theorem to obtain two final auxiliary results.

\begin{lemma}[Approximation of $g$]
	\label{lem:approximation_g}
	Let $a>0$. Under assumptions $F$, $M$, $T$, and $G$, 
		\vspace{0.75ex} $$
		\sup_{u \in [0,a]} \left\vert \rho_{n,m}\frac{\binom{m}{k}}{m!}  u_{n,\alpha} g^{(1)}(u_{n,\alpha} u)
		 - \alpha qu^{q-1}\right\vert \in O(u_{n,\alpha}).
		\vspace{0.75ex} $$
\end{lemma}

\begin{lemma}[Approximation of $h$]
	\label{lem:approximation_h}
	Let $a>0$. Under assumptions $F$, $M$, $T$, and $H$, 
	\vspace{0.75ex} $$
	\sup_{u,v \in [0,a]} 
	\bigg\vert \rho_{n,m}\frac{\binom{m}{k}}{m!}  u_{n,\alpha} h^{(1,1)}(u_{n,\alpha} u, v) - \alpha q u^{q-1}\frac{\tilde h^{(q,1)}(0,v)}{h^{(q,0)}(0,1)} \bigg \vert
	 \in O(u_{n,\alpha}).
	\vspace{0.75ex} $$
\end{lemma}




\vspace{0.75ex}
	\subsection{Proof of Theorems \ref{thm:1} to \ref{thm:3}}
\label{sec:4.3}

\bep[Proof of Theorem \ref{thm:1}]
Recall that we want to prove that $
\dkr(\xi_n^{2}, \zeta^{2})\in O(\r_{n, m+1}).
$	
By first applying Lemmas \ref{lem:dkrhat} and \ref{lem:totalvariationestimate},
it follows that
\vspace{0.75ex} $$
\dkr(\xi_n^2 ,\ze^2) \le \dkr(\xi_n^2 , 
\hat{\xi}_n^2) + \dkr(\hat{\xi}_n^2 , \ze^2 ) 
\in O\big(\r_{n, m+1}+ \rho_{n,m+1} +\dtv(\E[\hat\xi^2_n],\E[\zeta^2]) \big).
\vspace{0.75ex} $$
As $ r_n \leq \rho_{n,m+1}$ for all sufficiently large $n$,
it suffices to show that $\dtv(\E[\hat\xi^2_n],\E[\zeta^2])\in
O(r_n)$. Since
\vspace{0.75ex} $$
\frac{\lmax^\circ - \ell^\circ_{(0,y_2,\ldots,y_k),(0,y_2,\ldots,y_m)}}{u_{n,\alpha} } \in [0,1] 
 \iff \ell^\circ_{(0,y_2,\ldots,y_k),(0,y_2,\ldots,y_m)} \geq  \lmax^\circ - u_{n,\alpha},
\vspace{0.75ex} $$
it follows that
\begin{align*}
\hat\Theta_n^2(B) 
&= \rho_{n,m}\frac{\binom{m}{k}}{m!}  \int_{B} 
\int_{\R^d} \ldots \int_{\R^d} \one \{ (0,y_2,\ldots,y_m) \in C_1^-, \ r_{(0,y_2,\ldots,y_k)} \leq r_n \} \times \\
& \quad \quad \quad \quad  \quad \quad \quad \quad  \quad  \quad  \quad  \times
 \one\{\ell^\circ_{(0,y_2,\ldots,y_k),(0,y_2,\ldots,y_m)} \geq  \lmax^\circ - u_{n,\alpha}\}\k(y_1)^{m} 
\ \d y_m \ldots \d y_2 \d y_1.
\end{align*}
By combining this result with Lemma \ref{lem:expansion}, it follows for every Borel set $B$,
\vspace{0.75ex} $$
\hat\Theta_n^2(B)  = \a \int_{B} \k(y_1)^{m}\ \d y_1 = \E[\zeta^2(B)],
\vspace{0.75ex} $$
and therefore invoking Lemma \ref{lem:removekmix} completes the proof.
\enp

\bep[Proof of Theorem \ref{thm:2}]
Recall that for any compact $K \subseteq [0,\infty)$, we want to prove that
\vspace{0.75ex} $$
\dkr\big(\xi_n^3 \cap (\R^d \times K) , \zeta^3 \cap (\R^d \times K) \big)
 \in O\big(\r_{n, m+1}+ \r_{n, m}^{-1/q} \big). 
\vspace{0.75ex} $$
Similar to the proof of Theorem \ref{thm:1}, by Lemmas \ref{lem:dkrhat} and
\ref{lem:totalvariationestimate} with $B=\R^d \times K$,
\begin{align*}
\dkr(\xi_n^3\cap B,\ze^3\cap B) 
& \in O\big(\r_{n, m+1}+ \rho_{n,m+1} +\dtv(\E[\hat\xi^3_n\cap B],\E[\zeta^3\cap B]) \big).
\end{align*}
Further by introducing (\ref{eq:hat_exp_1_2}) via the triangle inequality,
\vspace{0.75ex} $$
\dtv(\E[\hat\xi^3_n\cap B],\E[\zeta^3\cap B] \big) 
\le \dtv(\E[\hat\xi^3_n\cap B],\hat\Theta^3_n\cap B \big)+ 
\dtv(\hat\Theta^3_n\cap B ,\E[\zeta^3\cap B)] \big),
\vspace{0.75ex} $$
then by
 Lemma \ref{lem:removekmix} it suffices to show
$\dtv\big(\hat\Theta^3_n(\cdot \cap B) ,\E[\zeta^3(\cdot \cap B)] \big) \in O(\rho_{n,m}^{-1/q})$.
This is also where the proof diverges from the proof of Theorem \ref{thm:1}:
Let $D \su \R^d$ be Borel and $t \in C$. By (\ref{eq:gdef_int}),
\vspace{0.75ex} $$
\hat\Theta^3_n(D \times [0,t])= \rho_{n,m} \frac{\binom{m}{k}}{m!}  \int_{D} g(u_{n,\alpha} t)
\kappa(y_1)^{m} \ \d y_1.
\vspace{0.75ex} $$
As $g(0)=0$ by assumption $G$, then by the
Fundamental Theorem of Calculus,
\vspace{0.75ex} $$
\hat\Theta^3_n(D \times [0,t])= \rho_{n,m} \frac{\binom{m}{k}}{m!}  \int_{D} \int_0^t u_{n,\alpha} g^{(1)}(u_{n,\alpha} u)  \kappa(y_1)^{m} \ \d u \d y_1,
\vspace{0.75ex} $$
By Dynkin's $\pi$-$\lambda$ theorem, then for any Borel set $A \su \R^d \times K$, 
\vspace{0.75ex} \begin{equation*}
\label{eq:comp_M2}
\hat\Theta^3_n(A)=\rho_{n,m} \frac{\binom{m}{k}}{m!}  \int_{A} u_{n,\alpha} g^{(1)}(u_{n,\alpha} u)  \kappa(y_1)^{m}  \  \d (y_1,u).
\vspace{0.75ex} \end{equation*}
By assumption $P$, then $\kappa$ 
is bounded and integrable. Thus there exists $c>0$ such that in combination with Lemma \ref{lem:approximation_g}, we see that
\vspace{0.75ex} \begin{equation*}
\label{eq:bounddiff2}
\abs{\hat\Theta^3_n(A) - \E[\zeta^3(A)]} 
\leq c
\sup_{u \in [0,\sup K]} \bigg\vert \rho_{n,m}\frac{\binom{m}{k}}{m!}  u_{n,\alpha} g^{(1)}(u_{n,\alpha} u) - \alpha qu^{q-1}\bigg\vert  \in O(u_{n,\alpha}).
\vspace{0.75ex} \end{equation*}
Taking the supremum over Borel sets $A \su \R^d \times K$ and using Remark \ref{rem:sparse_extension}(5) completes the proof.
\enp

\bep[Proof of Theorem \ref{thm:3}]
Recall that for any compact $K_1, K_2 \subseteq [0,\infty)$, we want to prove that
\vspace{0.75ex} $$
\dkr\big(\xi_n^4\cap (\R^d \times K_1 \times K_2), \zeta^4\cap (\R^d \times K_1 \times K_2) \big)
 \in O\big(\r_{n, m+1}+ \r_{n, m}^{-1/q} \big).
\vspace{0.75ex} $$
By Lemmas \ref{lem:dkrhat}, 
\ref{lem:totalvariationestimate} and \ref{lem:removekmix} with $B=\R^d \times K_1 \times K_2$
it (again) suffices to show that 
\vspace{0.75ex} $$
\dtv\big(\hat\Theta^4_n \cap B ,\E[\zeta^4 \cap B] \big) 
\in O(\rho_{n,m}^{-1/q}).
\vspace{0.75ex} $$
Continuing the strategy of the proof of Theorem \ref{thm:2},
let $D \su \R^d$ denote a Borel set and take $s \in C_1$, $st \in C_2$.
By equations (\ref{eq:hatE}), (\ref{eq:hdef_int}) and the fact that $h(u_{n,\alpha} s,u_{n,\alpha} st)=\tilde h(u_{n,\alpha} s,t)$, then
\vspace{0.75ex} $$
\hat\Theta^4_n(D\times [0,s] \times [0,st])= \rho_{n,m} \frac{\binom{m}{k}}{m!}  \int_{D} \tilde h(u_{n,\alpha} s, t) \kappa(y_1)^{m} \ \d y_1.
\vspace{0.75ex} $$
By Assumption $H$, the Fundamental Theorem of Calculus (twice) and Dynkin's $\pi$-$\lambda$ theorem,
\vspace{0.75ex} $$
\hat\Theta^4_n(A)= \rho_{n,m} \frac{\binom{m}{k}}{m!}  \int_{A} u_{n,\alpha} \tilde h^{(1,1)}(u_{n,\alpha} u, v)  \kappa(y_1)^{m} \ \d(u,v,y_1)
\vspace{0.75ex} $$
for any Borel set $A \subseteq \R^d \times K_1 \times K_2$. By assumption $P$ and Lemma \ref{lem:approximation_h}, 
there is a $c>0$ such that
\vspace{0.75ex} $$
\abs{\hat\Theta^4_n(A) - \E[\zeta^4(A)]}
\leq c \sup_{u,v \in [0,\sup (K_1 \cup K_2)]} \left\vert \rho_{n,m}\frac{\binom{m}{k}}{m!}  u_{n,\alpha} 
\tilde h^{(1,1)}(u_{n,\alpha} u, v) - \alpha q u^{q-1}\right\vert \in O(u_{n,\alpha}).
\vspace{0.75ex} $$
Invoking Remark \ref{rem:sparse_extension}(5) and taking the supremum over all
$A \subseteq \R^d \times K_1 \times K_2$ completes the proof.
\enp




\vspace{0.75ex}
	\subsection{Proof of Lemmas \ref{obs:lmaxincreasing} to \ref{lem:approximation_h}}
\label{sec:4.4}
\bep[Proof of Lemma \ref{obs:lmaxincreasing}] 
The statement consists of two claims which we prove separately.

\noindent \textit{Connected:} Let $\yy' \subseteq \yy$ be the $1$-connected component containing $\xx$.
By assumption $F$, we have $\ell^\circ_{\xx,\yy'} = \ell^\circ_{\xx,\yy} > \lmax^\circ(m-1)$, 
and by the definition of $\lmax^\circ(m-1)$, $\yy'$ contains $m$ points in $\R^d$.
Thus, we conclude $\yy' = \yy$, and subsequently, $\yy$ is 1-connected. 

\noindent \textit{Monotonicity:} Since $\lmax^\circ(m)\ge \lmax^\circ(m-1)$, suppose towards a contradiction that $\lmax^\circ(m)
= \lmax^\circ(m-1)$. Choose $\e>0$ in accordance with assumption $U$ and
$\yy' \in \R^{d(m-1)}$ and $\xx \in \mc N(\yy')$ with $r_\xx \leq 1$ and
$\ell_{x,\yy'} \geq \lmax^\circ(m-1)-\e = \lmax^\circ(m) - \e$. 
Place $y_0$ at a distance of 2 from $\yy'$.
Hence $\xx \in \mc N((y_0,\yy'))$ with $r_\xx \leq 1$ and $\ell_{\xx,\yy'}, \ell_{\xx,(y_0,\yy')} \geq \lmax^\circ(m) - \e$,
which contradicts assumption $U$.
\enp

\begin{proof}[Proof of Lemma \ref{lem:equalmeasures}]
By definition, the two measures are equal when $j = m$, so consider the case where 
$j < m$. Then $A^{(m)}= \emptyset$ by definition and therefore
$
\hat \xi^p_n=0.
$
Assume, for the sake of contradiction, that there is a $(k-2)$-dimensional large-lifetime cycle $\xx$ in $A$. Choose $\e
> 0$ in accordance with assumption $U$ and Lemma \ref{obs:lmaxincreasing}, 
as well as $n_0 \geq 1$ in accordance with Lemma \ref{lem:properties_of_g}, such
that 
\vspace{0.75ex} \begin{equation}
\label{eq:eps_and_N}
\inf_{n \ge n_0}(\lmax^\circ(m) - u_{n,\alpha})\ge \lmax^\circ(m) - \eps > \lmax^\circ(j).
\vspace{0.75ex} \end{equation}
By (\ref{eq:eps_and_N}), 
$
\ell^\circ_{\xx,A} > \lmax^\circ(j),
$
which is clearly a contradiction.
\end{proof}

\begin{proof}[Proof of Lemma \ref{lem:expansion}] 
The statement consists of two parts which we prove separately.

\noindent \textit{General $\lambda$:} Recall we want to prove that for any Borel set $D$,
\vspace{0.75ex} \begin{equation} 
\label{eq:general_lambda}
\begin{aligned}
\E \big[ \hat{\xi}^{4}_{n,\lambda }(D)\big] & = \rho_{n,m}\frac{\binom{m}{k}}{m!} 
\int \ldots  \int \one\{(0,y_2,\ldots,y_m)\in C_1^- \} \lambda(y_1) \prod_{j=2}^{m} \lambda(r_n y_j+y_1) \times \\
& \quad \quad \times \one\{ \varphi_n\big(r_n (0,y_2,\ldots,y_k) + y_1 , r_n (0,y_2,\ldots,y_m) + y_1 \big)\in D\}
\ \d y_1 \ldots \d y_m.
\end{aligned}
\vspace{0.75ex} \end{equation} 
First, introduce the set $\mc I$ and function $\pi_I$ for $I \in \mc I$ as
\vspace{0.75ex} $$
\mc I =\{I = (i_1, \ldots,i_{k}): 1 \leq i_1 < \ldots < i_{k} \leq m\} \quad \text{and} \quad \pi_I(y_1,\ldots,y_{m})=(y_{i_1},\ldots,y_{i_{k}}).
\vspace{0.75ex} $$
For $I \in \mc I$ and $r>0$, introduce also the set
$
C_r^I = \{ \yy \in C_r \colon \ \pi_I(\yy) \in \mc N(\yy) \}.
$
By Mecke's formula,
\vspace{0.75ex} \begin{align*}
\E \big[ \hat{\xi}^{4}_{n,\lambda }(D)\big]= n^m \f{1}{m!} \sum_{I \in \mc I}
& \int \ldots  \int \one\{(0,y_2,\ldots,y_m)\in C_{r_n}^I \} \prod_{j=1}^{m} \lambda(y_j) \times \\
& \quad \quad \quad \times \one\{ \varphi_n\big(\pi_I((y_1,\ldots,y_m) ), (y_1,\ldots,y_m) \big)\in D\}
\ \d y_1 \ldots \d y_m.
\vspace{0.75ex} \end{align*}
For each $I \in \mc I$, there is a permutation $\sigma_I$ such that
$\pi_I(\sigma_I(\yy))= \yy_k$, where $\yy_k = (y_1,\ldots,y_k)$. As permutations are continuously differentiable,
invertible, and the determinant of the Jacobian is $\pm 1$,
\vspace{0.75ex} \begin{equation*}
\begin{aligned}
\E \big[ \hat{\xi}^{4}_{n,\lambda }(D)\big]= n^m \f{\binom{m}{k}}{m!}
& \int \ldots  \int \one\{(0,y_2,\ldots,y_m)\in C_{r_n}^- \} \prod_{j=1}^{m} \lambda(y_j) \times \\
& \quad \quad \quad  \quad \quad \times \one\{ \varphi_n\big((y_1,\ldots,y_k) , (y_1,\ldots,y_m) \big)\in D\}
\ \d y_1 \ldots \d y_m.
\end{aligned}
\vspace{0.75ex} \end{equation*}
Applying the change of variables $\phi(w) = r_n w + y_1$ and 
using that under assumption $F$, additive lifetimes are homogenous, and multiplicative lifetimes are scale-invariant, 
this yields (\ref{eq:general_lambda}).

\noindent \textit{Special $\lambda\equiv 1$:} Recall we want to prove that
\vspace{0.75ex} \begin{equation*}
\E \big[ \hat{\xi}^{2}_{n,1}([0,1]^d)\big]
= \rho_{n,m}\frac{\binom{m}{k}}{m!} \int 
\one\big\{(0,\yy)\in C_1^-, \ \ell^\circ_{(0,\yy)_k,(0,\yy)}\ge \lmax^\circ - u_{n,\alpha }, \ r_{(0,\yy)_k} \leq r_n \big\}
\d\yy = \alpha.
\vspace{0.75ex} \end{equation*}
Linearity of the center $z(\cdot)$ and (\ref{eq:general_lambda})
imply that
\vspace{0.75ex} \begin{align*}
\E \big[ \hat{\xi}^2_{n,1}([0,1]^d)\big]
= \rho_{n,m}\frac{\binom{m}{k}}{m!} \int \ldots  \int& 
\one\big\{(0,y_2,\ldots,y_m)\in C_1^-, \ \ell^\circ_{(0,y_2,\ldots,y_k),(0,y_2,\ldots,y_m)}\ge \lmax^\circ - u_{n,\alpha } \big\} \times \\
& \ \times \one\{ r_{(0,y_2,\ldots,y_k)} \leq r_n, \ y_1 \in [0,1]^d-r_n z_{(0,y_2,\ldots,y_k)}\} \ \d y_1 \ldots \d y_m.
\vspace{0.75ex} \end{align*}
Plugging in the choice of $u_{n,\alpha}$ in assumption $T$ and using (\ref{eq:gdef_int}) completes the proof.
\end{proof}

\begin{proof}[Proof of Lemma \ref{lem:pruning}]
	Recall we want to prove that
	\vspace{0.75ex} $$
	\E\bigg[ \sum_{\xx \in \PP_{n\kappa}^{(j)}} \one\{\xx \in C_{r_n}(j) \} \bigg] \leq C j^{j-2}C_1^{j-1} \rho_{n,j}.
	\vspace{0.75ex} $$
	We now perform a similar pruning argument as in the proof of Lemma \ref{lem:Ln_bound}: Let $\mc T_j$ denote all spanning trees on the vertices $\{1,\ldots,j\}$. Since $\xx =
	(x_1,\ldots,x_j)$ is connected, it contains a spanning tree $T \in \mc T_j$ and
	hence by applying the union bound and Mecke's formula, then
	\vspace{0.6ex} $$
	\E\bigg[ \sum_{\xx \in \PP_{n\kappa}^{(j)}} \one\{\xx \in C_{r_n}(j) \} \bigg]  \leq 
	n^j \sum_{T \in \mc T_j} \int \prod_{(i,i') \in T} 
	\one\{\Vert x_i-x_{i'} \Vert  < r_n  \} \prod_{i=1}^j \kappa(x_i) \ \d \xx.
	\vspace{0.6ex} $$
	For $T \in \mc T_j$, select now a vertex $i_0$ of degree 1 in $T$ with unique
	neighbor $i_1$. Then
	\vspace{0.6ex} \begin{align*}
	\E\bigg[ \sum_{\xx \in \PP_{n\kappa}^{(j)}} \one\{\xx \in C_{r_n}(j) \} \bigg] 
	& \leq 
	n^j r_n^d \vol(B(0,1)) \sup_{y \in W} \kappa(y)
	\sum_{T \in \mc T_j} \int \ldots \iint \ldots \int  \prod_{\substack{i=1 \\ i \neq i_0}}^j \kappa(x_i) \times  \\
	& \quad \quad \quad \quad \quad \quad 
	\times \prod_{\substack{(i,i') \in T \\ i \neq i_0, i' \neq i_1}} 
	\one\{\Vert x_i-x_{i'} \Vert < r_n \}
	\ \d x_1 \ldots \d x_{i_0-1} \d x_{i_0+1} \ldots \d x_j.
	\vspace{0.6ex} \end{align*}
	Iterating this process on the remaining vertices of $T$
	and letting $c = \int \kappa(y) \ \d y < \infty$, then
	\vspace{0.75ex} $$
	\E\bigg[ \sum_{\xx \in \PP_{n\kappa}^{(j)}} \one\{\xx \in C_{r_n}(j) \} \bigg]  \leq 
	c n^j \big( r_n^d \vol(B(0,1))\sup_{y \in W} \kappa(y) \big)^{j-1} \sum_{T \in \mc T_j} 1.
	\vspace{0.75ex} $$
	By Cayley's formula, the number $ \abs{\mc T_j}$ of spanning trees is $j^{j-2}$ and therefore completes the proof.
	\end{proof}		
	
	\begin{proof}[Proof of Lemma \ref{lem:dkrhat}]
	Recall we want to prove that
	$
	\dkr(\xi_n^p ,\hat\xi_n^p) \in O(\r_{n, m+1}),
	$
	and note that it suffices to prove that $\dkr(\xi_n^4 ,\hat\xi_n^4) \in O(\r_{n, m+1})$.
	By the definition of the Kantorovich-Rubinstein distance,
	\vspace{0.75ex} \begin{equation*}
	\dkr(\xi_n^4 ,\hat\xi_n^4)
	\leq \E[\dtv(\xi_n^4 , \hat\xi_n^4)] 
	= \E \bigg[\max_{A \in \mc B(\R^d \times [0,\infty)^2)} 
	\vert \xi_n^4(A) - \hat\xi_n^4(A) \vert\bigg]. 
	\vspace{0.75ex} \end{equation*} 
	By the union bound and Lemma \ref{lem:equalmeasures}, it suffices to
	compute the difference of $\xi_n$ and $\hat\xi_n$ on $j$-clusters
	for $j \geq m+1$ and on these
	$\xi_n$ and $\hat\xi_n$ are both bounded by $\binom{j}{k}$ 
	by assumption $U$. Thus
	\vspace{0.75ex} $$
	\dkr(\xi_n^4 ,\hat\xi_n^4)
	\leq \sum_{j=m+1}^\infty \frac{2 \binom{j}{k}}{j!} \E\bigg[ \sum_{\xx \in \PP_n^{(j)}} 
	\one\{\xx \in C_{r_n}(j)\}\bigg].
	\vspace{0.75ex} $$
	By Lemma \ref{lem:pruning}, it holds that
	\vspace{0.75ex} $$
	\dkr(\xi_n^4 ,\hat\xi_n^4)  \leq \sum_{j=m+1}^\infty 
	\frac{2 c j^{j-2} (\sup_{y \in W} \kappa(y) \vol(B(0,1)))^{j-1}\rho_{n,j}}{k!(j-k)!}.
	\vspace{0.75ex} $$
	Using that $j! \leq j^j$ and $x \leq \exp(x)$, it follows by Stirling's
	formula that there is a $c' > 0$ such that
	\vspace{0.75ex} $$
	\dkr(\xi_n^4 ,\hat\xi_n^4)
	\leq 2cc' \sum_{j=m+1}^\infty \big(\exp(k+1)\sup_{y \in W} \kappa(y)\vol(B(0,1))\big)^{j-1} \rho_{n,j}.
	\vspace{0.6ex} $$
	Pulling out the $m$ first factors as well as the one factor of $n$, then
	\vspace{0.6ex} \begin{equation}
	\label{eq:geometric}
	\begin{aligned}
	\dkr(\xi_n^4 ,\hat\xi_n^4) &\leq 2cc'\big(\exp(k+1)\sup_{y \in W} \kappa(y) \vol(B(0,1))\big)^{m} \rho_{n,m+1} \times  \\
	& \quad \quad \quad \quad \quad \quad \quad \quad \times  \sum_{j=0}^{\infty} \big(\exp(k+1)\sup_{y \in W} \kappa(y) \vol(B(0,1))nr_n^d\big)^j.
	\end{aligned}
	\vspace{0.6ex} \end{equation}
	Since assumption M implies $nr_n^d \rightarrow 0$ as $n \to \infty$,
	the geometric series in (\ref{eq:geometric}) is finite for
	all sufficiently large $n$, which then completes the proof.
	\end{proof}
	
	\begin{proof}[Proof of Lemma \ref{lem:totalvariationestimate}]
		Recall for any Borel set $B$ in the domain of $\hat\xi^p_n$, we want to prove that 
		\vspace{0.6ex}
		$$
		\dkr(\hat\xi^p_n \cap B , \zeta) 
		\in O\big(\dtv\big(\E[\hat\xi^p_n\cap B ],\E[\zeta]\big) 
		+ \rho_{n,m+1} \big).
		\vspace{0.6ex}
		$$
		Let $\varphi_n^4(\xx,\yy) = \varphi_n(\xx,\yy)$, $\varphi_n^3(\xx,\yy) = (z_\xx,(\lmax^\circ - \hat \ell_{\xx,\yy}^\circ)/u_{n,\alpha})$, and 
		$\varphi_n^2(\xx,\yy) = z_\xx$, and define the set
		\vspace{0.75ex}
		$$
		A_n^p = \bigg\{ \yy \in C_{r_n}
		\colon \sum_{\xx \in \yy^{(k)}} 
		\one\{\xx \in \mc N(\yy)\} \de\{\varphi_n^p(\xx,\yy)\} > 0 \bigg\},
		\vspace{0.75ex}
		$$
		which is a set closed under coordinate permutations. By Assumption $U$ and since $u_{n,\alpha} \to 0$, it holds for all sufficiently large $n$ that
		\vspace{0.75ex}
		\begin{equation}
		\label{eq:u_statistic}
		\hat\xi_{n,\lambda}^p(D)  = \f1{m!} \sum_{ \yy\in \PP_{n\lambda}^{(m)}\cap A_n^p} \delta\bigg\{\f1{k!} \sum_{ \xx \in  \yy^{(k)}} 
		\one\big\{ \xx \in \mc N(\yy), \ \varphi_n^p(\xx,\yy) \in D\big\} \varphi_n^p(\xx,\yy) \bigg\}(D).
		\vspace{0.75ex}
		\end{equation}
		
		Define the set $r(A_n^p)$ as 
		\vspace{0.6ex}
		\begin{equation}
		\label{eq:rA_n}
		\begin{aligned}
		r(A_n^p) =  & \max_{1 \leq j \leq m-1} n^j  \int \ldots \int \prod_{i=1}^{j}\kappa(y_i) \ \times \\
		& \quad \times \bigg(n^{m-j} \int \ldots \int \one_{A_n^p}(y_1, \ldots, y_m)
		\prod_{i=j+1}^{m} \kappa(y_i) \ \d y_m \ldots \d y_{j+1} \bigg)^2  \ \d y_j \ldots \d y_1.
		\end{aligned}
		\vspace{0.6ex}
		\end{equation}
		
		If $\yy \in A_n^p$, then Assumption $F$ implies $y_{j+1},\dots, y_{m} \in B(y_1,2mr_n)$. Hence, using Assumption $P$,
		\begin{equation}
		\label{eq:max_bound}
		n^{m-j}\int \ldots \int \one_{A_n^p}(y_1, \ldots, y_m)
		\prod_{i=j+1}^{m} \kappa(y_i) \ \d y_m \ldots \d y_{j+1} \leq \big(\k_*\vol(B(0,2m))\big)^{m-j}(nr_n^d)^{m-j},
		\end{equation}
		where $\kappa_* = \sup_{y \in W} \kappa(y) < \infty$. Moreover, since $\yy \in A_n^p$ implies $\yy \in C_{r_n}$, then by Mecke's formula and Lemma \ref{lem:pruning}, it holds that
		\vspace{0.75ex}
		\begin{equation}
		\label{eq:cluster_bound}
		n^m  \int \one_{A_n^p}(\yy)
		\prod_{i=1}^{m} \kappa(y_i) \ \d \yy \leq c m^{m-2}\big(\kappa_*\vol(B(0,1))\big)^{m-1} \rho_{n,m}.
		\vspace{0.75ex}
		\end{equation}
		
		Now plugging in (\ref{eq:max_bound}) and (\ref{eq:cluster_bound}) into (\ref{eq:rA_n}) gives that
		\vspace{0.75ex}
		\begin{equation}
		\label{eq:rn_bound}
		r(A_n^p)
		\leq   c m^{m-2}\big(\kappa_*\vol(B(0,1))\big)^{m-1} \rho_{n,m}  \max_{1\le j \le m-1} \bigg(\big(\k_*\vol(B(0,2m))\big)^{m-j}(nr_n^d)^{m-j}\bigg).
		\vspace{0.75ex}
		\end{equation}
		
		By Assumption $M$, then $n r_n^d \to 0$, so the maximum in (\ref{eq:rn_bound}) is attained with $j=m-1$, and since $\rho_{n,m}nr_n^d = \rho_{n,m+1}$,
		 then $r(A_n^p) \in O(\rho_{n,m+1})$. Finally, using (\ref{eq:u_statistic}) and \cite[Theorem 3.1]{decreusefond},
		\vspace{0.75ex}
		\begin{equation*}
		\label{eq:poisest}
		\dkr(\hat \xi_n^p \cap B , \zeta) \leq \dtv(\E[\hat \xi_n^p  \cap B],\E[\zeta])  + \f{2^{m+1}}{m!} r(A_n^p) \in O(\dtv(\E[\hat \xi_n^p  \cap B],\E[\zeta]) + \rho_{n,m+1}),
		\vspace{0.75ex}
		\end{equation*}
		which then completes the proof.
		\end{proof}		

\begin{proof}[Proof of Lemma \ref{lem:removekmix}]
Recall that for $K_1,K_2$ compact and $B=\R^d \times K_1 \times K_2$, we want to prove that $
\dtv\big(\E[\hat\xi_n^4\cap B] ,\hat\Theta^4_n \cap B \big) \in O(r_n).
$
Introduce
\vspace{0.75ex} $$
W_n=\{y \in W \colon B(y,2(m-1)r_n) \su W\},
\vspace{0.75ex} $$
as well as $\tilde \rho_{n,m} = \rho_{n,m} \binom{m}{k}/(m!) $ and $\phi(u) = r_n u + y_1$.
Let $A \in \mc B(B)$. By Lemma \ref{lem:expansion},
\vspace{0.75ex} \begin{equation}
	\label{eq:expanded_3}
	\begin{aligned}
		\E & \big[ \hat{\xi}^4_n(A)\big] = \tilde \rho_{n,m}
		\int_{W_n}\int \one\{ (0,\yy) \in C_1^-, \varphi_n(\phi((0,\yy)_k), \phi((0,\yy)))\in A \}\kappa(y_1)
		 \prod_{j=2}^{m} \kappa(\phi(y_j)) \ \d\yy \d y_1 + \\
		& \quad \tilde \rho_{n,m}
		\int_{W_n^c}\int \one\{ (0,\yy) \in C_1^-, \varphi_n(\phi((0,\yy)_k), \phi((0,\yy)))\in A \}\kappa(y_1)
		 \prod_{j=2}^{m} \kappa(\phi(y_j)) \ \d\yy \d y_1 =: E_1 + E_2.
	\end{aligned}
\vspace{0.75ex} \end{equation}
Note that by Lemma \ref{lem:expansion} and assumption $P$, it follows that $E_2 \in O(r_n)$.
		If $(0,\yy)$ is 1-connected, then assumption $F$ implies that
		$r_n y_2+y_1,\ldots,r_n y_{m} +y_1 \in B(y_1,2(m-1)r_n)$ and this ball
		is contained in $W$ whenever $y_1 \in W_n$ by definition of $W_n$.
		Let 
\vspace{0.75ex} $$
d_n = \sup_{y ,y' \in W \colon |y-y'| \le 2(m-1)r_n} 
\abs{\k(y) - \k(y') }
\vspace{0.75ex} $$
and note that when $y_1 \in W_n$, then
\vspace{0.75ex} \begin{equation*}
 (\kappa(y_1)-d_n)^{m-1} \wedge 0 \leq \prod_{j=2}^{m} \kappa(r_ny_j+y_1)  \leq (\kappa(y_1)+d_n)^{m-1}.
	\vspace{0.75ex} \end{equation*}
By assumption $P$, then $d_n \to 0$ and by the 
The Binomial Theorem, there 
is a $c > 0$ such that
\vspace{0.75ex} \begin{equation}
	\label{eq:densitybounds}
	 \kappa(y_1)^{m-1}-c d_n  \leq \prod_{j=2}^{m} \kappa(r_ny_j+y_1) \leq \kappa(y_1)^{m-1}+c d_n.
	\vspace{0.75ex} \end{equation}
Now define the quantities $\hat\Theta^+_n,\hat\Theta^-_n$ by
	\vspace{0.75ex} $$
	\hat\Theta^\pm_n(A)=\tilde \rho_{n,m}
	\int \one\{ (0,\yy) \in C_1^-, \varphi_n(\phi((0,\yy)_k), \phi((0,\yy)))\in A \} \kappa(y_1)(\kappa(y_1)^{m-1}\pm c d_n) \ \d \yy.
	\vspace{0.75ex} $$
	Then by combining (\ref{eq:densitybounds}) and (\ref{eq:expanded_3}), we have that 
\vspace{0.75ex} 
\begin{equation}
	\label{eq:squeezeIM}
	\hat\Theta^-_n(A) + E_2 \leq \E \big[ \hat{\xi}_n^3(A) \big] \leq \hat\Theta^+_n(A) + E_2.
\vspace{0.75ex} 
\end{equation}
Let $t_* = \sup K_1$ and $c' = \int \kappa(y_1) \d y_1$. By assumptions $G$ and $T$, as well as a Taylor expansion, $\tilde \rho_{n,m}\in O(u_{n,\alpha}^{-1/q})$ and $g(t_* u_{n,\alpha}) \in O(u_{n,\alpha}^q)$. Hence, by (\ref{eq:hatE}), Lemma \ref{lem:expansion}, and assumption $P$,
\vspace{0.75ex} 
\begin{align*}
	\abs{\hat\Theta^\pm_n(A) - \hat\Theta^4_n(A) } 
	\leq c c' d_n \tilde \rho_{n,m} g(t_* u_{n,\alpha}) \in O(r_n).
\vspace{0.75ex} 
\end{align*}
Using the fact that $a \leq b \leq c$ implies 
$\vert b \vert \leq \vert a \vert \vee \vert c \vert$, it follows that 
\vspace{0.75ex} 
$$
\abs{\E \big[ \hat{\xi}_n^3(A) \big]- \hat\Theta^4_n(A)} 
\leq \bigg(\abs{\hat\Theta^-_n(A) - \hat\Theta^4_n(A)}\bigg) \vee
\bigg( \abs{ \hat\Theta^+_n(A) - \hat\Theta^4_n(A)} \bigg) 
\in O(r_n).
\vspace{0.75ex} 
$$ 
Hence, taking the supremum over all Borel sets $A$ completes the proof.
\end{proof}

\bep[Proof of Lemma \ref{lem:approximation_g}]
Let $q\geq 1$ be chosen in accordance with assumption $G$. 
By a Taylor expansion of $g$ and using the mean value form of the remainder,
there exists $z_n \in [0,u_{n,\alpha}]$ such that
\vspace{0.75ex} 
\begin{equation}
	\label{eq:taylor_g}
g(u_{n,\alpha}) = \frac{g^{(q)}(z_n)}{q!}u_{n,\alpha}^{q}.
\vspace{0.75ex} 
\end{equation}
By Lemma \ref{lem:expansion}, it 
follows that $ \rho_{n,m}\binom{m}{k}/(m!) = \alpha/ g(u_{n,\alpha}) $. 
Hence, using (\ref{eq:taylor_g}) for any $u \in [0,a]$,
\vspace{0.75ex} 
\begin{equation}
	\label{eq:frac_of_g_diff}
	\bigg\vert \rho_{n,m}\tfrac{\binom{m}{k}}{m!}  u_{n,\alpha} g^{(1)}(u_{n,\alpha} u) - \alpha qu^{q-1} \bigg\vert 
	= \alpha \bigg \vert \frac{q! g^{(1)}(u_{n,\alpha} u)}{u_{n,\alpha}^{q-1} g^{(q)}(z_n)} 
	- qu^{q-1} \bigg\vert.
\vspace{0.75ex} 
\end{equation}
By Remark \ref{rem:sparse_extension} and continuity, $g^{(q)}(z_n) \rightarrow
g^{(q)}(0)$ as $n \to \infty$. By L'Hôpital's rule,
\vspace{0.75ex} 
$$
\frac{g^{(1)}(u_{n,\alpha} u)}{u_{n,\alpha}^{q-1}} \longrightarrow \frac{g^{(q)}(0)}{(q-1)!} u^{q-1},
\vspace{0.75ex} 
$$
and thus there exists $g_* > 0$ such that $g^{(1)}(u_{n,\alpha} u)u_{n,\alpha}^{-(q-1)} \leq g_*$ uniformly. By the expansion,
\vspace{0.75ex} 
\begin{align*}
\frac{q!g^{(1)}(u_{n,\alpha} u)}{g^{(q)}(z_n)u_{n,\alpha}^{q-1}} - qu^{q-1} 
&= \bigg(\frac{q!}{g^{(q)}(z_n)}
-\frac{q!}{g^{(q)}(0)}\bigg)\frac{g^{(1)}(u_{n,\alpha} u)}{u_{n,\alpha}^{q-1}}
 + \bigg(\frac{g^{(1)}(u_{n,\alpha} u)}{u_{n,\alpha}^{q-1}}
  - \frac{g^{(q)}(0)}{(q-1)!} \bigg)\frac{q!u^{q-1} }{g^{(q)}(0)}.
\vspace{0.75ex} 
\end{align*}
Using a Taylor expansion of $g$ and the Mean Value Theorem (twice), there exist numbers $z_n' \in [0,z_n]$ and $u_{n,\alpha}' \in [0,u_{n,\alpha}]$,
where $z_n', u_{n,\alpha}' \leq u_{n,\alpha} \rightarrow 0$ by Lemma \ref{lem:properties_of_g},
such that
\vspace{0.75ex} 
\begin{equation}
	\label{eq:final_g_eq}
\bigg\vert \frac{q!g^{(1)}(u_{n,\alpha} u)}{g^{(q)}(z_n)u_{n,\alpha}^{q-1}} - qu^{q-1} \bigg\vert 
\leq \bigg\vert  \frac{q!g^{(q+1)}(z_n')}{\big(g^{(q)}(z_n')\big)^2}\bigg\vert z_n g_*
 + \bigg\vert\frac{g^{(q+1)}(u_{n,\alpha}')}{(q-1)!}\bigg\vert u_{n,\alpha} \frac{q!u^{q-1} }{g^{(q)}(0)} \in O(u_{n,\alpha}),
\vspace{0.75ex} 
\end{equation}
since $g^{(q)}, g^{(q+1)}$ are
continuous. Combining equations (\ref{eq:frac_of_g_diff}) and (\ref{eq:final_g_eq}) completes the proof.
\end{proof}

\bep[Proof of Lemma \ref{lem:approximation_h}]
Let $q\geq 1$ be chosen in accordance with Assumption $H$.
By equation (\ref{eq:taylor_g}),
\vspace{0.75ex} 
$$
h(u_{n,\alpha},1) = \frac{h^{(q,0)}(z_n,1)}{q!}u_{n,\alpha}^{q}.
\vspace{0.75ex} 
$$
Therefore, by Lemma \ref{lem:expansion} and equation (\ref{eq:hdef_int}), it follows 
for any $u,v \in [0,a]$ that
\vspace{0.75ex} 
\begin{equation*}
	\label{eq:frac_of_h_diff}
	\bigg\vert \rho_{n,m}\frac{\binom{m}{k}}{m!}  u_{n,\alpha} \tilde h^{(1,1)}(u_{n,\alpha} u,v) - \alpha qu^{q-1}\frac{\tilde h^{(q,1)}(0,v)}{h^{(q,0)}(0,1)}\bigg\vert 
	= \alpha \bigg \vert \frac{\tilde h^{(1,1)}(u_{n,\alpha} u,v) q!}{u_{n,\alpha}^{q-1} \tilde h^{(q,0)}(z_n,1)} 
	- qu^{q-1} \frac{\tilde h^{(q,1)}(0,v)}{h^{(q,0)}(0,1)}  \bigg\vert.
\vspace{0.75ex} 
\end{equation*}
By Lemma \ref{lem:properties_of_g} and continuity, $h^{(q,0)}(z_n,1) \rightarrow
h^{(q,0)}(0,1)$ as $n \to \infty$. By L'Hôpital's rule,
\vspace{0.75ex} 
$$
\frac{\tilde h^{(1,1)}(u_{n,\alpha} u,v)}{u_{n,\alpha}^{q-1}} \longrightarrow \frac{\tilde h^{(q,1)}(0,v)}{(q-1)!} u^{q-1}.
\vspace{0.75ex} 
$$
Rewriting and using the Mean Value Theorem in the same manner as in Lemma \ref{lem:approximation_g},
it follows by Lemma \ref{lem:properties_of_g} and the same line of argumentation 
as in Lemma \ref{lem:approximation_g} that 
\vspace{0.75ex} 
$$
\bigg\vert \rho_{n,m}\frac{\binom{m}{k}}{m!}  u_{n,\alpha} \tilde h^{(1,1)}(u_{n,\alpha} u,v) - \alpha qu^{q-1}\frac{\tilde h^{(q,1)}(0,v)}{h^{(q,0)}(0,1)}\bigg\vert
\in O(u_{n,\alpha}),
\vspace{0.75ex} 
$$
which completes the proof.
\enp






\vspace{1ex}
\section{Verifying assumptions}
\label{sec:5}
This section is dedicated to discussing the extent to which the model and topological assumptions 
in the unbounded and $m$-sparse regimes are valid. In Section \ref{sec:5.1}, we verify that the thresholds in assumption $T'$ are well-defined, 
which was the only point of concern in the unbounded regime. 
In Section \ref{sec:5.2}, we turn to the $m$-sparse regime. First, 
we analogously verify that the thresholds in assumption $T$ are well-defined and provide 
sufficient conditions for $\kappa$ and $W$ in assumption $P$. Next, we
partially verify the uniqueness of the near-maximal lifetime in assumption $U$. 
The remaining combinations of $d$, $k$, and $m$ are left as an open problem. 
Finally, we also partially verify the regularity conditions on $h$ in assumption $H$
in dimension $d=2$ and for additive lifetimes. Similarly as before, we leave $d \geq 3$ and multiplicative lifetimes as an open problem.
\subsection{The unbounded regime}




\label{sec:5.1} 
\begin{proposition}[Assumption $T'$]
	\label{prop:T'}
	Under assumption $F'$, it holds that $v : [1,\infty) \to [0,\infty)$ in (\ref{eq:v_def}) 
	is continuous and strictly decreasing. Hence, $v$ has a continuous inverse $v^{-1}$.
\end{proposition}
\begin{proof}
	The arguments in this proof are inspired by the proof of \cite[Lemma 2]{CH20}.

	\noindent \textit{Continuity:} First, we show that $v$ is continuous. Let $\ell \geq 1$ and 
	$\xx = (x_1,x_2,x_3) \in \mc N(\PP_n^\xx)$ with $\ell^*_{\xx,\PP_n^\xx} = \ell$.
	By the definition of the \v{C}ech filtration, there exist $x_1',x_2' \in \PP_n^\xx$ such that
	\[
	\ell = r_\xx/\big(\Vert x_1' - x_2' \Vert/2\big),
	\]
	i.e., the 1-cycle destroyed by $\xx$ is born when
	the edge between $x_1$ and $x_2$ is added. Assuming, without loss of generality, that $x_1',x_2' \neq x_3$, it
	follows that $x_3 \in D(x_1,x_2)$, where $D(x_1,x_2)$ is the union of the boundaries of the two balls
	of radius $\ell \Vert x_1' - x_2' \Vert/2 $ with $x_1$ and $x_2$ on the boundary.
	Since $D(x_1,x_2)$ is a null set with respect to $\vol(\cdot)$, it follows by Tonelli's theorem that
	\vspace{0.6ex} \begin{equation}
		\label{eq:point_lifetime_mass_torus}
		\int_{\T^{2\cdot 3}}\one\{\xx \in \mc N(\PP_n^\xx)\}
		\one\{ \ell^*_{\xx,\PP_n^\xx} = \ell \} \ \d \xx = 0.
	\vspace{0.6ex} \end{equation}
	Hence, for any $\e \geq 0$, it
	follows by (\ref{eq:point_lifetime_mass_torus}) and the Monotone Convergence Theorem that
	\vspace{0.6ex} $$
	\abs{v(\ell) - v(\ell + \e)} \leq \int_{\T^{2\cdot 3}}\one\{\xx \in \mc N(\PP_n^\xx)\}
	\one\{\ell \leq \ell^*_{\xx,\PP_n^\xx}\leq \ell + \e \}\
	\d \xx \longrightarrow 0
	\vspace{0.6ex} $$ 
	as $\e \to 0$. Using an analogous argument for $\e < 0$, we conclude $v$ is continuous.
	
	\noindent \textit{Strictly decreasing:} Note that it suffices to show that with positive Lebesgue measure,
	we have a 1-cycle with multiplicative lifetime inside $(\ell - \e,\ell + \e)$ for any $\ell > 1$ and $\e > 0$.
	Consider $\xx=(x_1,x_2,x_3)$ all on the boundary of the unit ball $B(0,1)$ such that $\Vert x_1 - x_2 \Vert = 1/\ell$ and
	such that $\Vert x_1 - x_3 \Vert = \Vert x_2 - x_3 \Vert > 1/\ell$. Now add points $x_4, x_5,\ldots, x_m$ on the circle such
	that the distance between any two neighbors is less than $1/\ell$ and with $x_1$ and $x_2$ remaining neighbors (see Figure \ref{fig:constructed_lifetime}).
	Since the interior
	of $B(z_\xx,r_\xx)$ is empty and $z_\xx$ is in the interior of $\text{conv}(\xx)$, then $\xx$ is negative by Remark \ref{rem:cech_properties}(1).
	Moreover, the multiplicative lifetime of $\xx$ is $1/(1/\ell)=\ell$ if there are no other points inside $B(0,2)$.
	If, for $\e' > 0$, we let each point $x_i$ for $i=1,\ldots,m$ vary inside a small half-ball of radius $\e'$ centered at $x_i$ (such that $B(z_\xx,r_\xx)$ 
	remains empty), it follows by 
	the Hausdorff stability of persistent homology (\cite[Corr. 11.29]{chazal}) that the birthtime and deathtime
	only change by at most $\e'$, and by choosing $\e'$ small enough, it follows that 
	the multiplicative lifetime is inside $(\ell - \e,\ell + \e)$.
	Scaling the points to fit inside the torus does not change the multiplicative lifetime.
\end{proof}
	\vspace{-2ex}

	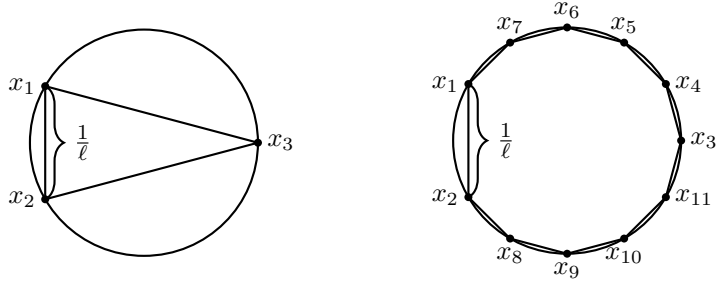
\begin{figure}[h]
    \centering

    \begin{subfigure}{0.25\textwidth}
        \centering
        \begin{tikzpicture}[scale=0.75]
            \draw[thick] (0,0) circle (2cm);

            \coordinate (A) at ({2*cos(0)},{2*sin(0)}); 
            \coordinate (B) at ({2*cos(150)},{2*sin(150)}); 
            \coordinate (C) at ({2*cos(210)},{2*sin(210)}); 

            \fill[black] (A) circle (2pt) node[right] {$x_3$};
            \fill[black] (B) circle (2pt) node[left] {$x_1$};
            \fill[black] (C) circle (2pt) node[left] {$x_2$};

            \coordinate (E) at ({2*cos(90)},{2*sin(90)}); 
            \coordinate (I) at ({2*cos(270)},{2*sin(270)}); 
            \fill[black] (I) circle (0.001pt) node[below = 5pt] {${ \ }$};

            \draw[thick] (A) -- (B) -- (C) -- cycle;
            \draw[decorate, decoration={brace, amplitude=7pt}, thick] (B) -- (C) node[midway, right = 7 pt] {\large $\frac{1}{\ell}$};

        \end{tikzpicture}
    \end{subfigure}
    \hspace{0.1\textwidth}
    \begin{subfigure}{0.25\textwidth}
        \centering
        \begin{tikzpicture}[scale=0.75]
            \draw[thick] (0,0) circle (2cm);

            \coordinate (A) at ({2*cos(0)},{2*sin(0)}); 
            \coordinate (B) at ({2*cos(150)},{2*sin(150)}); 
            \coordinate (C) at ({2*cos(210)},{2*sin(210)}); 

            \coordinate (D) at ({2*cos(120)},{2*sin(120)}); 
            \coordinate (E) at ({2*cos(90)},{2*sin(90)}); 
            \coordinate (F) at ({2*cos(60)},{2*sin(60)}); 
            \coordinate (G) at ({2*cos(30)},{2*sin(30)}); 

            \coordinate (H) at ({2*cos(240)},{2*sin(240)}); 
            \coordinate (I) at ({2*cos(270)},{2*sin(270)}); 
            \coordinate (J) at ({2*cos(300)},{2*sin(300)}); 
            \coordinate (K) at ({2*cos(330)},{2*sin(330)}); 

            \fill[black] (A) circle (2pt) node[right] {$x_3$};
            \fill[black] (B) circle (2pt) node[left] {$x_1$};
            \fill[black] (C) circle (2pt) node[left] {$x_2$};

            \fill[black] (D) circle (2pt) node[above] {$x_7$};
            \fill[black] (E) circle (2pt) node[above] {$x_6$};
            \fill[black] (F) circle (2pt) node[above] {$x_5$};
            \fill[black] (G) circle (2pt) node[right] {$x_4$};

            \fill[black] (H) circle (2pt) node[below] {$x_8$};
            \fill[black] (I) circle (2pt) node[below] {$x_9$};
            \fill[black] (J) circle (2pt) node[below] {$x_{10}$};
            \fill[black] (K) circle (2pt) node[right] {$x_{11}$};

            \draw[thick] (C) -- (B) -- (D) -- (E) -- (F) -- (G) -- (A) -- (K) -- (J) -- (I) -- (H) -- cycle;
            \draw[decorate, decoration={brace, amplitude=7pt}, thick] (B) -- (C) node[midway, right = 7 pt] {\large $\frac{1}{\ell}$};

        \end{tikzpicture}
    \end{subfigure}

    \caption{Placement of the points $x_4,\ldots,x_m$ in the proof of Proposition \ref{prop:T'} such that $\Vert x_1 - x_2 \Vert $ remains the largest distance.}
    \label{fig:constructed_lifetime}
\end{figure}




\subsection{The $m$-sparse regime}

\label{sec:5.2} 
\vspace{1ex}
\begin{proposition}[Assumption $T$]
	\label{lem:properties_of_g}
	Under assumptions $F$, $M$ and $U$, it holds that $g$ in 
	 (\ref{eq:gdef_int})
	is continuous and strictly increasing on $\big(0,\lmax^\circ(m) - \lmax^\circ(m-1)\big)$. 
	Hence $g$ has a strictly increasing, continuous inverse $g^{-1}$.
	Additionally, under assumption $T$, then $u_{n,\alpha} \to 0$.
\end{proposition}

\begin{proof} We split the proof into three parts: continuity, strictly increasing, and convergence.

	\noindent \textit{Continuity:} First in the \v Cech case, we can use the same arguments in Proposition \ref{prop:T'} to prove that 
	the event of a specific additive and multiplicative is a Lebesgue null set, which by the Monotone Convergence Theorem implies continuity.
	As for the Vietoris-Rips case, we can replace the null set argument with the corresponding argument in \cite[Lemma 2]{CH20} (which again works for both additive and multiplicative lifetimes) and
	then use the Monotone Convergence Theorem to show continuity.

\noindent \textit{Strictly increasing:} Let $ \e_1, \e_2 \in  (0,\lmax^\circ(m) - \lmax^\circ(m-1)\big)$ with $\e_1 < \e_2 $ and 
let $\yy' \in \R^{dm}$ for which there exists
$\xx' \in \mc N(\yy')$ with $r_{\xx'} < 1$ and $\lmax^\circ(m) - \e_2 \leq \ell_{\xx',\yy'} \leq \lmax^\circ(m) - \e_1$ and
$\ell_{\xx',\yy'}>\lmax^\circ(m-1)$ which is possible by Lemma \ref{obs:lmaxincreasing}.
Next choose a sufficiently small $\delta > 0$ and large $R>0$ such that
\begin{enumerate}[label=(\roman*)]
	\item 
    $
    \displaystyle
    \begin{cases}
        \big(\lmax^\circ(m) - \e_2 \big) \vee \lmax^\circ(m-1) < (r_{\xx'} \pm \delta) - (b_{\xx',\yy'} \pm \delta) < \lmax^\circ(m) - \e_1, & \text{if } \circ = + \\
        \big(\lmax^\circ(m) - \e_2 \big) \vee \lmax^\circ(m-1) < (r_{\xx'} \pm \delta) / (b_{\xx',\yy'} \pm \delta) < \lmax^\circ(m) - \e_1, & \text{if } \circ = * \\
    \end{cases}
    $
	\item $b_{\xx',\yy'} > \delta$,
	\item $r_{\xx'} + \delta < 1$,
	\item $2\delta < r_{\xx'} - b_{\xx',\yy'}$,
	\item $z_{\yy}\in B(0, R)$ whenever $\Vert y_i - y_i' \Vert < \delta/2$ for all $1 \leq i \leq m$.
\end{enumerate}
Let $ \yy = (y_1,\dots, y_{m})$ satisfy that $\Vert y_i - y_i' \Vert < \delta/2$ 
for all $ 1 \leq i \leq m$. Hence the Hausdorff distance
between $ \yy$ and $ \yy'$ is less than $\delta/2$ and by Hausdorff stability of
persistent homology (see \cite[Corr. 11.29]{chazal}), 
there exist $\delta$-matching $\rho(\cdot, \yy)$, which recall satisfies that
\vspace{0.75ex} \begin{equation}
	\label{eq:delta_property}
\abs{r_\xx- r_{\rho(\xx,\yy)}} < \delta \quad \text{ and } \quad \abs{b_{\xx,\yy'} - b_{\rho(\xx,\yy),\yy}} < \delta \quad \text{ for all } \xx \in \mc N_1(\yy'),
\vspace{0.75ex} \end{equation}
where $\mc N_1(\yy') = \{\xx \in \mc N(\yy') \colon r_\xx < 1\}$. By (\ref{eq:delta_property}), (ii) and (iv), then $r_{\rho(\xx',\yy)} > b_{\rho(\xx',\yy),\yy} > 0$.
By (\ref{eq:delta_property}) and (i), then
$\ell_{\rho(\xx',\yy),\yy} > \lmax^\circ(m-1)$, so by Lemma
\ref{obs:lmaxincreasing} this implies that
$\yy \in C_1$. Moreover (iii) and (\ref{eq:delta_property}),
implies that $r_{\rho(\xx',\yy)}< 1$. Finally, then (v) implies $z_{\yy} \in  B(0, R)$. 
Combining all of this,
\vspace{0.75ex} \begin{align*}
0 & < 
\int  \prod_{i=1}^{m}
\one\big\{ \Vert y_i - y_i' \Vert < \tfrac \delta 2\big\} \ \d \yy \\
& \leq \int \one_{C_1}(\yy) \one_{\mc N_1(\yy)}(\rho(\xx',\yy)) \one_{B(0,R)}(z_{\yy}) 
\one\{\lmax^\circ(m) - \e_2 \leq \ell_{\rho(\xx',\yy),\yy}\leq \lmax^\circ(m) - \e_1\} \ \d \yy.
\vspace{0.75ex} \end{align*}
Introducing $\mc I$ and $\pi_I$ as in the proof of Lemma \ref{lem:expansion} and applying the union bound, then
\vspace{0.75ex} \begin{equation*}
	0 < 
	\sum_{I \in \mc I}
	\int  \one_{C_1}(\yy) \one_{\mc N_1(\yy)}(\rho(\pi_I(\yy),\yy)) \one_{B(0,R)}(z_{\yy}) 
\one\{\lmax^\circ(m) - \e_2 \leq \ell_{\rho(\pi_I(\yy),\yy),\yy}\leq \lmax^\circ(m) - \e_1\} \ \d \yy.
\vspace{0.75ex} \end{equation*}
For each $I \in \mc I$, there is a permutation $\sigma_I$ such that
$\pi_I(\sigma_I(\yy))= \yy_k$, where $\yy_k = (y_1,\ldots,y_k)$. Hence applying this substitution and subsequently translating by $(0,y_1,\ldots,y_1)$, it follows that
\vspace{0.75ex} \begin{equation*}
	0 < 
	\int \one_{F_0^1}(\yy)\one_{B(0,R)}(z_{\yy}) 
	\one\{\lmax^\circ(m) - \e_2 \leq \ell^\circ_{\yy_k,\yy}\leq \lmax^\circ(m) - \e_1\} \d \yy =  \vol(B(0,R)) \big(g(\e_2) - g(\e_1)\big),
\vspace{0.75ex} \end{equation*}
Division by $ \vol(B(0,R)) > 0$ proves the claim.

\noindent \textit{Convergence:} Note that $g(0)=0$ from the fact that $\ell = \lmax^\circ$ is a Lebesgue null-set. Thus, by assumption $M$, we see that
$
u_{n,\alpha} = g^{-1}\big(\tfrac{\alpha}{\rho_{n,m} \binom{m}{k}/m! }\big) \to 
0$
as $n \to \infty$. This completes the proof.
\end{proof}

\begin{proposition}[Assumption $P$]
	\label{prop:assumption_P}
	Assume $W=\R^d$ or $W$ is convex and bounded, and assume $\kappa$ is bounded, integrable, and continuously differentiable on the interior of $W$. Then assumption $P$ holds.
\end{proposition}
\begin{proof} The proof is split into two parts based on the two convergence conditions in assumption $P$.
	
	\noindent \textit{Modulus of continuity:} Let $\Vert y - y' \Vert \leq 2(m-1)r_n$. By the multivariate Mean Value Theorem and convexity of $W$, then for some $t \in [0,1]$,
	it follows that
	\vspace{0.75ex} $$
		\vert \kappa(y) - \kappa(y') \vert \leq \Vert y - y' \Vert \vert \kappa'(ty + (1-t)y') \vert \in O(r_n),
	\vspace{0.75ex} $$
	since $\kappa'$ is continuous (and locally finite). Taking supremum over such $y,y'$ completes the proof.

	\noindent \textit{Boundary condition:} Let $W^{-r_n}=\{y \colon B(y,2m r_n) \nsubseteq W\}$. If $W = \R^d$, then $W^{-r_n}$ is empty and we are done. Instead assume $W$ is convex and bounded. Since $\kappa$ is bounded,
	it suffices to prove that $\vol(W) - \vol(W^{-r_n}) \in O(r_n)$ and to this end we rely on the lower bound of inner parallel volumes in \cite{CS10}
	for convex sets: Let $W_i(W;B(0,1))$ denote the $i$'th relative quermassintegral of $W$ and $r(W;B(0,1))$ the relative inradius of $W$.
	By \cite[Theorem 1.1]{CS10} and $s_n \to 0$, it follows that
	\vspace{0.75ex} \begin{align*}
	\vol(W^{-r_n}) & \geq \vol(W) - \binom{d}{1} W_1(W;B(0,1)) s_n + \binom{d}{2}(d-2)\int_0^{r_n} \frac{(r_n-s)^2W_2(W^{-s};B(0,1))}{r(W;B(0,1))-s} \ \d s \\
	& \geq \vol(W) - d W_1(W;B(0,1)) r_n,
	\vspace{0.75ex} \end{align*}
	As $W_1(W;B(0,1)) < \infty$, then $\vol(W) - \vol(W^{-r_n}) \in O(r_n)$, which completes the proof. 
\end{proof}

\begin{proposition}[Assumption $U$]
	\label{prop:assumption_U}
	For \v Cech and $k \leq d+1$ and $m=k$, then assumption $U$ holds.
	For Vietoris-Rips, $k\leq d+1$ and $m=k$ or $k=d+1$ and $m=2d$, then assumption $U$ holds.
\end{proposition}
\begin{proof} First, let $d \geq 2$ and $k \leq d+1$. If $m=k$, then for any $\PP'$ as in assumption $U$, we must have that $\PP'=\PP$,
	 and since $\vert \PP^{(k)} \vert = 1 $, then
	in particular there is at most one $\xx$ with the property described in assumption $U$. 
	If instead $k=d+1$ and choosing the Vietoris-Rips filtration, the claim
	follows from the lower bound on the number of vertices of a non-trivial cycle in \cite[Lemma 4.4]{KM18}.
\end{proof}

\begin{proposition}[Assumption $H$ -- \v Cech]
	\label{prop:4}
	Let $d=2$ and $m=k=3$.
Using the \v Cech filtration, 
then $\lmax^+(3,3)=1-\tfrac{\sqrt 3}{2}$ and realized by an equilateral triangle of
side length $\tfrac{\sqrt 3}{2}$.
Moreover, assumption $H$ holds with $q=3$ and therefore $\ell_{n,\alpha} \sim 1-\tfrac{\sqrt 3}{2} - 1/(nr_n^{4/3})$ by Remark \ref{rem:sparse_extension}(5). 
\end{proposition}
\begin{proof} To improve the exposition, we split the proof into two parts: maximal lifetime and regularity.
	
	\noindent \textit{Maximal lifetime:}
	First, note that any three points in $\R^2$ with circumradius less than 1
	must be $1$-connected and that the unique triangle ($2$-simplex) is necessarily negative.
	 Note also that any right angle triangle has circumradius equal
	to half the length of the hypotenuse and therefore has lifetime 0.
	Therefore, the lifetime is positive if and only if the triangle is acute.
	For $y_2,y_3 \in \R^2$, let
	\vspace{0.75ex}  \begin{equation}
		\label{eq:triangle_angles}
	\theta = \arccos\bigg(\frac{\langle y_2, y_3 \rangle}{\Vert y_2 \Vert \Vert y_3\Vert }\bigg)
	\quad \text{and} \quad
	\phi = \arccos\bigg(\frac{\langle y_2 - y_3, -y_3 \rangle}{\Vert y_2 - y_3 \Vert \Vert y_3\Vert}\bigg)
	\vspace{0.75ex}  \end{equation}
	be the interior angles
	of the triangle $\triangle(0,y_2,y_3)$ at the points $0$ and $y_3$ (see Figure \ref{fig:triangle_param}). 
	Define
	\vspace{0.75ex} $$
	D = \{(\theta,\phi) : 0<\theta< \pi/2, \ 0<\phi<\pi/2,\ 0<\pi - \theta - \phi < \pi/2\},
	\vspace{0.75ex} $$
	and note that $\triangle(0,y_2,y_3)$ is acute 
	if and only if $(\theta,\phi) \in D$. Next, let $r > 0$ and $\tau \in [0,2\pi)$
	satisfy that $z((0,y_2,y_3))=(r\cos(\tau),r\sin(\tau))$, see Figure \ref{fig:triangle_param}.
	With these coordinates, then
	the birthtime is $r (\sin(\theta) \vee \sin(\phi) \vee \sin (\theta + \phi))$,
	the deathtime $r$ and the lifetime $r (1-\sin(\theta) \vee \sin(\phi) \vee \sin (\theta + \phi))$.
	This expression is maximized by taking $\theta = \phi = \pi/3$ and $r=1$, and
	hence the shape with maximal lifetime is an equilateral triangle with side length
	$\sqrt{3}/2$ and the lifetime is $1-\sqrt{3}/2$ as claimed.

\begin{figure}[h!]
	\centering
	\begin{tikzpicture}[scale=5.5]
		\draw[->] (-0.15,0) -- (0.85,0) ;
		\draw[->] (0,-0.15) -- (0,0.85) ;
	
		\filldraw[black] (0,0) circle (0.25pt) node[below left] {$0$};
		\filldraw[black] (0.8,0.2) circle (0.25pt) node[above right] {$y_3$};
		\filldraw[black] (0.2,0.7) circle (0.25pt) node[above left] {$y_2$};

		\filldraw[gray] (0.36,0.26) circle (0.25pt) node[above = 3pt] {$z$};

	\draw[thin] (0.1,0.025) arc[start angle=15,end angle=75,radius=0.1] node[above = 3.5pt, right = 1pt] {\small $\theta$};
	\draw[thin] (0.725,0.26) arc[start angle=147,end angle=194,radius=0.1] node[above = 10pt, left = 2pt] {\small $\phi$};
	\draw[thin] (0.178,0.615) arc[start angle=260,end angle=315,radius=0.1] node[above = 5pt, right = 3 pt] {\small $\pi - \theta - \phi$};
	\draw[thin, gray] (0.3,0) arc[start angle=0,end angle=36,radius=0.3] node[midway, above=5 pt, right = 0.5 pt] {\small $\tau$};

		\draw[black] (0,0) -- (0.8,0.2);
		\draw[gray] (0,0) -- (0.36,0.26);
		\draw[black] (0.8,0.2) -- (0.2,0.7);
		\draw[black] (0.2,0.7) -- (0,0);

	\draw [decorate,decoration={brace,mirror, amplitude=5pt},gray] (0.36,0.26) -- (0,0) node[midway,above=6pt,gray] {\small $r$};
	\end{tikzpicture}
	\caption{In black, the 2-simplex $(0,y_2,y_3)$ and angles $\theta, \phi$ as defined in the proof of Proposition \ref{prop:4}. 
	In gray, the center $z$ of $(0,y_2,y_3)$ in polar coordinates $(r,\tau)$.}
	 \label{fig:triangle_param}
\end{figure}
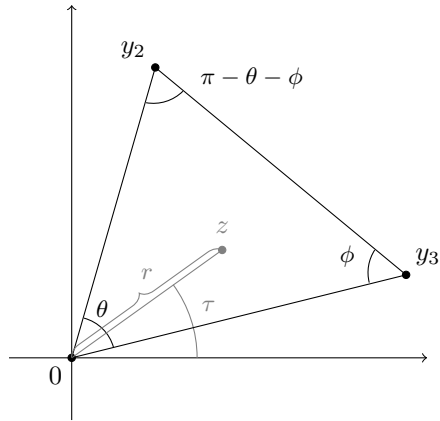

	\noindent \textit{Regularity:} Next, we verify assumption $H$. Observe that in this case (\ref{eq:hdef_int}) simplifies to
	\vspace{0.75ex}  $$
	h(u,v)  = \iint
	\one\{\ell_{(0,y_2,y_3),(0,y_2,y_3)} \geq \lmax^+ - u\} 
	\one\{ 1-v \leq r_{(0,y_2,y_3)}\leq 1\} \ \d y_2 \d y_3.
	\vspace{0.75ex}  $$
	Consider the transformation
	$T:(0,\infty)\times[0,2\pi)\times D \to \R^4$ defined as
	\vspace{0.75ex}$$
		T(r,\tau, \theta,\phi) = -2r\begin{pmatrix}
			\sin(\phi)\cos(\tau +\phi) \nonumber\\
			\sin(\phi)\sin(\tau +\phi) \nonumber\\
			\sin(\theta + \phi)\cos(\tau + \theta + \phi) \nonumber\\
			\sin(\theta + \phi)\sin(\tau + \theta + \phi) 
		\end{pmatrix},
		\vspace{0.75ex}$$
	and note $T(r,\tau, \theta,\phi)=(y_2,y_3)$.
	 Putting $\omega = \phi + \theta$, 
	 $\tilde \omega = \omega - \tau$ and $\tilde \phi = \phi - \tau $, the Jacobian of $T$ is
	\vspace{0.75ex} $$
	\scalemath{0.95}{
		2r \begin{pmatrix}
		\sin(\phi)\sin(\tilde \phi )/r & -\sin(\phi)\cos(\tilde \phi ) & 0 & \sin(\phi)\cos(\tilde \phi ) - \cos(\phi)\sin(\tilde \phi )\\
		\sin(\phi)\cos(\tilde \phi )/r & \sin(\phi)\sin(\tilde \phi ) & 0 & -\sin(\phi)\sin(\tilde \phi ) + \cos(\phi)\cos(\tilde \phi )\\
		\sin(\omega)\sin(\tilde \omega)/r & -\sin(\omega)\cos(\tilde \omega) & \sin(\omega)\cos(\tilde \omega)+\cos(\omega)\sin(\tilde \omega) & \sin(\omega)\cos(\tilde \omega) + \cos(\omega)\sin(\tilde \omega) \\
		\sin(\omega)\cos(\tilde \omega)/r  & \sin(\omega)\sin(\tilde \omega)& -\sin(\omega)\sin(\tilde \omega)+\cos(\omega)\cos(\tilde \omega) & -\sin(\omega)\sin(\tilde \omega) + \cos(\omega)\cos(\tilde \omega)
	\end{pmatrix}
	}
	\vspace{0.75ex} $$
	with determinant  $16r^3\sin(\phi)\sin(\theta)\sin(\phi + \theta)$.
	Hence applying $T$, it follows that
	\vspace{0.75ex}  $$
	h(u,v)  = 32\pi
	\int_{ 1-v}^1 \int_{D} \one\{r(1-\sin\theta \vee \sin\phi\vee \sin(\theta + \phi)) \geq \lmax^+ - u\}r^3 \sin(\phi)\sin(\theta)\sin(\phi + \theta)\  \d(\theta,\phi) \d r,
	\vspace{0.75ex}  $$
	First note when $(\theta,\phi)\in D$, then $\theta$ is larger
	than $\phi$ and $\pi - \theta - \phi$ if and only if
	$\pi/3 \leq \theta < \pi/2$ and  $\pi - 2\theta \leq \phi\leq \theta$.
	By partitioning $D$ into three sets,
	where $\theta$, $\phi$ and $\pi - \theta - \phi$ are the largest, respectively,
	and applying the change of variables $(r,\theta,\phi) \mapsto (r,\phi,\theta)$
	and $(r,\theta,\phi)\mapsto (r,\phi,\pi-\theta-\phi)$ in the latter two cases,
	it follows that
	\vspace{0.75ex}  $$
	h(u,v)  = 96 \pi
	 \iint_{\pi/3}^{\pi/2}\int_{\pi - 2\theta}^{\theta} \one\{1-v \leq r \leq 1 \} \one\{r(1-\sin\theta) \geq \lmax^+ - u\}r^3 \sin(\phi)\sin(\theta)\sin(\phi + \theta)\  \d \phi \d \theta \d r.
	\vspace{0.75ex}  $$
	Applying now the substitution $\ell = r(1-\sin(\theta))$, for $v < 1$, it holds that 
	\vspace{0.75ex}  \begin{equation}
		\label{eq:h_expanded}
	h(u,v)  = 96 \pi
 \int_{\lmax^+ - u}^{\lmax^+} \int_{\arcsin\big(1-\tfrac{\ell}{1-v}\big)}^{\arcsin(1-\ell)} \ell^3 f(\theta) \  \d \theta \d \ell, 
	\vspace{0.75ex}  \end{equation}
and 
\vspace{0.75ex}  \begin{equation}
	\label{eq:h_expanded_2}
h(u,1)  = 96 \pi
\int_{\lmax^+- u}^{\lmax^+} \int_{\pi/3}^{\arcsin(1-\ell)} \ell^3 f(\theta) \  \d \theta \d \ell, 
\vspace{0.75ex}  \end{equation}
where, using the trigonometric identity $\sin(\alpha)\sin(\beta)=\cos(\alpha - \beta) - \cos(\alpha + \beta)$, then
\vspace{0.75ex}  $$
	f(\theta)  = \frac{\sin\theta }{(1 - \sin\theta)^4}  \int_{\pi - 2\theta}^\theta \sin\phi\sin(\theta + \phi)  d\phi
	= \frac{\sin\theta((3\theta - \pi)\cos\theta - \sin(3\theta))}{2(1 - \sin\theta)^4}.
	\vspace{0.75ex}  $$
	Note that $f$ is smooth on $(-\pi/2,\pi/2)$ and  that 
	\vspace{0.75ex}  \begin{equation}
		\label{eq:f_properties}
	f(\tfrac \pi 3)  = 0 \quad \text{ and } \quad f'(\tfrac \pi 3)  = \frac{\sin(\pi/3)(3\cos(\pi/3)-(3\pi/3 - \pi)\sin(\pi/3)-3\cos(3\pi) )}{2(1 - \sin(\pi/3))^4} > 0.
	\vspace{0.75ex}  \end{equation}
	From the smoothness of $f$ and the product rule it follows that $u \mapsto h(u,v)$ and $v \mapsto h(u,v)$ are smooth.
	Also from (\ref{eq:h_expanded_2}), it follows that $h(0,1)=0$. Moreover
	by the Fundamental Theorem of Calculus and letting $u \to 0$, then
	\vspace{0.75ex}  $$
	h^{(1,0)}(u,1) = 96 \pi (\lmax^+ - u)^3 \int_{\pi/3}^{\arcsin(\tfrac{\sqrt 3}{2} + u)} f(\theta) \  \d \theta \longrightarrow 0
	\vspace{0.75ex}  $$
	and using (\ref{eq:f_properties}),
	\vspace{0.75ex}  \begin{align*}
	h^{(2,0)}(u,1) & = -288 \pi (\lmax^+ - u)^2 \int_{\pi/3}^{\arcsin(\tfrac{\sqrt 3}{2} + u)} f(\theta) \  \d \theta + \\
	& \quad \quad + 96\pi(\lmax^+ - u)^3 f(\arcsin(\tfrac{\sqrt 3}{2} + u))\bigg(\frac{\d}{\d u} \arcsin(\tfrac{\sqrt 3}{2} + u)\bigg) \longrightarrow 0,
	\vspace{0.75ex}  \end{align*}
	while using (\ref{eq:f_properties}) again yields that 
	\vspace{0.75ex}  \begin{align*}
		h^{(3,0)}(u,1) & = 476 \pi (\lmax^+ - u) \int_{\pi/3}^{\arcsin(\tfrac{\sqrt 3}{2} + u)} f(\theta) \  \d \theta - \\
		& \quad \quad -476 \pi (\lmax^+ - u) f(\arcsin(\tfrac{\sqrt 3}{2} + u))\bigg(\frac{\d}{\d u} \arcsin(\tfrac{\sqrt 3}{2}+u)\bigg) -  \\
		& \quad \quad -288 \pi(\lmax^+ - u)^2 f(\arcsin(\tfrac{\sqrt 3}{2} + u))\bigg(\frac{\d}{\d u} \arcsin(\tfrac{\sqrt 3}{2}+u)\bigg) + \\
		& \quad \quad + 96\pi(\lmax^+ - u)^3 f'(\arcsin(\tfrac{\sqrt 3}{2} + u))\bigg(\frac{\d}{\d u} \arcsin(\tfrac{\sqrt 3}{2}+u)\bigg)^2 + \\
		& \quad \quad + 96\pi(\lmax^+ - u)^3 f(\arcsin(\tfrac{\sqrt 3}{2} + u))\bigg(\frac{\d^2}{\d u^2} \arcsin(\tfrac{\sqrt 3}{2}+u)\bigg) \longrightarrow 384\pi (\lmax^\circ)^3 f'(\tfrac\pi 3) > 0.
		\vspace{0.75ex}  \end{align*}
		where we note that the squared first-order derivate of $u \mapsto \arcsin(\sqrt{3}/2 + u)$ evaluated at $u=0$ is equal to 4. Thus this verifies assumption $G$
		and partly assumption $H$ with $q=3$ as claimed. Now by (\ref{eq:h_expanded}), for any $v<1$, then
		\vspace{0.75ex}  \begin{equation}
		\tilde h(u,v)  = 96 \pi
	 \int_{\lmax^+ - u}^{\lmax^+} \int_{\arcsin\big(1-\tfrac{\ell}{1-uv}\big)}^{\arcsin(1-\ell)} \ell^3 f(\theta) \  \d \theta \d \ell, 
		\vspace{0.75ex}  \end{equation}
		and therefore $\tilde h(0,v)=0$. Moreover, by the Fundamental Theorem of Calculus and the Dominated Convergence Theorem,
		\vspace{0.75ex}  $$
		\tilde h^{(0,1)}(u,v)  = - 96 \pi
		\int_{\lmax^+ - u}^{\lmax^+}  \ell^4
		\frac{uf\bigg(\arcsin\big(1-\tfrac{\ell}{1-uv}\big)\bigg) }{\sqrt{1-\big(1-\tfrac{\ell}{1-uv}\big)^2}(1-uv)^2} \ \d \ell, 
		   \vspace{0.75ex}  $$
	Now similar to above $\ell^4$ or constants will not play a role in determining $q$, so for fixed $v$, define
	\vspace{0.75ex}  $$
	p(u)  =
	\int_{\lmax^+ - u}^{\lmax^+}
	\frac{uf\bigg(\arcsin\big(1-\tfrac{\ell}{1-uv}\big)\bigg) }{\sqrt{1-\big(1-\tfrac{\ell}{1-uv}\big)^2}(1-uv)^2} \ \d \ell, 
	   \vspace{0.75ex}  $$
	and note $p^{(j)}(0)=0$ if and only if $\tilde h^{(j,1)}(0,v)=0$. By letting $u \to 0$, then
	\vspace{0.75ex}  $$
	p'(u)  = 
	\frac{uf\big(\arcsin(b(u))\big) }{\sqrt{1-b(u)^2}(1-uv)^2} + \int_{\lmax^+ - u}^{\lmax^+} \frac{\d}{\d u}
	\frac{uf\bigg(\arcsin\big(1-\tfrac{\ell}{1-uv}\big)\bigg) }{\sqrt{1-\big(1-\tfrac{\ell}{1-uv}\big)^2}(1-uv)^2} \ \d \ell \longrightarrow 0,
	   \vspace{0.75ex}  $$
	where
	\vspace{0.75ex} \begin{equation}
		\label{eq:b_properties}
	b(u) = 1 - \frac{\lmax^+ - u}{1-uv} \to \tfrac{\sqrt 3}{2} > 0 \quad \text { and } \quad b'(u) = \frac{v\lmax^+ - 1}{(1-uv)^2} \to v\lmax^+ - 1 < 0,
	\vspace{0.75ex}  \end{equation}
	and using (\ref{eq:f_properties}), then
	\vspace{0.75ex}  \begin{align*}
	p''(u)  &= 
	2 \frac{f\big(\arcsin(b(u))\big) }{\sqrt{1-b(u)^2}(1-uv)^2} + u \frac{\d}{\d u} \frac{f\big(\arcsin(b(u))\big) }{\sqrt{1-b(u)^2}(1-uv)^2} +  \\[0.5em]
	& \quad \quad \quad \quad \quad \quad \quad \quad \quad \quad \quad + u \bigg(\frac{\d}{\d u}
	\frac{f\bigg(\arcsin\big(1-\tfrac{\ell}{1-uv}\big)\bigg) }{\sqrt{1-\big(1-\tfrac{\ell}{1-uv}\big)^2}(1-uv)^2}\bigg)\bigg\vert_{\ell = \lmax^+-u} + \\[0.5em]
	& \quad \quad \quad \quad \quad \quad \quad \quad \quad \quad \quad  + \int_{\lmax^+ - u}^{\lmax^+} \frac{\d^2}{\d u^2}
	\frac{uf\bigg(\arcsin\big(1-\tfrac{\ell}{1-uv}\big)\bigg) }{\sqrt{1-\big(1-\tfrac{\ell}{1-uv}\big)^2}(1-uv)^2} \ \d \ell \longrightarrow 0.
	   \vspace{0.75ex}  \end{align*}
	Now using (\ref{eq:f_properties}), (\ref{eq:b_properties}) and letting $u \to 0$, it follows that
	\vspace{0.75ex}  \begin{align*}
		\frac{\d}{\d u} & \frac{f\big(\arcsin(b(u))\big) }{\sqrt{1-b(u)^2}(1-uv)^2} = 
		\frac{f'\big(\arcsin(b(u))\big) \tfrac{-2b(u) b'(u)}{\sqrt{1-b(u)^2}}}{\sqrt{1-b(u)^2}(1-uv)^2}+ \\[0.5em]
		& \quad \quad +\frac{2f\big(\arcsin(b(u))\big)\bigg(\tfrac{b(u)(1-uv)^2}{1\sqrt{1-b(u)^2}}+v(1-uv)\sqrt{1-b(u)^2}\bigg)}{(1-b(u)^2)(1-uv)^4} 
		\longrightarrow -4\sqrt{3} f'(\tfrac{\pi}{3})(v \lmax^+ - 1) > 0,
		   \vspace{0.75ex}  \end{align*}
	and it follows (by direct calculation) that $p'''(0) > 0 $, which therefore completes the proof.
\end{proof}

\begin{proposition}[Assumption $H$ -- Vietoris-Rips]
	\label{prop:5}
	Let $d=2$, $k=3$ and $m=4$.
Using the Vietoris-Rips filtration,
then $\lmax^+(3,4)=1-\tfrac{1}{\sqrt{2}}$ and realized by an equilateral diamond with
side length $\tfrac{1}{\sqrt{2}}$.
Moreover, assumption $H$ holds with $q=4$ and therefore $\ell_{n,\alpha} \sim 1-\tfrac{1}{\sqrt{2}} - n(nr_n^2)^{-3/4}$. 
\end{proposition}
If one considers four points as an equilateral diamond with
side length $\sqrt{2}$, then the birthtime is $b=\tfrac{1}{\sqrt{2}}$, the deathtime is $r=1$  and the additive lifetime is $\ell= 1-\tfrac{1}{\sqrt{2}}$. In order to prove Proposition \ref{prop:5}, 
we will first prove that any configuration of four points that contains a 1-cycle with birthtime $b \leq \tfrac{1}{\sqrt{2}}+u$ and deathtime $r=1$
cannot be too far away from the equilateral diamond.
\begin{lemma}[Deviation from the diamond]
	\label{lem:beta_balls}
	Let $u\geq 0$ sufficiently small and assume $((0,0),(2,0),y_3)$ is a negative simplex in $((0,0),(2,0),y_3,y_4)$ 
	with birthtime less than $1/\sqrt{2}+u$ and that $\Vert y_3 - y_4 \Vert \geq 2$. Supposing $y_3$ lies
	in the positive half-plane, then $y_4$ lies in the negative half-plane. Moreover, there exist $\beta,\beta_0, \beta_1 >0$ (independent of $u$) such that
	$y_3 \in \big[ 1 - \beta_1 u, 1 + \beta_1 u \big] \times \big[ 1 - \beta_0 u, 1 + 2\sqrt{2}u \big]$, and
	\vspace{0.75ex} $$
	 \big\Vert y_3 - \big(1,1) \big\Vert \leq \beta u \quad \text{ and } \quad \big\Vert y_4 - \big(1,-1\big) \big\Vert \leq \beta u.
	\vspace{0.75ex} $$
\end{lemma}
\begin{proof}
	First, define the balls, 
	\vspace{0.75ex} \begin{equation}
		\label{eq:u_balls}
B_0^u  = B\big((0,0),\sqrt{2} + 2u\big) \quad \text{and} \quad B_2^u  = B\big((2,0),\sqrt{2}  + 2u\big). 
\vspace{0.75ex} \end{equation}
	We can find the intersection points $p_3(u)$ and $p_4(u)$ of the two circles $\partial B_1^u$ and $\partial B_2^u$ by solving
	\vspace{0.75ex} \begin{equation}
		\label{eq:circles}
		x^2 + y^2 = \big( \sqrt{2} + 2u \big)^2 \quad \text{ and } \quad (x-2)^2 + y^2 = \big( \sqrt{2} + 2u \big)^2,
	\vspace{0.75ex} \end{equation}
	which have the solutions
	\vspace{0.75ex} $$
	 p_{3}(u) = \bigg(1,  2\sqrt{\frac{1}{4} + \sqrt{2}u + u^2}  \bigg)\quad \text{ and } \quad p_{4}(u) = \bigg(1, -2\sqrt{\frac{1}{4} + \sqrt{2}u + u^2}  \bigg).
	\vspace{0.75ex} $$
	Note that if $\overline{B_0^u \cap B_2^u} \cap B(y_3,2)^c$ is non-empty, it must contain $p_4(u)$. Based on this observation, we now determine a lower bound on the 
	second coordinate (height) of $y_3$ under the condition that $\Vert y_3 - y_4 \Vert \geq 2$. To that end, we introduce the two points
	\vspace{0.75ex} $$
	 \hat p_{3}(u) = \big(1,  1 + 2\sqrt{2}u  \big)\quad \text{ and } \quad \hat p_{4}(u) = \big(1, -1 -2\sqrt{2}u \big),
	\vspace{0.75ex} $$
	Note that $\hat p_{3}(u)$ is slightly above $p_3(u)$ and $\hat p_{4}(u)$ is slightly below $p_4(u)$, respectively (see Figure \ref{fig:p_hat}). 
	
	\begin{figure}[h!]
    \begin{tikzpicture}[scale=1.5]
         \draw[black] (0,0) circle (1.6);
         \draw[black] (2,0) circle (1.6);
         \fill[black] (0,0) circle (1pt);
         \fill[black] (2,0) circle (1pt);

         \node[below] at (0,0) {$(0,0)$};
         \node[below] at (2,0) {$(2,0)$};
     
         \coordinate (P3) at (1,1.25); 
         \coordinate (P4) at (1,-1.25); 
         \fill[black] (P3) circle (1pt);
         \fill[black] (P4) circle (1pt);
     
         \node[below=3pt] at (P3) {$p_3$};
         \node[above=3pt] at (P4) {$p_4$};
     
         \coordinate (Phat3) at (1,1.5); 
         \coordinate (Phat4) at (1,-1.5); 
         \fill[black] (Phat3) circle (1pt);
         \fill[black] (Phat4) circle (1pt);
     
         \node[above] at (Phat3) {$\hat{p}_3$};
         \node[below] at (Phat4) {$\hat{p}_4$};

          \draw[dotted] (-1,-1.5) arc[start angle=180,end angle=0,radius=2.1];
     
    \end{tikzpicture}
    \caption{Illustration of the two circles $\partial B_0^{u}$ and $\partial B_2^{u}$ with intersection points $p_3 = p_3(u)$ and $p_4 = p_4(u)$. Additionally, the two
    approximations $\hat p_3$ and $\hat p_4$ are illustrated, and the ball $B(\hat p_4,2)$ is (partly) displayed as the dotted semi-circle.}
    \label{fig:p_hat}
\end{figure}
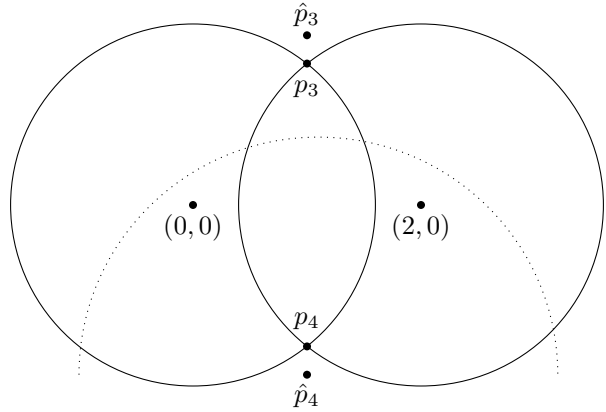

	We now find the intersection point 
of the two circles 
$\partial B_0^u$ and $\partial B(\hat p_4(u),2)$ inside the lens $\overline{B_0^u \cap B_2^u}$ 
by solving equations similar to (\ref{eq:circles}), which leads to the expression for $x$:
\vspace{0.75ex} \begin{equation}
	\label{eq:circles_x}
	x = y(1+2\sqrt{2}u) + 4\sqrt{2}u + 6u^2,
\vspace{0.75ex} \end{equation}
which subsequently leads to a quadratic equation in $y$ that has the solution (inside the lens):
\vspace{0.75ex} \begin{equation}
	\label{eq:circles_y}
	y = \frac{-4(1+2\sqrt{2}u)(2\sqrt{2}u+3u^2) + 2\sqrt{d}}{2((1+2\sqrt{2}u)^2 + 1)},
\vspace{0.75ex} \end{equation}
where the discriminant $d$ is given by
\vspace{0.75ex} \begin{equation}
	\label{eq:circles_d}
	d = 4(1+2\sqrt{2}u)^2(2\sqrt{2}u+3u^2)^2-4((1+2\sqrt{2}u)^2 + 1)((2\sqrt{2}u+3u^2)^2-(\tfrac{1}{\sqrt{2}} + u)^2).
\vspace{0.75ex} \end{equation}
Hence, by applying L'Hôpital's rule, it follows from equations (\ref{eq:circles_y}) and (\ref{eq:circles_d}) that $y \geq 1 - \beta_0 u$ for some $\beta_0 > 0$. 
Moreover, it follows that the second coordinate of $y_3$ must be at least $1 - \beta_0 u$. Inserting this into (\ref{eq:circles_x}), 
it follows that for some $\beta_1 > 0$ depending on $\beta_0$, $y_3$ must lie in the rectangle $R(u)$ defined as
\vspace{0.75ex} $$
R(u) = \big[ 1 - \beta_1 u, 1 + \beta_1 u \big] \times \big[ 1 - \beta_0 u, 1 + 2\sqrt{2}u \big],
\vspace{0.75ex} $$
and note that by the Pythagorean Theorem, this simultaneously proves that $\Vert y_3 - (1,1) \Vert \leq \beta_2 u$ for some $\beta_2 > 0$ depending on $\beta_0$ and $\beta_1$. 
By symmetry, we may repeat the argument for $y_4$, which completes the proof.
\end{proof}
\begin{proof}[Proof of Proposition \ref{prop:5}] Again, we split the proof into two parts as in the proof of Proposition \ref{prop:4}.

	\noindent \textit{Maximal lifetime:} Let $(y_1, y_2, y_3, y_4) \in \R^{2 \cdot 4}$ be 1-connected and assume, without loss of generality, 
	that \(y_1 = (0,0)\) and \(y_2 = (2r,0)\) for some \(r > 0\). Assume \(((0,0),(2r,0),y_3)\) is a negative simplex in \(((0,0),(2r,0),y_3,y_4)\)
	 with deathtime \(r \leq 1\) and lifetime greater than \(\lmax^+(3,4) - ru\). In particular, the lifetime is greater than \(r(1-1/\sqrt{2}-u)\), 
	 and \(\Vert y_3 - y_4 \Vert \geq 2r\). Now, scaling the four points by \(1/r\) gives that \(((0,0),(2,0),y_3)\) has deathtime equal to 1 and 
	 lifetime greater than \(1-1/\sqrt{2}-u\), which implies that the birthtime is less than \(1/\sqrt{2}+u\). 
	 Thus, it follows from Lemma \ref{lem:beta_balls}, by letting \(u \to 0\), that \(\lmax^+(3,4) = 1-1/\sqrt{2}\), 
	 and this lifetime is realized by \(y_3 = (1,1)\) and \(y_4 = (1,-1)\) as claimed.
	 \noindent \textit{Regularity:} By the definition of the differential quotient, it suffices to prove that 
	 \begin{enumerate}[label=(\roman*)]
		 \item $\displaystyle \lim_{u \to 0} \frac{h(u,1)}{u^5} \in (0,\infty) $, \\[0.25ex]
		 \item $\displaystyle \lim_{u \to 0} \frac{\tilde h^{(0,1)}(u,v)}{u^5} \in (0,\infty) $ for any $v < 1$,
	 \end{enumerate}
	 in order to conclude $q=5$ works in assumption $H$. From the definition of $h$ in (\ref{eq:hdef_int}),
	 \vspace{0.75ex} \begin{align*}
	 h(u,v)  = \iiint
	 \one\big\{(0,y_2,y_3,y_4) \in F_{1-v}^1, \ \ell_{(0,y_2,y_3),(0,y_2,y_3,y_4)} \geq \lmax^+- u\big\}  \ \d y_2 \d y_3 \d y_4.
	 \vspace{0.75ex}  \end{align*}
	 Applying the translation with the vector $(y_1,y_1,y_1,y_1)$ and Tonelli's theorem, it follows that
	 \vspace{0.75ex} \begin{align*}
	 h(u,v)  = \iiiint
	 \one\big\{(y_1,y_2,y_3,y_4) \in F_{1-v}^1, \ & \ell_{(y_1,y_2,y_3),(y_1,y_2,y_3,y_4)} \geq \lmax^+ - u \big\} \times \\
	 & \quad \times \one \bigg\{ \frac{1}{4} \sum_{j=1}^4 y_i \in [0,1]^2 \bigg\}  \ \d y_1 \d y_2 \d y_3 \d y_4.
	 \vspace{0.75ex}  \end{align*}
	 By \cite[Lemma 4.4]{KM18}, since there are four points, the created 1-cycle is isomorphic to a 2-dimensional cross-polytope, i.e., a diamond or rhombus.
	 Thus, it automatically follows that $(0,y_2,y_3,y_4)$ is 1-connected. Moreover, since $(y_1,y_2,y_3)$ is negative, all pairwise distances
	 between $y_1$, $y_2$, and $y_3$ are less than $2 r_{(y_1,y_2,y_3)}$. As one diagonal distance in the diamond
	 must be larger than $2 r_{(y_1,y_2,y_3)}$, this leaves exactly three disjoint options, namely the distance between $y_4$ and $y_i$ for $i=1,2,3$. Thus, by symmetry
	 and choosing $i=3$, meaning the deathtime $r_{(y_1,y_2,y_3)}=\Vert y_1 - y_2 \Vert /2$, then
	 \vspace{0.75ex} \begin{align*}
	 h(u,v)  = 3 &\iiiint
	 \one\big\{(y_1,y_2,y_3) \in \mc N((y_1,y_2,y_3,y_4)), \ \ell_{(y_1,y_2,y_3),(y_1,y_2,y_3,y_4)} \geq \lmax^+ - u \big\} \times \\
	 & \quad \times \one \bigg\{ 1-v \leq \tfrac{\Vert y_1 - y_2 \Vert}{2} \leq 1,  \ \frac{1}{4} \sum_{j=1}^4 y_i \in [0,1]^2, \Vert y_3 - y_4 \Vert \geq \Vert y_2 - y_1 \Vert \bigg\}  \ \d y_1 \d y_2 \d y_3 \d y_4.
	 \vspace{0.75ex}  \end{align*}
	 Translating back to $y_1 = 0$ and using that the lifetimes are additive,
	 \vspace{0.75ex} \begin{align*}
	 h(u,v)  = 3 \iiint
	  \one\big\{(0,y_2,y_3) & \in \mc N((0,y_2,y_3,y_4)), \ b_{(0,y_2,y_3),(y_1,y_2,y_3,y_4)} \leq \Vert y_2 \Vert - \lmax^+ + u \big\} \times \\
	 & \quad \quad \quad \times \one \big\{ 1-v \leq \tfrac{\Vert y_2 \Vert}{2} \leq 1,  \ \Vert y_3 - y_4 \Vert \geq \Vert y_2 \Vert \big\}  \ \d y_2 \d y_3 \d y_4.
	 \vspace{0.75ex}  \end{align*}
	 Now transforming into polar coordinates such that $y_2 = r(\cos(\theta),\sin(\theta))$, then rotating $y_2$, $y_3$, and $y_4$ by $-\theta$,
	 and changing $y_3$ and $y_4$ back into Cartesian coordinates, it follows that 
	 \vspace{0.75ex} \begin{align*}
	 h(u,v)  = 24 \pi   \int_{1-v}^1 r \iint
	 \one\big\{&(0,(2r,0),y_3) \in \mc N((0,(2r,0),y_3,y_4))\big\} \times \\
	 & \ \times \one \big\{\ b_{(0,y_2,y_3),(y_1,y_2,y_3,y_4)} \leq r -\lmax^+ + u, 
	  \ \Vert y_3 - y_4 \Vert \geq 2r \big\}  \ \d y_3 \d y_4 \d r.
	 \vspace{0.75ex}  \end{align*}
	 Substituting $y_3$ and $y_4$ with $r y_3$ and $r y_4$ as well as introducing $u_r  = \max\{\lmax^+ -\tfrac{\lmax^+-u}{r},0\}$,
	 \vspace{0.75ex} \begin{align*}
	 h(u,v)  = 24 \pi   \int_{1-v}^1 r^5 \iint
	 \one\big\{&(0,(2,0),y_3) \in \mc N((0,(2,0),y_3,y_4))\big\} \times \\
	 & \ \times \one \big\{b_{(0,y_2,y_3),(y_1,y_2,y_3,y_4)} \leq \tfrac{1}{\sqrt{2}} + u_r, 
	  \ \Vert y_3 - y_4 \Vert \geq 2 \big\}  \ \d y_3 \d y_4 \d r,
	 \vspace{0.75ex}  \end{align*}
	 where we have also plugged in the value of $\lmax^+$ as found above.
	 Now, introduce
	 \vspace{0.75ex}  $$
	 C_3^{u_r} = \overline{B_0^{u_r} \cap B_2^{u_r}} \cap B((1,1),\beta {u_r}) \quad \text{ and } \quad    A(y_3) = \overline{B_0^{u_r} \cap B_2^{u_r}} \cap B(y_3,2)^c,
	 \vspace{0.75ex}  $$
	 where $B_0^u$ and $B_2^u$ are defined as in (\ref{eq:u_balls}).
	 By Lemma \ref{lem:beta_balls}, 
	 \vspace{0.75ex}  $$
	 h(u,v)  = \int_{1-v}^{1} \int_{C_3^{u_r}} \vol(A(y_3)) \  \d y_3 \d r.
	 \vspace{0.75ex}  $$	 
	 The idea is now to approximate the admissible area $A(y_3)$ by an isosceles right triangle, which we do through several steps. For a triangle with vertices $P$, $Q$, and $R$, 
	 we let $\angle PQR$ denote the interior angle at $Q$ (i.e., defined analogously as in (\ref{eq:triangle_angles})). 
	 First, if the intersection of $\partial B(y_3,2)$ with $\partial B_0^{u_r}$ and $\partial B_2^{u_r}$ is empty, then $\vol(A(y_3))=0$ by definition of $A(y_3)$.
	 Therefore, assuming without loss of generality the intersection is non-empty, let
	 $Q_0$ and $Q_2$ denote the intersections of $\partial B(y_3,2)$ with $\partial B_0^{u_r}$ and $\partial B_2^{u_r}$ inside the lens $\overline{B_0^{u_r} \cap B_2^{u_r}}$, respectively. Let $Q_1$
	 denote the point on $\partial B(y_3,2)$ that lies inside the lens and is equidistant to $Q_0$ and $Q_2$. 
	 Consider points $\tilde Q_0$ on $\partial B_0^{u_r}$ and $\tilde Q_2$ on $\partial B_2^{u_r}$ and both inside the lens such that $\angle \tilde Q_2 Q_1 y_3$ 
	 and $\angle \tilde Q_0 Q_1 y_3$ are right angles, see Figure \ref{fig:triangle_construction}. Let
	 \vspace{0.75ex} $$
	 \tilde A(y_3) = A(y_3) \cap \text{conv}(\tilde Q_0,\tilde Q_2,y_3)^c.
	 \vspace{0.75ex} $$
	 By Lemma \ref{lem:beta_balls} and symmetry, it follows that
	 \vspace{0.75ex} \begin{equation}
		 \label{eq:first_step}
	 \vol\big( A(y_3) \triangle \tilde A(y_3)\big) \leq 2 \beta^2u_r^2 \sin(\angle \tilde Q_2 y_3 Q_1) \cos(\angle \tilde Q_2 y_3 Q_1) \leq 2 \beta^3u_r^3 \in O(u_r^3).
	 \vspace{0.75ex} \end{equation}
	 In other words, we have changed the "curved top" of $A(y_3)$ to a "straight and slanted top." Next, we argue that we can go from "slanted" to "horizontal". 
	 Consider the point $Q_3$ in the lens such that $\angle Q_1 Q_3 y_3$ is a right angle and write $Q_1 = (q_{1,1},q_{1,2})$. Let
	 \vspace{0.75ex} $$
	 \overline A(y_3) = A(y_3) \cap (\R \times (-\infty,q_{1,2}]).
	 \vspace{0.75ex} $$
	 Again, by Lemma \ref{lem:beta_balls} and symmetry, it follows that
	 \vspace{0.75ex} \begin{equation}
		 \label{eq:second_step}
	 \vol\big( \tilde A(y_3) \triangle \overline A(y_3)\big) \leq 2 \beta^2 u_r^2 \sin(\angle \tilde Q_3 y_3 Q_1) \cos(\angle \tilde Q_3 y_3 Q_1) \in O(u_r^3).
	 \vspace{0.75ex} \end{equation}
	 Finally, we argue that we can go from a "curved bottom" to a "right-angled bottom." Let $P_4$ denote the bottom intersection point of $\partial B_0^{u_r}$ and $\partial B_2^{u_r}$, 
	 and let $Q_4$ and $Q_5$ denote two points on the horizontal line $y=q_{1,2}$ such that
	 $\angle Q_4P_4Q_5$ is a right angle, see Figure \ref{fig:triangle_construction}. Consider the triangle,
	 \vspace{0.75ex} $$
	 \hat A(y_3) = \text{conv}(P_4,Q_4,Q_5),
	 \vspace{0.75ex} $$
	 Take $Q_6$ on the horizontal line through $P_4$ such that $\angle Q_4Q_6P_4$ is a right angle. 
	 As above,
	 \vspace{0.75ex} \begin{equation}
		 \label{eq:third_step}
	 \vol\big( \overline A(y_3) \triangle \hat A(y_3)\big) \leq 2 \beta^2 u_r^2 \sin(\angle \tilde Q_4 P_4 Q_6) \cos(\angle \tilde Q_4 P_4 Q_6) \in O(u_r^3).
	 \vspace{0.75ex} \end{equation}	 

	\begin{figure}[h!]
    \centering

    \begin{subfigure}[b]{0.32\textwidth}
        \centering
        \begin{tikzpicture}[scale=1]
            \draw[gray] (0,0) arc[start angle=270, end angle=450, radius=2];
            \draw[gray] (2,0) arc[start angle=270, end angle=90, radius=2];
            \draw[gray,dotted] (-1.5,3) arc[start angle=180,end angle=360,radius=2];
            \coordinate (Y3) at (0.5,3);
         \fill[black] (Y3) circle (1pt);
         \node[below=3pt] at (Y3) {$y_3$};

         \coordinate (Q1) at (1.2,1.14);
         \fill[gray] (Q1) circle (1pt);
         \node[gray, above=3pt] at (Q1) {$Q_1$};
        
         \coordinate (Q0) at (1.92, 1.4);
         \fill[black] (Q0) circle (1pt);
         \node[below, right] at (Q0) {$\tilde Q_0$};

         \coordinate (Q2) at (0.39, 0.84);
         \fill[black] (Q2) circle (1pt);
         \node[below=2pt] at (Q2) {$\tilde Q_2$};
         
         \draw[black] (Q0) -- (Q2);

        \end{tikzpicture}
    \end{subfigure}
    \begin{subfigure}[b]{0.32\textwidth}
        \centering
        \begin{tikzpicture}[scale=1]
            \draw[gray] (0,0) arc[start angle=270, end angle=450, radius=2];
            \draw[gray] (2,0) arc[start angle=270, end angle=90, radius=2];
            \draw[gray,dotted] (-1.5,3) arc[start angle=180,end angle=360,radius=2];
            \coordinate (Y3) at (0.5,3);
         \fill[black] (Y3) circle (1pt);
         \node[below=3pt] at (Y3) {$y_3$};

         \coordinate (Q1) at (1.2,1.14);
         \fill[black] (Q1) circle (1pt);
         \node[black, above=3pt] at (Q1) {$Q_1$};
        
         \coordinate (Q0) at (1.92, 1.4);

         \coordinate (Q2) at (0.39, 0.84);
         
         \draw (0.2,1.14) -- (1.8,1.14);
        \end{tikzpicture}
    \end{subfigure}
    \begin{subfigure}[b]{0.32\textwidth}
        \centering
        \begin{tikzpicture}[scale=1]
            \draw[gray] (0,0) arc[start angle=270, end angle=450, radius=2];
            \draw[gray] (2,0) arc[start angle=270, end angle=90, radius=2];
            \draw[gray,dotted] (-1.5,3) arc[start angle=180,end angle=360,radius=2];
            \coordinate (Y3) at (0.5,3);
         \fill[black] (Y3) circle (1pt);
         \node[below=3pt] at (Y3) {$y_3$};
    
        \coordinate (Q0) at (1.92, 1.4);
        \coordinate (Q2) at (0.39, 0.84);
        
         \coordinate (Q4) at (-0.1,1.14);
         \fill[black] (Q4) circle (1pt);
         \node[black, left=2pt] at (Q4) {$Q_4$};

         \coordinate (Q5) at (2.1,1.14);
         \fill[black] (Q5) circle (1pt);
         \node[black, right=2pt] at (Q5) {$Q_5$};

         \coordinate (P4) at (1,0.28);
         \fill[black] (P4) circle (1pt);

         \draw (Q5) -- (Q4) -- (P4) -- cycle;
        \end{tikzpicture}
    \end{subfigure}

    \caption{From left to right: The construction of $\tilde A(y_3)$, $\overline A(y_3)$, and $\hat A(y_3)$
    as well as an illustration of the points $\tilde Q_0$, $Q_1$, $\tilde Q_2$, $Q_4$, $Q_5$
    in the proof of Proposition \ref{prop:5}.}
    \label{fig:triangle_construction}
\end{figure}
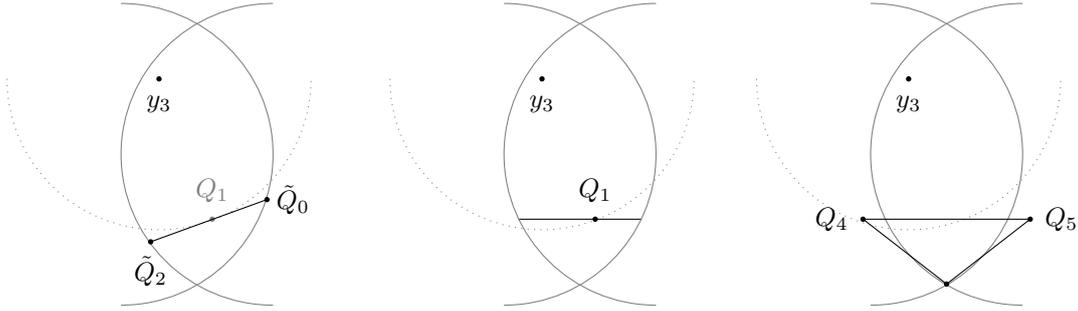

	Thus, we now write 
	\vspace{0.75ex} $$
	h(u,v)  = \int_{1-v}^{1} r^5 \int_{C_3^{u_r}} 
	 \vol\big(\hat A(y_{3})\big) \  \d y_3 \d r  + \int_{1-v}^{1} \int_{C_3^{u_r}}  \vol\big(\hat A(y_{3}) \triangle A(y_3) \big) \  \d y_3 \d r =: E_1(u,v) + E_2(u,v).
	\vspace{0.75ex} $$
	By (\ref{eq:first_step}), (\ref{eq:second_step}), and (\ref{eq:third_step}), it suffices to show (i) and (ii) with $h(u,v)$ replaced by $E_1(u,v)$. For $y_{3,2}$ fixed, it follows
	from the intersections of the horizontal line $y=y_{3,2}$ with $\partial B_0^u$ and $\partial B_2^u$ that
	\vspace{0.75ex} $$
	(y_{3,1},y_{3,2}) \in C_3^u \iff 2 - \sqrt{(\sqrt{2}+2u)^2 - y_{3,2}^2} \leq y_{3,1} \leq \sqrt{(\sqrt{2}+2u)^2 - y_{3,2}^2}.
	\vspace{0.75ex} $$
	Then it follows that 
	\vspace{0.75ex} $$
	E_1(u,v) = \int_{1-v}^{1} r^5 \int_{2-2\sqrt{\frac{1}{4} + \sqrt{2}u_r + u_r^2}}^{2\sqrt{\frac{1}{4} + \sqrt{2}u_r + u_r^2}} 
	\int_{2- \sqrt{(\sqrt{2}+2u_r)^2 - y_{3,2}^2}}^{\sqrt{(\sqrt{2}+2u_r)^2 - y_{3,2}^2}}
	\vol\big(\hat A((y_{3,1},y_{3,2}))\big) \  \d y_{3,1} \d y_{3,2} \d r .
	\vspace{0.75ex} $$
	In the above construction, using analogous trigonometric arguments, the dependence of $\hat A((y_{3,1},y_{3,2}))$ on $y_{3,1}$ is of order $O(u_r^3)$;
	hence, we may assume without loss of generality that $\vol\big(\hat A((y_{3,1},y_{3,2}))\big)$ is independent of $y_{3,1}$. Since the
	area of an isosceles right triangle of height $\Delta$ is $\Delta^2/2$,
	\vspace{0.75ex} \begin{equation}
		\label{eq:sanity_check}
	\begin{aligned}
	E_1(u,v) = \int_{1-v}^{1} \frac{r^5}{2} & \int_{2-2\sqrt{\frac{1}{4} + \sqrt{2}u_r + u_r^2}}^{2\sqrt{\frac{1}{4} + \sqrt{2}u_r + u_r^2}} 
	\bigg(y_{3,2} - 2 + 2\sqrt{\frac{1}{4} + \sqrt{2}u_r + u_r^2}\bigg)^2  \times \\ 
	& \quad \quad \quad \times \bigg( \sqrt{(\sqrt{2}+2u_r)^2-y_{3,2}^2} - (2- \sqrt{(\sqrt{2}+2u_r)^2 - y_{3,2}^2})\bigg)
	\ \d y_{3,2} \d r. 
	\end{aligned}
	\vspace{0.75ex} \end{equation}
	As a brief sanity check, observe that as $u \to 0$, the domain of $y_{3,2}$ in (\ref{eq:sanity_check}) approaches $[1,1]$
	and hence the integrand approaches $(1-2+1)^2(\sqrt{2-1}-2+\sqrt{2-1})=0$;
	so $E_1(u,v) \to 0$. Let 
	\vspace{0.75ex} $$
	f(u_r) = \int_{2-2\sqrt{\frac{1}{4} + \sqrt{2}u_r + u_r^2}}^{2\sqrt{\frac{1}{4} + \sqrt{2}u_r + u_r^2}} 
	\bigg(y_{3,2} - 2 + 2\sqrt{\frac{1}{4} + \sqrt{2}u_r + u_r^2}\bigg)^2
	 \bigg( 2\sqrt{(\sqrt{2}+2u_r)^2-y_{3,2}^2} - 2 \bigg) \ \d y_{3,2}. 
	\vspace{0.75ex} $$ 
	By L'Hôpital's rule, we see that 
	\vspace{0.75ex} \begin{align*}
		\lim_{u_r \to 0} & \frac{4\sqrt{\frac{1}{4} + \sqrt{2}u_r + u_r^2}-2}{u_r} = 2 \sqrt{2}, \\[0.5ex]
		\lim_{u_r \to 0} & \frac{2\sqrt{(\sqrt{2}+2u_r)^2-1} - 2 }{u_r} = 4\sqrt{2}, \\[0.5ex]
		\lim_{u_r \to 0} & \frac{\bigg(1 - 2 + 2\sqrt{\frac{1}{4} + \sqrt{2}u_r + u_r^2}\bigg)^2}{u_r^2} = \sqrt{2}.
	\vspace{0.75ex} \end{align*} 
	Thus, there exists a $c>0$ such that $f(u_r)/u_r^4 \to c$ as $u_r \to 0$. Furthermore, for $u_r$ sufficiently small,
	\vspace{0.75ex} \begin{equation}
		\label{eq:bound_on_f}
	(c-\e) u_r^4\leq f(u_r) \leq (c+\e)u_r^4,
	\vspace{0.75ex} \end{equation}
	for some $\e >0$. Finally, we are prepared to prove the statements regarding the regularity of $h$. First, regarding point (i), setting $v=1$ and plugging in the definition of $u_r$, we see that 
	\vspace{0.75ex} $$
	\frac{1}{2} \int_{0}^1 r^5 (c\pm \e) u_r^4 \ \d r =  \frac{c\pm \e}{2} \int_{1-u/\lmax^+}^1 r (u-(1-r)\lmax^+)^4\ \d r. 
	\vspace{0.75ex} $$
	Applying the substitution $r = 1 - uw$, it follows that
	\vspace{0.75ex} \begin{equation}
		\label{eq:one_order}
		\frac{1}{2}  \int_{0}^1 r^5 (c\pm \e) u_r^4 \ \d r =  u^5\frac{c\pm \e}{2}  \int_{0}^{1/\lmax^+} (1-uw) (1-w\lmax^+)^4 \ \d w.
	\vspace{0.75ex} \end{equation}
	By combining (\ref{eq:sanity_check}), (\ref{eq:bound_on_f}), and (\ref{eq:one_order}), we complete the proof of (i). Next, let $v < 1$. By the Fundamental Theorem of Calculus,
	\vspace{0.75ex} \begin{equation}
		\label{eq:two_order}
		\frac{\d}{\d v} \frac{1}{2}  \int_{1-uv}^1 r^5 (c\pm \e) u_r^4 \ \d r = -u^5\frac{c\pm \e}{2} (1-uv)(1-v\lmax^+)^4,
	\vspace{0.75ex} \end{equation} 
	so by combining (\ref{eq:sanity_check}), (\ref{eq:bound_on_f}), and (\ref{eq:two_order}), we complete the proof of (ii).
\end{proof}






\pagebreak
\section{Examples}
\label{sec:6}

This section is dedicated to demonstrating some applications of 
Theorems \ref{thm:4}, \ref{thm:1}, \ref{thm:2}, and \ref{thm:3},
and, in particular, to studying the distributional properties of the largest lifetime.
Note that in all examples, when situated in the $m$-sparse regime, we only consider
models where Propositions \ref{prop:assumption_P} and \ref{prop:assumption_U} guarantee that assumptions
$P$ and $U$ are satisfied.
\begin{example}[\textit{"Constant intensity"}]
	\label{ex:1}
Let $\PP_n$ denote a stationary Poisson point process with
constant intensity $n \geq 1$ supported on $W=\T^2$. Suppose further that the
\v Cech filtration is built on top of $\PP_n$ and that the lifetimes are
measured multiplicatively. Choose $k=3$, i.e.,
we are interested in 1-cycles with a large lifetime
that are killed by solid triangles, and let $\alpha = 1$. Now choosing $\ell_{n,1}$ according to assumption $T'$, Theorem \ref{thm:4} with $\e = 1/36 - 1/37 > 0$ gives that
\vspace{0.75ex} $$
\dkr(\xi_n^1, \zeta^1)  = O(n^{-1/37}),
\vspace{0.75ex} $$
where $\zeta^1$ is a stationary Poisson point process with intensity 1. Alternatively, to put ourselves in the 3-sparse setting,
we let $W=[0,1]^2$ and introduce a deathtime bound $r_n> 0$, which satisfies that
$n(nr_n^2)^{2} \rightarrow \infty$, yet  $n(nr_n^2)^{3} \rightarrow 0$.
As $r_n = n^{-\beta}$ satisfies this for $\beta \in (2/3,3/4)$, we can choose $n^{-7/10}$ as the deathtime bound.
Then, choosing $\ell_{n,1}$ in accordance with
assumption $T$ and
letting $\zeta^2 = \zeta^1$, it follows
by Theorem \ref{thm:1}
that
\vspace{0.75ex} $$
\dkr(\xi_n^2, \zeta^2) \in O\big( n(n(n^{-7/10})^2)^3 \big) = O(n^{-1/5}).
\vspace{0.75ex} $$
Moreover, by Theorem \ref{thm:2}, Proposition 5.4, and Remark \ref{rem:sparse_extension}(5), it follows that
\vspace{0.75ex} $$
\dkr(\xi_n^3, \zeta^3) \in O\big( n(n(n^{-7/10})^2)^3 + (n(n^{-7/10})^{4/3})^{-1} \big) = O(n^{-1/15}),
\vspace{0.75ex} $$
where $\zeta^3$ is a Poisson point process with intensity function $(y,u) \mapsto 3u^2$. Thus, the probability of a certain (scaled) lifetime $\hat \ell$ is
proportional to the squared deviation of $\hat \ell$ to $\lmax^*(3,3)=2/\sqrt 3 \approx 1.155$. 

\end{example}

\begin{example}[\textit{"Deathtime bound"}]
	\label{ex:2}
Let $r_n = n^{-\beta}$ for some $\beta > 0$ as in Example \ref{ex:1}. Based on the verified instances of assumption $U$
in Proposition \ref{prop:assumption_U}, we will compute the possible values of $\beta$, as well
as the value of $\beta$ when $d=2$ and $m$ varies freely.

\noindent \textit{\v Cech:}
Let $d\geq 2$ and $m=d+1$. Then, it holds that
\vspace{0.75ex} \begin{equation}
	\label{eq:r_n_choice_1}
	n(nr_n^d)^{d} \rightarrow \infty \quad \text{yet} \quad n(nr_n^d)^{d+1} \rightarrow 0 \quad \iff \quad \beta \in \bigg(\frac{d+2}{d(d+1)},\frac{d+1}{d^2}\bigg).
\vspace{0.75ex} \end{equation}

\noindent \textit{Vietoris-Rips:} Let $d \geq 2$ and $m=2d$. Then, it holds that
\vspace{0.75ex} \begin{equation}
	\label{eq:r_n_choice_2}
	n(nr_n^d)^{2d-1} \rightarrow \infty \quad \text{yet} \quad n(nr_n^d)^{2d} \rightarrow 0 \quad \iff \quad \beta \in \bigg(\frac{2d+1}{2d^2},\frac{2}{2d-1}\bigg).
\vspace{0.75ex} \end{equation}

\noindent \textit{2-dimensional:} Let $d=2$ and $m \geq 3$. Then, it holds that
\vspace{0.75ex} \begin{equation}
	\label{eq:r_n_choice_3}
	n(nr_n^2)^{m-1} \rightarrow \infty \quad \text{yet} \quad n(nr_n^2)^{m} \rightarrow 0 \quad \iff \quad \beta \in \bigg(\frac{m+1}{2m},\frac{m}{2(m-1)}\bigg).
\vspace{0.75ex} \end{equation}
\end{example}

\begin{example}[\textit{"Gaussian intensity"}]
	\label{ex:3}
Let $\kappa$ denote the $d$-dimensional standard Gaussian density,
\vspace{0.75ex} $$
\kappa(y) = (2\pi)^{-\frac{d}{2}}
\exp\big(-\frac{1}{2}\Vert y \Vert^2\big).
\vspace{0.75ex} $$
Suppose we want to investigate the Poisson point process $\PP_{n\kappa}$ in $\R^d$ with 
intensity function $n \cdot \kappa$
using the Vietoris-Rips filtration.
For $k=d-1$ and $m=2d$, then by Example \ref{ex:2}, we may choose 
$ \beta = \frac{8d^2-1}{4d^2(2d-1)}$ and
 $r_n = n^{-\beta}$. Moreover, for any $\alpha > 0$, choose $\ell_{n,\alpha}$ 
according to assumption $T$. If $\zeta$ denotes a Poisson point process
with intensity function $ y \mapsto 
\alpha \cdot (2\pi)^{-d^2}\exp\big(-d\Vert y \Vert^2\big),
\vspace{0.75ex} $
then by Theorem \ref{thm:1}, it follows that
\vspace{0.75ex} $$
\dkr(\xi_n^2, \zeta^2) \in O(\rho_{n,m+1}) 
= O\big(n^{1+2d - (8d^2-1)/(2(2d-1))} \big). \\[0.1em]
\vspace{0.75ex} $$
In particular, when $d=2$, then $\beta = 31/48$ and 
$\dkr(\xi_n^2, \zeta^2) \in O(n^{-1/6})$.
\end{example}

\begin{example}[\textit{"Largest lifetime"}]
	\label{ex:4}
We will now investigate what the Poisson approximation results tell us about the distributional properties of the largest lifetime in both regimes.  

\noindent \textit{The $m$-sparse regime:} Suppose that all assumptions in Theorem \ref{thm:2} hold. Define the largest (scaled) lifetime $\ell_n^{(1)} $ as
\vspace{0.75ex} \begin{equation}
	\label{eq:largest_lifetime_def}
\ell_n^{(1)} = 
\max_{\xx \in \mc N(\PP_{n\kappa})} \hat \ell_{\xx,\PP_{n\kappa}}.
\vspace{0.75ex} \end{equation}
Then for any $t \geq  0$,
\vspace{0.75ex} $$
P\big( (\lmax^\circ - \hat \ell_n^{(1)})/u_{n,\alpha} \geq t \big) 
= P\big( \min_{\xx \in \mc N(\PP_{n\kappa}) \cap \PP^{(k)}}
(\lmax^\circ -  \hat \ell_{\xx,\PP_{n\kappa}})/u_{n,\alpha} \geq t \big) 
= P\big(\xi_n^3(W \times [0,t]) = 0\big).
\vspace{0.75ex} $$
By Theorem \ref{thm:2}, it follows that as $n \to \infty$,
	\vspace{0.75ex} $$
	P\bigg( \frac{\lmax^\circ - \ell_n^{(1)}}{u_{n,\alpha}} \geq t \bigg) 
	\longrightarrow \exp\big(- \alpha \int_0^t \int_W \kappa^{m}(y) qu^{q-1} \ \d y \d u\big)
	= \exp \big(-\alpha \gamma t^q \big),
	\vspace{0.75ex} $$
	where $\gamma = \int_W \kappa^{m}(y) \d y $. Thus, we conclude that
	\vspace{0.75ex} $$
	\frac{\lmax^\circ -  \ell_n^{(1)}}{u_{n,\alpha}}
	\overset{\sim}{\longrightarrow} \text{Weibull}((\gamma \alpha)^{-1/q},q) \quad \text{ as } n \to \infty,
	\vspace{0.75ex} $$
where $\overset{\sim}{\longrightarrow} $ denotes convergence in distribution. In other words, the deviation from the maximal lifetime asymptotically follows a Weibull law in the $m$-sparse setting. 
If $\kappa$ is as in 
Example \ref{ex:3} and $\alpha = 1$,
then the convergence is to a 
$\text{Weibull}(2(2\pi^3)^{1/4},4) \approx \text{Weibull}(5.6,4)$.

\noindent \textit{The unbounded regime:} Relying on Remark \ref{rem:unbounded_extension}(4), we can actually extend the ideas above to the unbounded regime.
	Letting $\ell_n^{(1)}$ denote the largest lifetime again, it follows for any $\alpha > 0$ that 
	\vspace{0.75ex} $$
	P\big( n^3 v(\ell_n^{(1)})  \geq \alpha \big) 
	= P\big( \Xi_n(\T^2 \times [0,\alpha]) = 0\big) \longrightarrow \exp(-\alpha),
	\vspace{0.75ex} $$
	as $n \to \infty$, which means the transformed maximal lifetime $n^3 v(\ell_n^{(1)})$ follows a Weibull(1,1), or Exponential(1), law asymptotically.
\end{example}
\begin{example}[\textit{"The persistence diagram and deathtimes"}]
	\label{ex:6}
	The persistence diagram is a tool in TDA used to summarize the lifetimes of cycles in a point cloud. 
Each point in a 2D diagram corresponds to a cycle, represented as a pair $(\ell, r)$, where $\ell$ is the \textit{lifetime} and $r$ is the \textit{deathtime}. 
Taking a view towards statistical applications, provided the assumptions hold, Theorem \ref{thm:3} tells us that the deathtimes of all large-lifetime cycles (i.e., $\ell \geq \lmax - u_{n,\alpha}$) 
are independent and identically distributed (i.i.d.). Thus, when visualizing the persistence diagram, the right section of the lifetime axis (i.e., to the right of $\lmax - u_{n,\alpha}$) should 
asymptotically exhibit this behavior: subdividing the axis into intervals, the deathtime values should be i.i.d. 
The main issue, of course, is that the value of $\lmax - u_{n,\alpha}$ is unknown, and we only have a rough asymptotic estimate through a Taylor approximation. 
Alternatively, one can subdivide the point cloud into sections and, in each section, consider only the deathtimes of the largest and second-largest lifetimes. 
This allows us to investigate whether these deathtimes, viewed as two samples, are i.i.d. as provided by Theorem \ref{thm:3}.

\end{example}

	
	
	
	

	\section{Simulations}
	\label{sec:7}

	This section is dedicated to supplementing the theoretical results with simulations and exploring some stated conjectures in Section 2 regarding the values of \(\lmax\) and \(q\). 
Finally, we illustrate how to apply Theorem \ref{thm:3} to devise a statistical test based on deviations in the deathtimes, as alluded to in Example \ref{ex:6}. 
To keep this section somewhat concise, we will only carry out these simulations for the \v Cech filtration and additive lifetimes, as well as work with a moderate number of iterations.
The cases of the Vietoris-Rips filtration and multiplicative lifetimes are left as future work, and we think it could be interesting to combine this with larger-scale simulation
studies, which try to investigate the behavior of the threshold $u_{n,\alpha}$ and, in particular, how $u_{n,\alpha}$ depends on $n$ and $\alpha$.
All simulations are carried out using the programming language \texttt{R}. In particular, the \v Cech filtration is built as the Alpha filtration (see Section \ref{sec:3.3})
using the \texttt{R}-package \texttt{TDA}, 
where we have used the \texttt{GUDHI} and \texttt{Dionysus} libraries. All source code is available through the link here:
\href{https://github.com/Statolaj/largelifetimecyclescode.git}{\texttt{GitHub repository}}.

\begin{experiment}[\textit{Optimal point configuration}]
	\label{exp:1}
As already discussed in Section \ref{sec:5}, it is not immediately clear what shape \(m\) points containing a \((k-2)\)-cycle with the largest possible lifetime \(\lmax(m)\) 
has or even what the value of \(\lmax(m)\) is. In the \v Cech filtration when \(d=2\), \(k=3\), and the deathtime is bounded by 1, we conjecture that the configuration
reaching maximal lifetime will 
be \(m\) points equally spaced on the unit circle.
This conjecture is consistent with Proposition \ref{prop:4}. As we know that the lifetime 
obtained by equally spaced points must be a lower bound for the maximal lifetime, the conjecture is actually that this is, in fact, an upper bound too.
In Table \ref{fig:1}, we list the conjectured upper bound for \(m=4,5,6,7,8\) alongside the proven case of \(m=3\).

\begin{table}[h!]
	\centering
	\begin{tabular}{|c|c|c|c|c|c|c|}
		\hline
		Value of \(m\) & 3  & 4 & 5 & 6 & 7 & 8 \\ \hline
		Value of maximal lifetime & 0.134 & 0.293 & 0.412 & 0.500 & 0.566 & 0.617 \\ \hline
	\end{tabular}
	\caption{Conjectured value for the maximal \textit{additive} lifetime \(\lmax^{+}(m)\) for different values of \(m\) when \(d=2\) and \(k=3\) using the \v Cech filtration.}
	\label{tab:1}
\end{table}
We now investigate this conjecture for additive lifetimes and \(m = 4\), \(m = 5\), and \(m = 6\). 
For optimization, we use simulated annealing as implemented in the \texttt{R}-package \texttt{optimization}, and we initialize the optimization scheme from \(m\) 
uniformly chosen points inside the unit ball. The results are displayed in Figure \ref{fig:optimal}, and these three final
(blue) configurations reach an additive lifetime of 0.273, 0.374, and 0.418, respectively. Thus, we see that as \(m\) increases, 
the simulated value of the maximal lifetime deviates further from the conjectured upper bound, which we attribute to the challenging nature of the optimization problem. 
With an increasing number of points, it becomes more difficult to find the global maximum.
However, the conjectured upper bounds in Table \ref{tab:1} are not violated, and the final configurations seem close to that of equidistant points on the unit circle, as we also 
conjectured above.

		\begin{figure}[h!]
			\centering
			\includegraphics[width=0.95\textwidth]{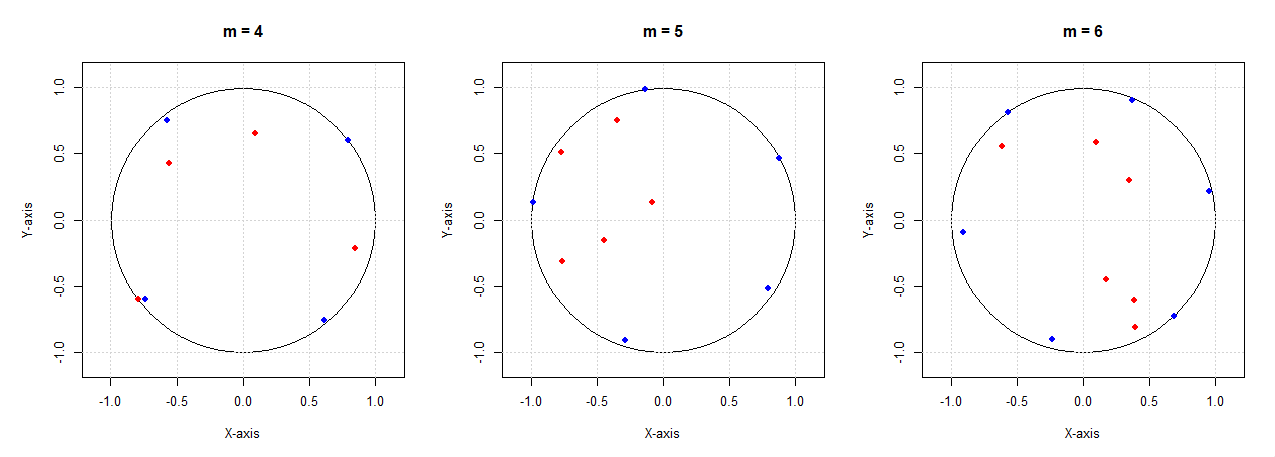}
			\caption{Illustration of the initial configuration (red points) and final configuration (blue points) for the
			 largest lifetime achieved through optimization for $m=4,5,6$ in Experiment \ref{exp:1}. The lifetimes of these 
			 configurations are 0.273, 0.374 and 0.418 and can be compared with conjectured values in Table \ref{tab:1}.}
			\label{fig:optimal}
		\end{figure}

	\end{experiment}
	
	\begin{experiment}[\textit{Distribution of the largest lifetime}]
		The goal of this simulation experiment is twofold: First, we aim to observe whether the deviation of the largest additive lifetime $\ell_n^{(1)}$ still approximately 
follows a Weibull distribution with a shape parameter \(q\), as computed in Example \ref{ex:6}, when \(n\) is of moderate size. 
Secondly, we use this convergence result to explore the shape parameter \(q\) for other values of \(m\), still considering \(d=2\) and \(k=3\).
Recall that the survival function of the Weibull distribution with rate $\lambda > 0$ and shape \(q > 0\) is \(S(t) = \exp(-(t/\lambda)^q)\). Based on this property,
we see that plotting the points
\vspace{0.75ex} \begin{equation}
	\label{eq:weibull_plot}
\big(\log(t), \log(-\log(S(t)))\big)
\vspace{0.75ex} \end{equation}
for different \(t\) should display a linear relationship with \(q\) as the slope and \(-q\log(\lambda)\) as the intercept.
To address the first objective, we choose a stationary Poisson point process on the unit cube \([0,1]^2\) with intensity \(n=20,50,100,200,500,1000\).
For each \(n\), ignoring the scaling \(u_{n,\alpha}\), we will estimate the survival function \(S\) of \(\lmax^+ - \ell_n^{(1)}\)
based on 5000 independent samples. Choosing the \v Cech filtration, \(m=3\), and \(r_n = n^{-0.74}\) yields six plots as displayed in Figure \ref{fig:mod_slope}.
On top of the plot of the pair of points in (\ref{eq:weibull_plot}), we fit a straight line using the \texttt{lm} routine in \texttt{R}. Observe that the slope of this line seems to increase to the theoretical value \(q=3\), as we would expect.
Additionally, at least for \(n=200,500,1000\), the linear relationship
also looks convincing. Thus, from these simulations, it might be reasonable to rely on the approximate Weibull distribution of the deviation
of the largest lifetime already for \(n \geq 200\).

		\begin{figure}[h!]
			\centering
			\includegraphics[width=0.95\textwidth]{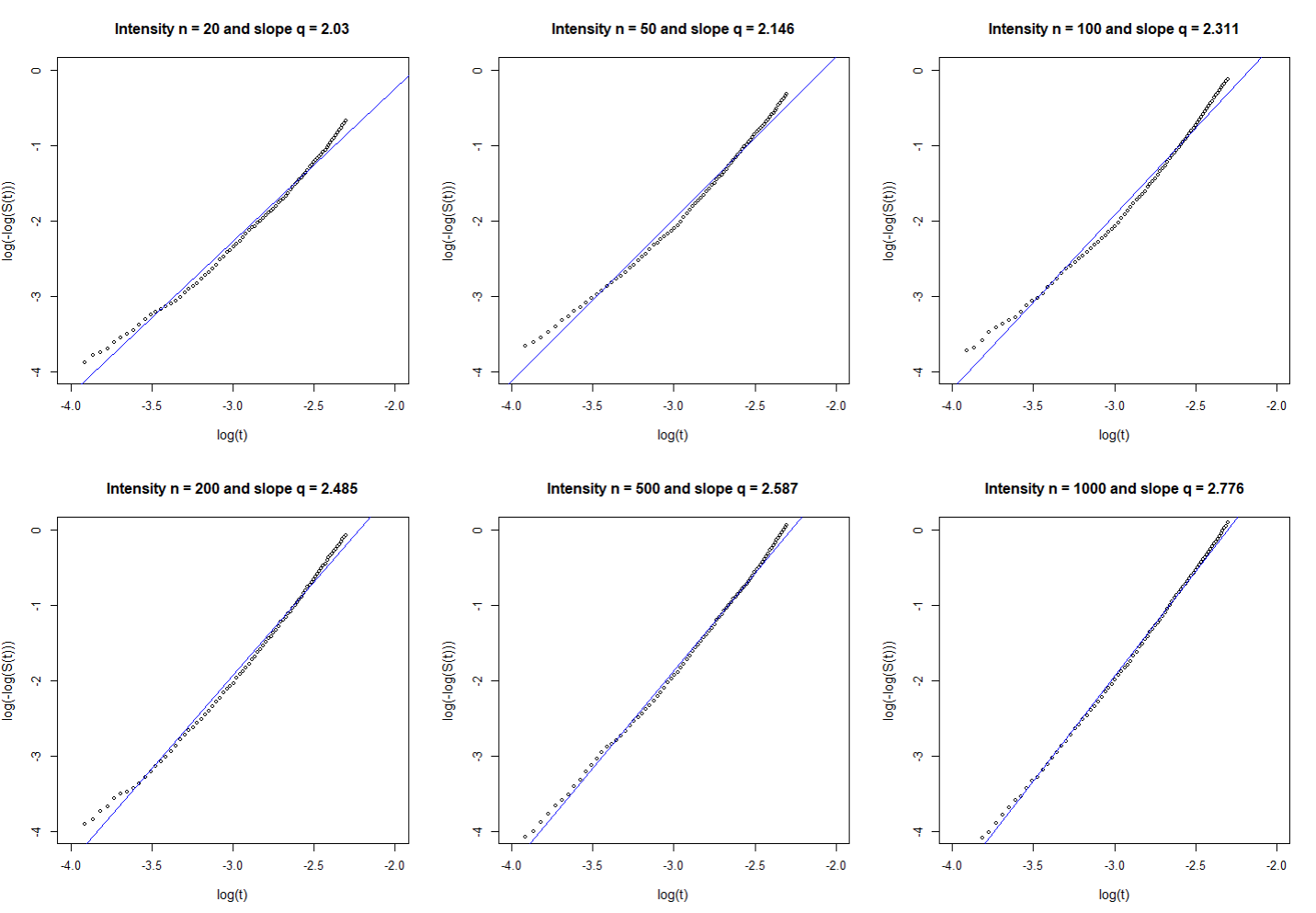}
			\caption{The estimated value of $\log(-\log(S(t)))$ plotted against $\log(t)$ when $m=3$ and a straight line is fitted (blue) with the slope recorded as $q$.
			This is done for the intensity $n=20,50,100,200,500,1000$.}
			\label{fig:mod_slope}
		\end{figure}

		Next, to address the second objective, we will repeat the method above for the values \(m=4\), \(m=5\), and \(m=6\). Following the same idea as described in
Remark \ref{rem:sparse_extension}, one might think that each point added would allow \(u^2\) more freedom, and thus we should expect (provided assumption \(G\) holds) that
\(q=5\), \(q=7\), and \(q=9\). From Example \ref{ex:2}, we choose the deathtime bounds \(r_n = n^{-0.66}\), \(r_n = n^{-0.62}\), and \(r_n = n^{-0.60}\), respectively. Here we will
rely on the conjectured values of the maximal additive lifetime as computed in Table \ref{tab:1} above.
The results are displayed in Figure \ref{fig:large_slope} below for intensities \(n=1000,2000,3500,5000\),
where, as before, we have plotted the points in (\ref{eq:weibull_plot}) and fitted a line. Contrary to what we just discussed,
it seems that \(q=4\), \(q=6\), and \(q=8\) based on the observed convergence from these simulations. Also, the linear relationship is the worst for \(m=4\),
which might indicate that something interesting happens here. If the conjecture in Experiment \ref{exp:1} holds, then these maximal lifetimes would be lifetimes
of four points that are close to a right-angled rhombus. Based on the results of the simulations, the above conjecture, i.e., that
\(q=2m-3\), does not seem violated for \(m>4\), however.

		\begin{figure}[h!]
			\centering
			\includegraphics[width=0.99\textwidth]{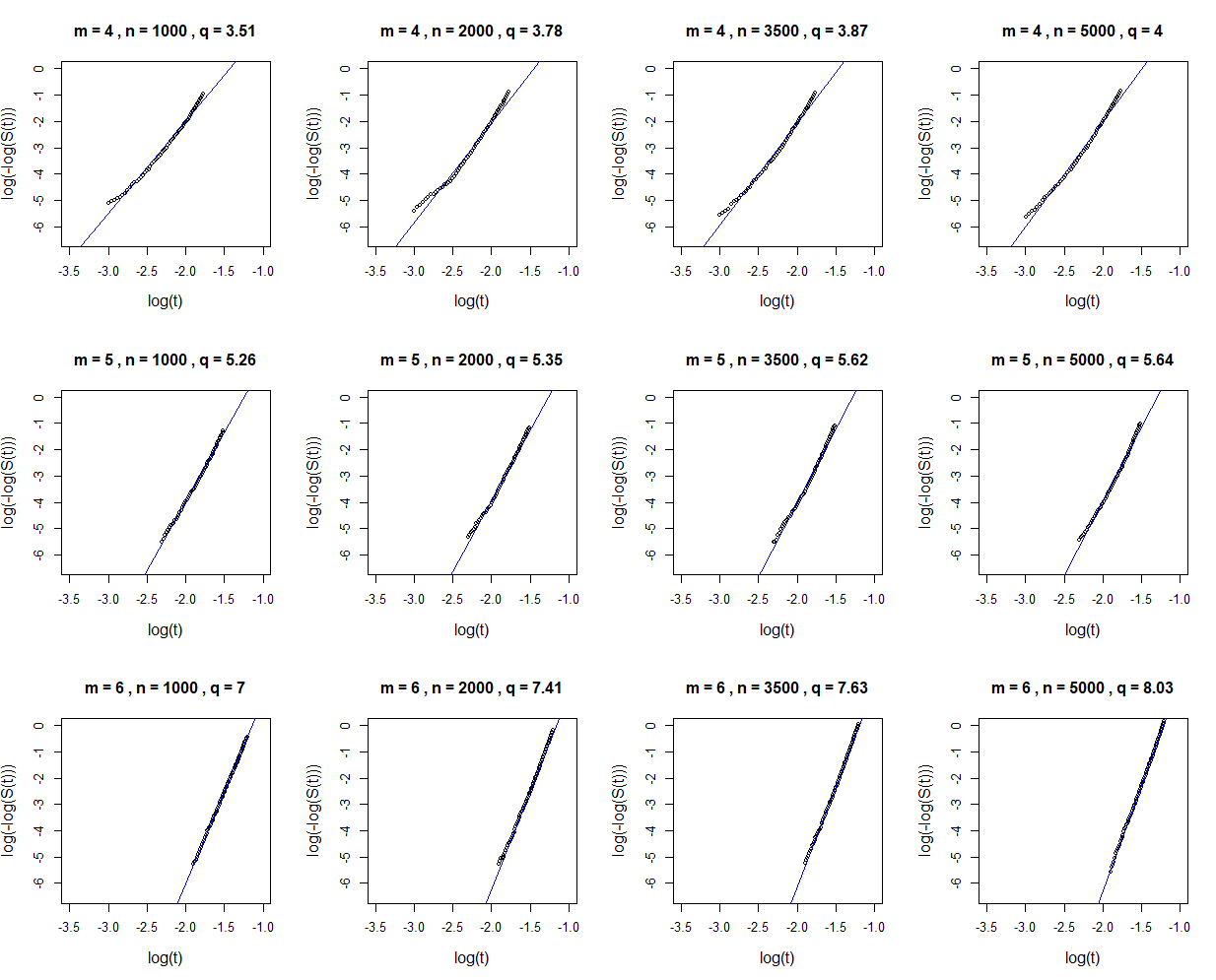}
			\caption{The results of checking the fit of the Weibull distribution for the deviation of the largest lifetime for $m=4,5,6$, where
			 $\log(-\log(S(t)))$ is plotted against $\log(t)$ and a straight line is fitted (blue) to determine a candidate for the value $q$.
			 Each value of $m$ is explored for the intensity $n=1000,2000,3500,5000$.}
			\label{fig:large_slope}
		\end{figure}
	\end{experiment}
\pagebreak
	\begin{experiment}[\textit{Deathtimes of largest and second largest lifetimes}]
		The goal of this simulation experiment is to investigate and illustrate the model-checking procedure discussed in Example \ref{ex:6},
involving the distribution of the deathtimes of the two cycles with the largest and second-largest lifetime. 
Let \(d=2\), \(m=k=3\), the filtration \v Cech, and the lifetimes additive.
Since \(m=3\), we will only consider \(1\)-cycles with deathtime less than \(r_n \leq n^{-0.67}\).
For a specific value of \(n\), we sample 100 instances of a stationary Poisson point process with intensity \(10n\) on the unit cube \([0,1]^2\),
compute the largest and second-largest lifetime and the corresponding deathtime among \(1\)-cycles,
and compute the empirical correlation from the 100 instances of the deathtimes of the two largest lifetimes. We repeat this for a total of 30 repetitions
and compute the average empirical correlation from these 30 empirical correlations. In Figure \ref{fig:deathtime_correlation} below, we plot the average correlation resulting for \(n=100,200,\ldots,1000\). 
Already for \(n \geq 200\), the correlation stays between \(-0.01\) and \(0.06\).

\begin{figure}[h!]
	\centering
	\includegraphics[width=0.6\textwidth]{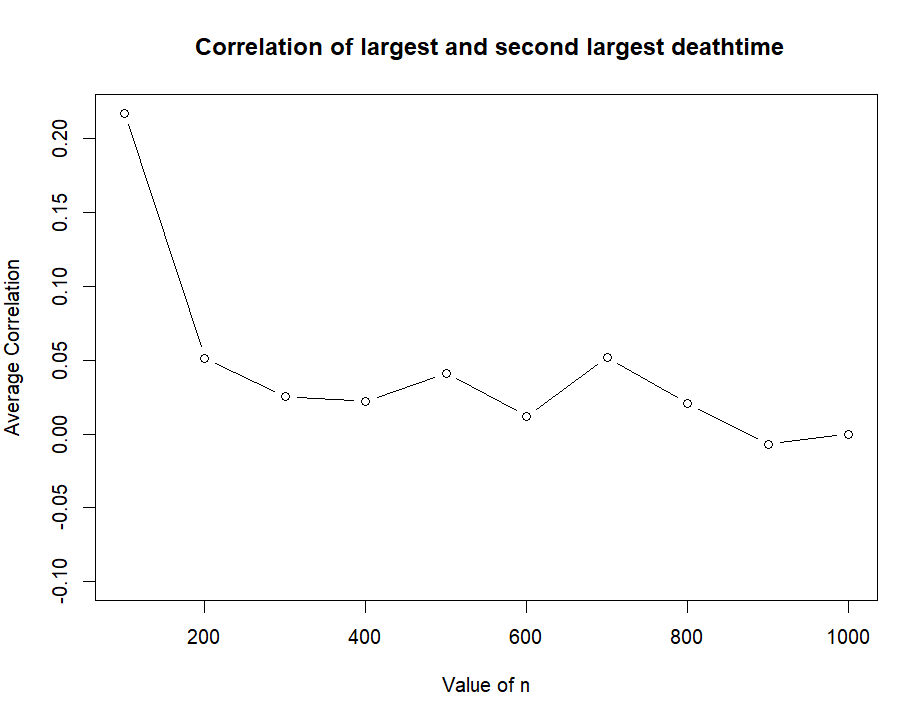}
	\caption{Average correlation between deathtimes of the two largest lifetimes, which is carried out with intensity $n=100,200,\ldots,1000$.}
	\label{fig:deathtime_correlation}
\end{figure}
\end{experiment}


\pagebreak
\vspace{1ex}

\begin{center}
\Large{\textsc{Acknowledgements}}
\end{center}
\vspace{1.5 ex}
The authors would like to thank Esben Trøllund Boye
for his contributions to an earlier version of the manuscript as well as thank Casper Urth Pedersen for his help with finishing touches.
Also, C. Hirsch was supported by a research grant (VIL69126) from VILLUM FONDEN
and M. Otto was supported by the NWO Gravitation project NETWORKS under grant agreement no. 024.002.003.


\bibliographystyle{abbrv}
\bibliography{lit}


\end{spacing}
\end{document}